\documentclass[11pt]{article}
\usepackage{amsmath} \usepackage{amssymb} \usepackage{latexsym} \usepackage{enumitem} \usepackage{mathrsfs} \usepackage{comment}
\usepackage{color}
\usepackage{cite}

\usepackage[colorlinks=true,urlcolor=blue,
citecolor=red,linkcolor=blue,linktocpage,pdfpagelabels,bookmarksnumbered,bookmarksopen]{hyperref}

\allowdisplaybreaks[4]

\newtheorem{assumption}{Assumption}

\def\qed{ \ \vrule width.2cm height.2cm depth0cm\smallskip}
\newenvironment{proof}{\noindent {\bf Proof.\/}}{$\qed$\vskip 0.1in}

\newcommand{\ba}{\begin{array}}
\newcommand{\ea}{\end{array}}
\newcommand{\be}{\begin{equation}}
\newcommand{\ee}{\end{equation}}
\newcommand{\bea}{\begin{eqnarray}}
\newcommand{\eea}{\end{eqnarray}}
\newcommand{\beaa}{\begin{eqnarray*}}
\newcommand{\eeaa}{\end{eqnarray*}}

\def\dbE{\mathbb{E}}
\def\dbF{\mathbb{F}}
\def\dbG{\mathbb{G}}

\def\dbL{\mathbb{L}}

\def\dbP{\mathbb{P}}
\def\dbR{\mathbb{R}}

\def\dbZ{\mathbb{Z}}

\def\rme{\mathrm{e}}
\def\rmd{\mathrm{d}}
\def\mcu{\mathcal{U} }
\def\mcv{\mathcal{V} }

\def\si{\sigma}

\def\O{\Omega}
\def\cB{{\cal B}}

\def\cF{{\cal F}}
\def\cG{{\cal G}}
\def\cH{{\cal H}}

\def\cK{{\cal K}}
\def\cL{{\cal L}}

\def\cN{{\cal N}}

\def\cP{{\cal P}}

\def\cW{{\cal W}}

\def\no{\noindent}

\def\q{\quad}

\def\pa{\partial}

\def\qed{ \hfill \vrule width.25cm height.25cm depth0cm\smallskip}

\newcommand{\basa}{\begin{assumption}}
\newcommand{\easa}{\end{assumption}}

\newcommand{\bas}{\begin{assum}}
\newcommand{\eas}{\end{assum}}

\def\liminf{\mathop{\underline{\rm lim}}}

\def\pa{\partial}

\def\1{{\bf 1}}

\def\:{\!:\!}
\def\reff{\eqref}
\def \proof{{\noindent \bf Proof.\quad}}

\def\Oin{\Omega_{\text{input}}}
\def\Oout{\Omega_{\text{output}}}
\def\Ocan{\Omega_{\text{canon}}}
\def\Oto{\Omega_{\text{total}}}
\def\Qin{Q_{\text{input}}}
\def\oin{\omega_{\text{input}}}

\def\gxX{\nabla X^{x,\xi}}
\def\gxY{\nabla Y^{x,\xi}}
\def\gxZ{\nabla Z^{x,\xi}}

\def\ggxX{\nabla^2 X^{x,\xi}}
\def\ggxY{\nabla^2 Y^{x,\xi}}
\def\ggxZ{\nabla^2 Z^{x,\xi}}

\def\ggXxx{\nabla_{\tilde{x}\mu}X^{x,\xi,\tilde{x}}}
\def\ggYxx{\nabla_{\tilde{x}\mu}Y^{x,\xi,\tilde{x}}}
\def\ggZxx{\nabla_{\tilde{x}\mu}Z^{x,\xi,\tilde{x}}}

\def\dX{\delta X^{\xi,\eta}}
\def\dY{\delta Y^{\xi,\eta}}
\def\dZ{\delta Z^{\xi,\eta}}
\def\tdX{\delta \tilde{X}^{\xi,\eta}}

\def\dxX{\delta X^{x,\xi,\eta}}
\def\dxY{\delta Y^{x,\xi,\eta}}
\def\dxZ{\delta Z^{x,\xi,\eta}}

\def\gXi{\nabla X^{\xi,x_i}}
\def\gYi{\nabla Y^{\xi,x_i}}
\def\gZi{\nabla Z^{\xi,x_i}}
\def\gXii{\nabla X^{\xi,x_i,\star}}
\def\gYii{\nabla Y^{\xi,x_i,\star}}
\def\gZii{\nabla Z^{\xi,x_i,\star}}
\def\tgXi{\nabla \tilde{X}^{\xi,x_i}}
\def\tgXii{\nabla \tilde{X}^{\xi,x_i,\star}}

\def\gXxi{\nabla_{\mu}X^{x,\xi,x_i}}
\def\gYxi{\nabla_{\mu}Y^{x,\xi,x_i}}
\def\gZxi{\nabla_{\mu}Z^{x,\xi,x_i}}

\def\gXx{\nabla X^{\xi,x}}
\def\gYx{\nabla Y^{\xi,x}}
\def\gZx{\nabla Z^{\xi,x}}
\def\tgXx{\nabla \tilde{X}^{\xi,x}}
\def\tgxX{\nabla \tilde{X}^{x,\xi}}

\def\gXxx{\nabla_{\mu}X^{x,\xi,\tilde{x}}}
\def\gYxx{\nabla_{\mu}Y^{x,\xi,\tilde{x}}}
\def\gZxx{\nabla_{\mu}Z^{x,\xi,\tilde{x}}}
\def\tgXtx{\nabla \tilde{X}^{\xi,\tilde{x}}}
\def\tgtxX{\nabla \tilde{X}^{\tilde{x},\xi}}

at 9pt

\definecolor{alp}{rgb}{0.0, 0.5, 0.0}

\newtheorem{thm}{Theorem}[section]
\newtheorem{lem}[thm]{Lemma}
\newtheorem{cor}[thm]{Corollary}

\newtheorem{rem}[thm]{Remark}

\newtheorem{defn}[thm]{Definition}
\newtheorem{assum}[thm]{Assumption}

\begin{document}

\title{\bf Regularity of the Value Function in Discounted Infinite-Time Mean Field Games  } 
\author{ Yongsheng Song$^*$ and Zeyu Yang\thanks{State Key Laboratory of Mathematical Sciences, Academy of Mathematics and Systems Science, Chinese Academy of Sciences, Beijing 100190, China, and School of Mathematical Sciences, University of Chinese Academy of Sciences, Beijing 100049, China. E-mails: yssong@amss.ac.cn (Y. Song), yangzeyu@amss.ac.cn (Z. Yang).}  }
\date{\today}

\maketitle

\begin{abstract}
In \cite{yang2025discounted}, we introduced discounted infinite-time
mean field games. Subsequently, in \cite{song2025infinite}, we studied
the connection between infinite-time mean field FBSDEs and elliptic
master equations. In this paper, we further investigate the regularity
of the representative player's value function \(V(x,\mu)\). We first
establish strong well-posedness and uniqueness in law for an extended
class of infinite-time FBSDEs.  In the weighted \(L_r^2\) framework,
we construct directional variations and obtain a bounded continuous
representation of the Lions-derivative of \(\partial_xV\).
Under additional smoothness and quantitative stability assumptions,
\(L_r^p\) estimates for scalar linear FBSDEs allow us to establish
the existence of higher-order derivatives of \(\partial_xV\).
Consequently, \(\partial_xV\) is a classical solution of the
elliptic master equation for the value function's derivative.
\end{abstract}

\no{\bf Keywords.}  discounted infinite-time 
  mean field games, infinite-time FBSDEs, Lions-derivative, classical solution

\vfill\eject


\section{Introduction}
\label{sect-Introduction}
\setcounter{equation}{0} 

The study of mean field games was initiated independently by Lasry-Lions
\cite{lasry2006jeux,lasry2006mean,lasry2007mean} and Huang-Malhamé-Caines \cite{huang2006large}, which is an analysis
of limit models for symmetric weakly interacting $(N+1)-$player differential games.
We refer the
reader to \cite{carmona2018probabilistic,cardaliaguet2019master,gangbo2022mean} for a comprehensive exposition on the subject.
Forward-backward stochastic differential equations (FBSDEs) serve as a powerful tool for the study of mean field games.
The investigation of general nonlinear BSDEs was pioneered by Pardoux and Peng \cite{pardoux1990adapted,pardoux2005backward} in the
early 1990s.
\cite{shi2020forward} studied the infinite-time FBSDEs  and established connections with quasilinear elliptic PDEs.
Recently,  \cite{ bayraktar2023solvability} extended this framework to 
the McKean-Vlasov FBSDEs.
In this paper, we study the spatial and measure regularity of the
value function's derivative in discounted infinite-time mean field games.
We establish strong well-posedness and uniqueness in law for a broader
class of infinite-time FBSDEs and use their variational equations to
construct the Lions-derivative and higher-order derivatives.

In the recent work \cite{yang2025discounted}, we proposed the discounted infinite-time  mean field games, 
which extends the traditional framework to infinite-time case.
Within this framework, we introduce the following two systems of infinite-time FBSDEs: 
\begin{equation}
  \label{eq: 11}
  \begin{cases}
    \rmd X_{t}^{\xi} = \partial_yH(X_{t}^{\xi},\mathcal{L} _{X_{t}^{\xi}},Y_{t}^{\xi}) \rmd t +\rmd B_{t}, \\
    \rmd Y_{t}^{\xi} = -\left[\partial_x {H} (X_{t}^{\xi},\mathcal{L} _{X_{t}^{\xi}}, Y_{t}^{\xi})-rY_{t}^{\xi}\right] \rmd t + Z_{t}^{\xi}\rmd B_{t}, \\
    X_0^{\xi}=\xi,
    \end{cases}
\end{equation}
\begin{equation}
  \label{eq: 12}
  \begin{cases}
    \rmd X_{t}^{x,\xi} = \partial_yH(X_{t}^{x,\xi},\mathcal{L} _{X_{t}^{\xi}},Y_{t}^{x,\xi}) \rmd t +\rmd B_{t}, \\
    \rmd Y_{t}^{x,\xi} = -\left[\partial_x {H} (X_{t}^{x,\xi},\mathcal{L} _{X_{t}^{\xi}}, Y_{t}^{x,\xi})-rY_{t}^{x,\xi}\right] \rmd t + Z_{t}^{x}\rmd B_{t}, \\
    X_0^x=x.
    \end{cases}
\end{equation}
Here $r>0$ is the discount factor,
\begin{equation}
  H(x,\mu,y)\triangleq \min_{a\in \dbR} \left[ b(x,\mu,a)\cdot y+f(x,\mu,a) \right], 
\end{equation}
and denote by $\hat{\alpha}(x,\mu,y)$  the unique minimizer.
The process $X^{\xi} $ represents the state process of the social equilibrium, 
while  $X^{x,\xi} $ denotes the state process of the representative player with initial state $x$.
The value function of the representative player is
\begin{equation}
  \label{eq: value}
V(x,\mu)\triangleq \mathbb{E}\bigg[\int_{0}^{+\infty}\rme^{-rt}f\big(X_{t}^{x,\xi},\mathcal{L}_{X_{t}^{\xi}},
\hat{\alpha}(X_{t}^{x,\xi},\mathcal{L}_{X_{t}^{\xi}}, Y_{t}^{x,\xi})        \big)\rmd t\bigg]
\end{equation}
where $\cL_\xi=\mu$. Under the assumptions of
\cite{song2025infinite}, $V(x,\mu)$ is a viscosity solution of the
elliptic master equation
\begin{equation}
\label{eq: master}
\begin{split}
  r U(x,\mu)=&H(x,\mu,\partial_xU(x,\mu))+\frac{1}{2}\partial_{xx}U(x,\mu)\\&+\tilde{\mathbb{E} }
  \left[\frac{1}{2}\partial_{\tilde{x}}\partial_\mu U(x,\mu,\tilde{\xi })+\partial_\mu U(x,\mu,\tilde{\xi })
  \partial_y H(\tilde{\xi },\mu,\partial_xU(\tilde{\xi},\mu))\right].
\end{split}
\end{equation}
Here $\partial_x,\partial_{xx}$ are standard spatial derivatives, $\partial_{\mu},\partial_{\tilde{x}\mu}$ are $\cW_2$-Wasserstein derivatives, $\tilde \xi$ is a random variable with law $\mu$ and $\tilde \dbE$ is the expectation with respect to its law. 
The spatial derivative is identified in \cite{song2025infinite} by
\begin{equation}
  Y_0^{x,\xi}=\pa_x V(x,\mu).
\end{equation}
Set $\mcv(x,\mu)\triangleq Y_0^{x,\xi}=\partial_xV(x,\mu)$.
It was proved in \cite{yang2025discounted} that $\mcv(x,\mu)$ is a
viscosity solution of
\begin{equation}
  \label{eq: intro-pa}
  \begin{split}
  r \mcu(x,\mu)=&\partial_xH(x,\mu,\mcu(x,\mu))+ \partial_yH(x,\mu,\mcu(x,\mu))\cdot \partial_x\mcu(x,\mu) +
  \frac{1}{2}\partial_{xx}\mcu(x,\mu)\\&+\tilde{\mathbb{E} }
  \left[\frac{1}{2}\partial_{\tilde{x}}\partial_\mu \mcu(x,\mu,\tilde{\xi })+\partial_\mu \mcu(x,\mu,\tilde{\xi })
  \partial_y H(\tilde{\xi },\mu,\mcu(\tilde{\xi},\mu))\right].
\end{split}
\end{equation}

In this paper, we prove that $\mcv(x,\mu)$ is Lions-differentiable with respect to the measure $\mu$
and derive the explicit form of its Lions-derivative.
In the finite-time mean field games\cite{carmona2018probabilistic}, the Lions-differentiability 
of the value function and its derivative  follow directly from the 
analysis of finite-time FBSDEs.  Explicit expressions for these Lions-derivatives are 
provided by Mou et al. \cite{mou2024wellposedness} and Gangbo et al. \cite{gangbo2022mean} using FBSDEs. However, applying this conventional methodology 
to our setting of infinite-time FBSDEs introduces significant challenges and fundamental differences.
Specifically, the directional derivative of FBSDEs (\ref{eq: 11}) and (\ref{eq: 12}) with respect to $\xi$ 
along the direction $\eta$
can be represented by the following FBSDEs:
\begin{equation}
  \begin{cases}
    \begin{aligned}
      \rmd \dX_t=&\left\{\dX_t \partial_{xy}H(X_{t}^{\xi},\mathcal{L} _{X_{t}^{\xi}},Y_{t}^{\xi}) 
                        +\dY_t \partial_{yy}H(X_{t}^{\xi},\mathcal{L} _{X_{t}^{\xi}},Y_{t}^{\xi}) \right.\\
          &\left.+\tilde{\dbE}_{\cF_t} \left[
            \partial_{y\mu}H(X_{t}^{\xi},\mathcal{L} _{X_{t}^{\xi}},Y_{t}^{\xi},\tilde{X}_{t}^{\xi}) \tdX_t
          \right]\right\}\rmd t,
    \end{aligned}\\
    \begin{aligned}
      \rmd \dY_t=-&\left\{\dX_t \partial_{xx}H(X_{t}^{\xi},\mathcal{L} _{X_{t}^{\xi}},Y_{t}^{\xi}) 
                        +\dY_t \partial_{xy}H(X_{t}^{\xi},\mathcal{L} _{X_{t}^{\xi}},Y_{t}^{\xi})-r\dY_t
                         \right.\\
          &\left.+\tilde{\dbE}_{\cF_t} \left[
            \partial_{x\mu}H(X_{t}^{\xi},\mathcal{L} _{X_{t}^{\xi}},Y_{t}^{\xi},\tilde{X}_{t}^{\xi}) \tdX_t
          \right]\right\}\rmd t +  \dZ_t\rmd B_t ,
    \end{aligned}\\
    \dX_0=\eta;
  \end{cases}
\end{equation}

\begin{equation}
  \begin{cases}
    \begin{aligned}
      \rmd \dxX_t=&\left\{\dxX_t \partial_{xy}H(X_{t}^{x,\xi},\mathcal{L} _{X_{t}^{\xi}},Y_{t}^{x,\xi}) 
                        +\dxY_t \partial_{yy}H(X_{t}^{x,\xi},\mathcal{L} _{X_{t}^{\xi}},Y_{t}^{x,\xi}) \right.\\
          &\left.+\tilde{\dbE}_{\cF_t} \left[
            \partial_{y\mu}H(X_{t}^{x,\xi},\mathcal{L} _{X_{t}^{\xi}},Y_{t}^{x,\xi},\tilde{X}_{t}^{\xi}) \tdX_t
          \right]\right\}\rmd t,
    \end{aligned}\\
    \begin{aligned}
      \rmd \dxY_t=-&\left\{\dxX_t \partial_{xx}H(X_{t}^{x,\xi},\mathcal{L} _{X_{t}^{\xi}},Y_{t}^{x,\xi}) 
                        +\dxY_t \partial_{xy}H(X_{t}^{x,\xi},\mathcal{L} _{X_{t}^{\xi}},Y_{t}^{x,\xi})
                         \right.\\
          &\left.-r\dxY_t+\tilde{\dbE}_{\cF_t} \left[
            \partial_{x\mu}H(X_{t}^{x,\xi},\mathcal{L} _{X_{t}^{\xi}},Y_{t}^{x,\xi},\tilde{X}_{t}^{\xi}) \tdX_t
          \right]\right\}\rmd t\\& +  \dxZ_t\rmd B_t ,
    \end{aligned}\\
    \dxX_0=0.
  \end{cases}
\end{equation}
In the conventional framework, the Lions-differentiability 
of $\mcv(x,\mu)$ with respect to $\mu$ can be established by showing that the continuity of 
$\dxY_0$ in $\xi$ is uniform with respect 
to $\eta$, which then allows the classical Lions-derivative analysis to be applied. However, this approach is 
not feasible in our setting.
To overcome this difficulty, we adapt the construction of the Lions-derivative from \cite{gangbo2022mean,mou2024wellposedness} 
to the infinite-time FBSDE framework by introducing the directional derivative $\pa_\mu \mcv(x,\mu,\tilde{x})$ 
of $\mcv(x,\mu)$, 
which satisfies: 
\begin{equation}
  \lim_{\delta\to 0}\frac{1}{\delta}\left[ 
\mcv(x,\cL_{\xi+\delta\eta})-\mcv(x,\cL_{\xi})
  \right] =\dbE[\pa_{\mu}\mcv(x,\cL_{\xi},\xi)\cdot\eta].
\end{equation}
Finally, by proving that $\pa_\mu \mcv(x,\mu,\tilde{x})$ is bounded and continuous, 
we demonstrate that $\pa_\mu \mcv(x,\mu,\tilde{x})$ is indeed the Lions-derivative of $\mcv(x,\mu)$.

To establish the higher-order regularity of $\mcv(x,\mu)$, we
differentiate the spatial variation and the representation of its
Lions-derivative. The equation for the second spatial variation has
quadratic terms in the first variation and therefore requires an
$L_r^4$ estimate. Using the Riccati construction of Guatteri and
Tessitore \cite{GuatteriTessitore2008}, we establish weighted
$L_r^p$ estimates for every $p\geq2$ for the relevant scalar linear equations.
These estimates allow us to construct all the required higher-order
derivatives of $\mcv(x,\mu)=\partial_xV(x,\mu)$ and prove that it is a classical
solution of the elliptic master equation \eqref{eq: intro-pa} for
the value function's derivative.

This paper is organized as follows: in section \ref{sec: preliminaries}, we present the preliminaries of problems in this paper;
in section \ref{sec: solution}, we prove the strong existence and uniqueness, 
as well as the uniqueness in distribution, for a broader class of infinite-time FBSDEs, in preparation for the subsequent analysis;
 in section \ref{sec: diff} we present an explicit construction of the directional derivative of $\mcv(x,\mu)$
 with respect to $\mu$ and 
 ultimately establish the Lions-differentiability of $\mcv(x,\mu)$;
 in section \ref{sec: classical}, we prove the two remaining higher-order
 derivatives and the classical solvability of the elliptic master equation
 for the value function's derivative \(\mcv=\partial_xV\).

\section{Preliminaries}
\label{sec: preliminaries}
\setcounter{equation}{0}

We will use the filtered probability space $(\Omega, \mathcal{F} ,\mathbb{P},\mathbb{F} )$
endowed
with a Brownian motion $B$.
Its filtration $\mathbb{F} \triangleq (\mathcal{F} _t)_{t\ge 0}$ is
augmented
by all $\mathbb{P}$-null sets and a sufficiently rich sub-$\sigma$-algebra $\mathcal{F}_0$ independent
of $B$, such that it
can support any measure on $ \mathbb{R} $ with finite second moment.

Let $( \O', \mathcal{F}' , \dbP',\dbF')$ be a copy of the filtered probability space $(\Omega, \mathcal{F} ,\mathbb{P},\mathbb{F})$
 with corresponding Brownian motion $B'$, define the larger filtered probability space by
\begin{equation}
\tilde \O \triangleq \O\times \O' ,\q \tilde{ \mathcal{F}} \triangleq \mathcal{F}\otimes \mathcal{F}'\q \tilde\dbF = \{\tilde \cF_t\}_{t\ge 0} \triangleq \{\cF_t \otimes  \cF'_t\}_{t\ge 0},\q \tilde \dbP \triangleq \dbP\otimes \dbP',\q \tilde \dbE\triangleq \dbE^{\tilde \dbP}.
\end{equation}
Throughout the paper we will use the probability space $(\Omega, \mathcal{F} ,\mathbb{P},\mathbb{F})$. However, when we deal with the distribution-dependent master equation, independent copies of random variables or processes are needed. Then we will tacitly use their extensions to the larger space $(\tilde \O,\tilde{\mathcal{F}},  \tilde \dbP,\tilde \dbF)$.

Let $\cP \triangleq\cP(\dbR)$ be the set of all probability measures on $\mathbb R$ and let $\cP_p(p\ge1)$ denote the set of $\mu\in \cP$ with finite $p$-th moment.
For any sub-$\si$-field $\cG\subset \cF$ and $\mu\in \cP_p$, we define $\dbL^p(\cG)$ to be the set of $\dbR$-valued, $\cG$-measurable, and $p$-integrable random variables $\xi$ , and $\dbL^p(\cG;\mu)$ to be the set of $\xi\in \dbL^p(\cG)$ such that the law $\cL_\xi=\mu$ .
For any $\mu,\nu\in \cP_p$, we define the $\cW_p$--Wasserstein distance between them as follows: 
\beaa
\cW_p(\mu, \nu) \triangleq \inf\Big\{\big(\dbE[|\xi-\eta|^p]\big)^{1/p}: \xi\in \dbL^p(\cF; \mu),\ \eta\in \dbL^p(\cF; \nu)\Big\}.
\eeaa

We introduce the Wasserstein space and differential calculus on Wasserstein space. 
For a $\cW_2$-continuous functions $U: \cP_2 \to \dbR$, its $\cW_2$-Wasserstein derivatives\cite{carmona2018probabilistic}(also called Lions-derivative), takes the form $\pa_\mu U: (\mu,\tilde x)\in \cP_2\times \dbR\to \dbR$ and satisfies: 
\bea
\label{pamu}
U(\cL_{\xi +  \eta}) - U(\mu) = \dbE\big[\langle \pa_\mu U(\mu, \xi), \eta \rangle \big] + o(\|\eta\|_2), \ \forall\ \xi\in\mathbb L^2(\mathcal{F};\mu),\eta\in\mathbb L^2(\mathcal{F}).
\eea
Let $C^0(\cP_2)$ denote the set of $\cW_2$-continuous functions $U:\cP_2\to\dbR$.
For $C^1(\cP_2)$, we mean the space of functions $U\in C^0(\cP_2)$ such that $\pa_\mu U$ exists and is continuous on $ \cP_2\times \dbR$, which is uniquely determined by \reff{pamu}.
Let $C^{2,1}(\dbR\times\cP_2)$ denote the set of continuous functions $U:\dbR\times\cP_2\to\dbR$ such that $\pa_xU,\pa_{xx}U$ exist and are jointly continuous on $\dbR\times \cP_2$, $\pa_\mu U,\pa_{\tilde x\mu}U$ exist and are continuous on $\dbR\times\cP_2\times\dbR$.

Due to our interest in infinite-time mean field FBSDEs, for any $K\in \mathbb{R} $, we denote by $L^2_K(t_0, \infty, \mathbb{R} )$
 the Hilbert space of all $\mathbb{R} -$valued adapted stochastic processes ($v_t$) starting from $t_0$ such that
\begin{equation}\mathbb{E}\left[\int_{t_0}^\infty e^{-Kt}|v_t|^2dt\right]<+\infty.\end{equation}
To simplify, we set $L^2_K \triangleq L^2_K(0, \infty,\mathbb{R} )$.

\section{Solutions to infinite-time FBSDEs}
\label{sec: solution}
\setcounter{equation}{0}

\subsection{Strong solutions to infinite-time FBSDEs}

For the needs of subsequent problems, we aim to establish a more general theorem on the existence and 
uniqueness of solutions for infinite-time FBSDEs.
Consider the following form of infinite-time FBSDEs:
\begin{equation}
  \label{eq: fbsde}
  \begin{cases}
    \rmd X_t=G(t,\omega,X_t,Y_t,\cL_{ (X_t,A_t)})\rmd t+\sigma \rmd B_t,\\
        \rmd Y_t=-F(t, \omega,X_t,Y_t,\mathcal{L}_{(X_t,A_t)})\rmd t+Z_t\rmd B_t, \\
    X_0=\xi,
  \end{cases}
\end{equation}
where $G,F:\dbR_+\times \Omega\times \dbR\times \dbR\times\cP_2(\dbR^{2})\rightarrow \dbR$ 
are two progressively measurable functions, $A_t$ is a given adapted square integrable process,
$\sigma\in\dbR$ is a constant
and $\xi$ is an $\cF_0-$measurable square integrable random variable.
For any $(v_t)\in L_K^2$, we define the exponentially weighted $L^2$ norm
\begin{equation}
  \Vert v\Vert _K^2\triangleq \dbE\left[\int_0^\infty \rme^{-Kt}|v_t|^2\rmd t \right].
\end{equation}
For simplicity, we only solve (\ref{eq: fbsde}) for one dimensional $(X_t, Y_t, Z_t)$
and starting time $t_0=0$
, but our result can
be easily generalized to multidimensional case and arbitrary starting time $t_0>0$.
The key idea of our proof follows \cite{bayraktar2023solvability,shi2020forward}.

\begin{assum}
  \label{assum: fbsde}
Assume that for some constant $K$, the functions $F$ and $G$ satisfy:

\noindent (i) For any $L_K^2$ processes $(X_t,Y_t)$, $G(t,\omega,X_t,Y_t,\cL_{(X_t,A_t)})$ and 
$F(t,\omega,X_t,Y_t,\cL_{(X_t,A_t)})$ belong to $L_K^2$.

\noindent  (ii) There exists a positive constant $\ell$ such that for any $x,x',y,y'\in \mathbb{R},$
 and any square
integrable random variables $X,X',A$
  \begin{equation}
    \begin{aligned}
      &|G(t,\omega,x,y,\cL_{ (X,A)})  -G(t,\omega,x^{\prime},y^{\prime},\cL_{ (X',A)})|+
      |F(t,\omega,x,y,\cL_{ (X,A)})-F(t,\omega,x^{\prime},y^{\prime},\cL_{ (X',A)})| \\
        \leq& \ell(|x-x^{\prime}|+|y-y^{\prime}|+ \dbE[|X-X'|^2]^{\frac{1}{2}}     )
        .\quad\mathrm{a.s.}
      \end{aligned}
  \end{equation}

\noindent (iii) There exists a constant $\kappa>K/2$, such that for any $t\geq0$ and any square
integrable random variables $X,X',Y,Y',A$,
\begin{equation}
  \label{eq: condition}
  \begin{aligned}
    &\mathbb{E}\left[-K\hat{X}\hat{Y}-\hat{X}(F(t, \omega,U)-F(t,\omega,U^{\prime}))+\hat{Y}(G(t,\omega,U)-G(t,\omega,U^{\prime}))\right] \\
      &\leq-\kappa\mathbb{E}\left[\hat{X}^{2}+\hat{Y}^{2}\right],
    \end{aligned}
\end{equation}

where $ \hat{X}\triangleq X-X^{\prime},\hat{Y}\triangleq Y-Y^{\prime}$ and $U\triangleq (X,Y,\mathcal{L}_{(X,A)}),U'\triangleq(X',Y',\mathcal{L}_{(X',A)}) .$
\end{assum}

\begin{thm}
\label{thm: fbsde}
Under Assumption \ref{assum: fbsde}, for each $\cF_0$-measurable square integrable random 
variable $\xi$ , (\ref{eq: fbsde}) has a unique solution $(X_t,Y_t,Z_t)$ in $L_K^2$.
\end{thm}

\proof
\noindent First, we prove the uniqueness.
Suppose that \((X_t,Y_t,Z_t)\) and \((X_t',Y_t',Z_t')\) are two
solutions of \eqref{eq: fbsde} in \(L_K^2\). Set
 \begin{equation}
  \hat{X}\triangleq X-X'\quad \hat{Y}\triangleq Y-Y'\quad \hat{Z}\triangleq Z-Z'.
 \end{equation}
Since $\dbE\int_0^\infty\rme^{-Kt}|\hat X_t\hat Y_t|\rmd t<\infty$, there
is a sequence $T_i\to\infty$ along which the integrand tends to zero.  We
choose it so that
\begin{equation}
  \dbE\left[\rme^{-KT_i}\hat{X}_{T_i}\hat{Y}_{T_i}\right]\rightarrow 0.
\end{equation}
Applying It\^{o}'s formula to $\rme^{-Kt}\hat{X}_{t}\hat{Y}_{t}$, we get that
\begin{equation}\begin{aligned}
 & \mathbb{E}\left[\rme^{-KT_{i}}\hat{X}_{T_{i}}\hat{Y}_{T_{i}}\right] \\
  =&\mathbb{E}\left[\int_{0}^{T_{i}}\rme^{-Kt}
  \left(-K\hat{X}_{t}\hat{Y}_{t}-\hat{X}_{t}(F(t,\omega,X_t,Y_t,\cL_{(X_t,A_t)})-F(t,\omega,X_t',Y_t',\cL_{(X_t',A_t)})) \right.\right.\\
 &\left.\left. +\hat{Y}_{t}(G(t,\omega,X_t,Y_t,\cL_{(X_t,A_t)})-G(t,\omega,X_t',Y_t',\cL_{(X_t',A_t)}))\right)dt\right] \\
  \leq& -\kappa\mathbb{E}\left[\int_0^{T_i}e^{-Kt}\left(\hat{X}_t^2+\hat{Y}_t^2\right)\rmd t\right].
\end{aligned}\end{equation}
Letting $T_i\to \infty$, we get that 
\begin{equation}
  \|\hat{X}\|_K^2=\|\hat{Y}\|_K^2=0,
\end{equation}
and hence we complete the proof of the uniqueness.

Next, we prove the existence of solutions, for this purpose, we use the continuity method.
We study the following family of infinite-time FBSDEs parametrized
by $\lambda\in [0,1]$,
\begin{equation}
  \label{eq: lambda}
  \begin{cases}
    \begin{aligned}
  \rmd X_{t}^{\lambda}=&\left[ \lambda G(t,\omega, X_{t}^{\lambda},Y_{t}^{\lambda},\mathcal{L}_{(X_{t}^{\lambda},A_t)}) \right.
       \left.  -\kappa(1-\lambda)Y_{t}^{\lambda}+\phi_t(\omega) \right]\rmd t+\sigma \rmd B_{t}, 
\end{aligned}\\
\begin{aligned}
  \rmd Y_{t}^{\lambda}=&-\left[\lambda F(t,\omega,X_{t}^{\lambda},Y_{t}^{\lambda},\mathcal{L}_{(X_{t}^{\lambda},A_t)})\right.
   \left. +\kappa(1-\lambda)X_{t}^{\lambda}+\psi_t(\omega)\right]\rmd t+Z_{t}^{\lambda}\rmd B_{t}, 
\end{aligned}\\
\begin{aligned}
  & X_{0}^{\lambda}=\xi,\quad (X_{t}^{\lambda},Y_{t}^{\lambda},Z_{t}^{\lambda})\in L_{K}^{2},
\end{aligned}
  \end{cases}
\end{equation}
where $\phi,\psi$ are two arbitrary processes in $L_K^2$.
Note that when $\lambda = 1, \phi =\psi =0$, (\ref{eq: lambda})
becomes (\ref{eq: fbsde}), and when $\lambda = 0$, (\ref{eq: lambda}) becomes
\begin{equation}
  \label{eq: lis0}
  \begin{cases}
 & \rmd X_t^0=(-\kappa Y_t^0+\phi_t(\omega))\rmd t+\sigma \rmd B_t, \\
 & \rmd Y_t^0=-(\kappa X_t^0+\psi_t(\omega))\rmd t+Z_t^0 \rmd B_t, \\
 & X_0^0=\xi.  
\end{cases}\end{equation}
It has been proved in \cite[Lemma~2.1]{bayraktar2023solvability} that (\ref{eq: lis0}) has unique solution $(X^0,Y^0,Z^0)\in L_K^2$.

Now, suppose for some $\lambda_0\in[0,1)$, we have that 
for any $\cF_0$-measurable square integrable
random variable $\xi$ and $\phi,\psi \in L^2_K$,
 (\ref{eq: lambda}) has a unique solution $(X^{\lambda_0} , Y^{\lambda_0} , Z^{\lambda_0} )$
in $L^2_K$. We  try to find a constant $\delta_0$, such that for any $\delta\in[0,\delta_0]$,
the FBSDE(\ref{eq: lambda}) has a unique solution $(X^{\lambda_0+\delta} , Y^{\lambda_0+\delta} , Z^{\lambda_0+\delta} )$
in $L^2_K$ for any given $\xi,\phi,\psi$.
To do this, we consider the following FBSDE:
\begin{equation}
\label{eq: contra}
\begin{cases}
\begin{aligned}
\rmd X_t ={} & \bigl[\lambda_0 G(t,\omega,X_t,Y_t,\mathcal{L}_{(X_{t},A_t)}) - \kappa(1-\lambda_0)Y_t \\
         & + \delta \bigl(G(t,\omega,x_t,y_t,\mathcal{L}_{(x_{t},A_t)}) + \kappa y_t \bigr) + \phi_t(\omega) \bigr]\rmd t + \sigma \rmd B_t, \\
\rmd Y_t ={} & -\bigl[\lambda_0 F(t,\omega,X_t,Y_t,\mathcal{L}_{(X_{t},A_t)}) + \kappa(1-\lambda_0)X_t \\
         & + \delta \bigl(F(t,\omega,x_t,y_t,\mathcal{L}_{(x_{t},A_t)}) - \kappa x_t \bigr) + \psi_t(\omega) \bigr]\rmd t + Z_t \rmd B_t, \\
X_0 ={} & \xi.
\end{aligned}
\end{cases}
\end{equation}
For any pair $(x_t,y_t)$ in $L_K^2$, we have 
$G(t,\omega,x_t,y_t,\mathcal{L}_{(x_{t},A_t)})$ and $F(t,\omega,x_t,y_t,\mathcal{L}_{(x_{t},A_t)}) $
belong to $L_K^2$ according to our hypothesis, then the FBSDE (\ref{eq: contra}) 
admits  a unique solution $(X, Y , Z)$ in $L_K^2$.
We can define a map $\Phi$ through 
\begin{equation}\Phi:(x,y)\mapsto(X,Y).\end{equation}
We then prove that the map $\Phi$ is a contraction on $L_K^2$.

Take another pair $(x_t',y_t')$ in $L_K^2$ and its image $(X_t',Y_t')$. Denote
\begin{equation}
  \begin{aligned}
    &U_t=(X_t,Y_t,\cL_{(X_t,A_t)}),\quad u_t=(x_t,y_t,\cL_{(x_t,A_t)}),\\
    &U_t'=(X_t',Y_t',\cL_{(X_t',A_t)}),\quad u_t=(x_t',y_t',\cL_{(x_t',A_t)}),\\
    &\hat{X}_t=X_t-X_t',\quad \hat{Y}_t=Y_t-Y_t',\\
    &\hat{x}_t=x_t-x_t',\quad \hat{y}_t=y_t-y_t'.
  \end{aligned}
\end{equation}
Applying It\^{o}'s formula to $\rme^{-Kt}\hat{X}_t\hat{Y}_t$, we get that 
\begin{equation}
\begin{aligned}
& \dbE \left[\mathrm{e}^{-KT}\hat{X}_{T}\hat{Y}_{T}\right] \\
& = \lambda_{0}\mathbb{E}\biggl[\int_{0}^{T}\mathrm{e}^{-Kt}\biggl(
      -K\hat{X}_{t}\hat{Y}_{t} - \hat{X}_{t}\bigl(F(t,\omega,U_{t}) - F(t,\omega,U_{t}^{\prime})\bigr) 
       \quad + \hat{Y}_{t}\bigl(G(t,\omega,U_{t}) - G(t,\omega,U_{t}^{\prime})\bigr)
    \biggr)\mathrm{d}t\biggr] \\
& - \kappa(1-\lambda_0)\mathbb{E}\left[\int_0^{T}\mathrm{e}^{-Kt}\left(
      \hat{X}_t^2 + \hat{Y}_t^2
    \right)\mathrm{d}t\right] \\
& - (K-\lambda_0K)\mathbb{E}\left[\int_0^{T}\mathrm{e}^{-Kt}
      \hat{X}_t\hat{Y}_t
    \mathrm{d}t\right] \\
& + \kappa\delta\mathbb{E}\left[\int_0^{T}\mathrm{e}^{-Kt}\left(
      \hat{X}_t\hat{x}_t + \hat{Y}_t\hat{y}_t
    \right)\mathrm{d}t\right] \\
& + \delta\mathbb{E}\biggl[\int_{0}^{T}\mathrm{e}^{-Kt}\biggl(
      -\hat{X}_{t}\bigl(F(t,\omega,u_{t}) - F(t,\omega,u_{t}^{\prime})\bigr) 
       \quad + \hat{Y}_{t}\bigl(G(t,\omega,u_{t}) - G(t,\omega,u_{t}^{\prime})\bigr)
    \biggr)\mathrm{d}t\biggr].
\end{aligned}
\end{equation}
It holds that
\begin{equation}
\begin{aligned}
&\mathbb{E}\biggl[ -K\hat{X}_t\hat{Y}_t 
 - \hat{X}_t\bigl(F(t,\omega,U_t) - F(t,\omega,U_t^{\prime})\bigr) 
 + \hat{Y}_t\bigl(G(t,\omega,U_t) - G(t,\omega,U_t^{\prime})\bigr) \biggr] \\
&\leq -\kappa\mathbb{E}\left[\hat{X}_t^2 + \hat{Y}_t^2\right]
\end{aligned}
\end{equation}
and
\begin{equation}
\begin{aligned}
  &\dbE\left[-\hat{X}_{t}\bigl(F(t,\omega,u_{t}) - F(t,\omega,u_{t}^{\prime})\bigr) + \hat{Y}_{t}\bigl(G(t,\omega,u_{t}) - G(t,\omega,u_{t}^{\prime})\bigr)\right]\\
  \le&\dbE\left[|\hat{X}_{t}|\cdot|F(t,\omega,u_{t}) - F(t,\omega,u_{t}^{\prime})|+ |\hat{Y}_{t}|\cdot|G(t,\omega,u_{t}) - G(t,\omega,u_{t}^{\prime})|\right]\\
 \le& \frac{3\ell}{2}\dbE\left[|\hat{X}_{t}|^2+|\hat{Y}_{t}|^2\right]
 +2\ell\dbE\left[|\hat{x}_{t}|^2+|\hat{y}_{t}|^2\right].
\end{aligned}
\end{equation}
So we have
\begin{equation}
\begin{aligned}
\mathbb{E}\left[\mathrm{e}^{-KT}\hat{X}_{T}\hat{Y}_{T}\right] 
\leq & -\left(\kappa - \frac{K}{2} - \frac{\kappa\delta + 3l\delta}{2}\right)\mathbb{E}\left[\int_{0}^{T}\mathrm{e}^{-Kt}\left(\hat{X}_{t}^{2} + \hat{Y}_{t}^{2}\right)\mathrm{d}t\right] \\
& + \frac{\kappa\delta + 4l\delta}{2}\mathbb{E}\left[\int_0^{T}\mathrm{e}^{-Kt}\left(\hat{x}_t^2 + \hat{y}_t^2\right)\mathrm{d}t\right].
\end{aligned}
\end{equation}
We take 
\begin{equation}
  \delta_0=\frac{2\kappa-K}{3\kappa+11\ell}
\end{equation}
and choose a sequence of $T_i\to \infty$ such that 
\begin{equation}
  \dbE\left[\rme^{-KT_i}\hat{X}_{T_i}\hat{Y}_{T_i}\right]\rightarrow 0.
\end{equation}
For any $\delta\in [0,\delta_0]$, we have
\begin{equation}
  \mathbb{E}\left[\int_{0}^{\infty}\mathrm{e}^{-Kt}\left(\hat{X}_{t}^{2} + \hat{Y}_{t}^{2}\right)\mathrm{d}t\right]
  \le \frac{1}{2} \mathbb{E}\left[\int_0^{\infty}\mathrm{e}^{-Kt}\left(\hat{x}_t^2 + \hat{y}_t^2\right)\mathrm{d}t\right].
\end{equation}
Therefore $\Phi$ is a contraction.

By repeating this procedure
for $[1/\delta_0]$ many times, we conclude that there exists a solution to (\ref{eq: lambda}) with $\lambda = 1$. In
particular,  we get a $L_K^2$ solution to (\ref{eq: fbsde}).
\qed

We will also use the following standard consequence of the proof of
Theorem~\ref{thm: fbsde}.  It applies equally to finite-dimensional
linear systems and to the linear mean-field systems appearing below.

\begin{lem}
\label{lem: affine stability}
For $n\in\mathbb N\cup\{\infty\}$, suppose that
Assumption~\ref{assum: fbsde} holds with $K=r$ and constants independent
of $n$ for the affine linear FBSDE
\begin{equation}
\label{eq: approximate linear fbsde}
\begin{cases}
 \rmd U_t^n=\bigl(\alpha_t^nU_t^n+\beta_t^nV_t^n+f_t^n\bigr)\rmd t,\\
 \rmd V_t^n=-\bigl(\gamma_t^nU_t^n+D_t^nV_t^n-rV_t^n+g_t^n\bigr)\rmd t
                 +R_t^n\rmd B_t,\\
 U_0^n=\eta^n.
\end{cases}
\end{equation}
The coefficients are progressively measurable and uniformly bounded in
$n$.

Assume that $f^n,g^n\in L_r^2$ and that
\begin{equation}
\begin{gathered}
 \eta^n\longrightarrow\eta^\infty\quad\hbox{in }\dbL^2,\\
 \|\alpha^n-\alpha^\infty\|_{L_r^2}+\|\beta^n-\beta^\infty\|_{L_r^2}+
 \|\gamma^n-\gamma^\infty\|_{L_r^2}+\|D^n-D^\infty\|_{L_r^2}
 \xrightarrow{n\to\infty}0,\\
 \|f^n-f^\infty\|_{L_r^2}+\|g^n-g^\infty\|_{L_r^2}
 \xrightarrow{n\to\infty}0.
\end{gathered}
\label{eq: approximate linear fbsde convergence assumptions}
\end{equation}
Then every equation in \eqref{eq: approximate linear fbsde} has a unique
$L_r^2$ solution, and
\begin{equation}
\begin{aligned}
 &\dbE|V_0^n-V_0^\infty|^2+\|U^n-U^\infty\|_r^2
  +\|V^n-V^\infty\|_r^2
 \xrightarrow{n\to\infty}0.
\end{aligned}
\label{eq: affine stability}
\end{equation}
\end{lem}

\proof
Existence and uniqueness for each fixed $n$ follow from the continuation
argument in Theorem~\ref{thm: fbsde}; the additive terms $f^n,g^n$ do not
alter the monotonicity of differences.  To prove convergence, set
\begin{equation}
 \Delta U^n\triangleq U^n-U^\infty,\qquad
 \Delta V^n\triangleq V^n-V^\infty,\qquad
 \Delta R^n\triangleq R^n-R^\infty.
\end{equation}
Subtracting the $n=\infty$ equation from the $n$th equation gives
\begin{equation}
\label{eq: approximate linear error fbsde}
\begin{cases}
 \rmd\Delta U_t^n
 =\bigl(\alpha_t^n\Delta U_t^n+\beta_t^n\Delta V_t^n+e_{f,t}^n\bigr)\rmd t,\\
 \rmd\Delta V_t^n
 =-\bigl(\gamma_t^n\Delta U_t^n+D_t^n\Delta V_t^n-r\Delta V_t^n+e_{g,t}^n\bigr)\rmd t
   +\Delta R_t^n\rmd B_t,\\
 \Delta U_0^n=\eta^n-\eta^\infty,
\end{cases}
\end{equation}
where the forward and backward remainders are
\begin{align}
 e_{f,t}^n\triangleq{}&f_t^n-f_t^\infty
 +(\alpha_t^n-\alpha_t^\infty)U_t^\infty
 +(\beta_t^n-\beta_t^\infty)V_t^\infty,\nonumber\\
 e_{g,t}^n\triangleq{}&g_t^n-g_t^\infty
 +(\gamma_t^n-\gamma_t^\infty)U_t^\infty
 +(D_t^n-D_t^\infty)V_t^\infty.
 \label{eq: affine coefficient perturbation sources}
\end{align}
Applying It\^o's formula to
$\rme^{-rt}\Delta U_t^n\Delta V_t^n$ and then to
$\rme^{-rt}|\Delta V_t^n|^2$ gives
\begin{equation}
\begin{aligned}
 &\dbE|\Delta V_0^n|^2+\|\Delta U^n\|_r^2
 +\|\Delta V^n\|_r^2+\|\Delta R^n\|_r^2\\
 &\qquad\leq C\left(\dbE|\eta^n-\eta^\infty|^2
 +\|e_f^n\|_{L_r^2}^2+\|e_g^n\|_{L_r^2}^2\right),
 \label{eq: affine coefficient perturbation estimate}
\end{aligned}
\end{equation}
where $C$ is independent of $n$.  The initial-data term on the right
converges to zero by
\eqref{eq: approximate linear fbsde convergence assumptions}.
The direct source differences in $e_f^n,e_g^n$ converge to zero in
$L_r^2$ by the same assumption.  For the remaining coefficient terms,
the assumed $L_r^2$ convergence and uniform boundedness imply convergence
in measure with respect to the finite measure
$\rme^{-rt}\rmd t\otimes\dbP$.  The dominated convergence theorem,
using $(U^\infty,V^\infty)\in L_r^2\times L_r^2$, therefore gives
\begin{equation}
 \|e_f^n\|_{L_r^2}+\|e_g^n\|_{L_r^2}\xrightarrow{n\to\infty}0.
\end{equation}
The right-hand side of \eqref{eq: affine coefficient perturbation estimate}
thus tends to zero, which proves \eqref{eq: affine stability}.
\qed

\subsection{Distributional uniqueness }

At first we denote by $C([0,\infty);\dbR)$ the space of continuous $\dbR$-valued functions on $[0,\infty)$
equipped with the metric of uniform convergence on compacts:
\begin{equation}
  d(\omega^1,\omega^2)=\sum_{n\ge0}2^{-n}\sup_{t\in[0,n]}\min(|\omega_t^1-\omega_t^2|,1).
\end{equation}
Then, let $C_0([0,\infty);\dbR)$
 denote the subspace of $C([0,\infty);\dbR)$ consisting of functions satisfying $\omega(0)=0$, 
 which is endowed with the Wiener measure $\mu_W$ (the law of Brownian motion).
Let $\cB_t=\sigma\{\omega(s),0\le s\le t\}$ and denote by $\bar{\cB}_\infty $
 the completion of $\cB_\infty$
 with respect to $\mu_W$. Then $(C_0([0,\infty);\dbR),\bar{\cB}_\infty,\mu_W)$
 forms a complete probability space, and the filtration $\{\bar{\cB}_t, t\ge 0 \}$satisfies the usual conditions.

Now we consider two functions $\phi,\psi: \dbR\times \dbR^+\times C_0([0,\infty);\dbR)\rightarrow \dbR^m$ which satisfy the following conditions:
\begin{enumerate}
  \item $\phi,\psi$ are $\cB(\dbR)\times \cB(\dbR^+)\times \bar{\cB}_\infty/\cB(\dbR^m)$-measurable;
  \item for any $x\in\dbR, \omega\in C_0([0,\infty);\dbR)$, we have $\phi(x,\cdot,\omega),\psi(x,\cdot,\omega)\in C([0,\infty);\dbR^m)$;
  \item for any $x\in \dbR, t\in\dbR^+$, $\phi(x,t,\cdot)$ and $\psi(x,t,\cdot)$ are $\bar{\cB}_t/\cB(\dbR^m)$-measurable.
\end{enumerate}
Under the above conditions, we make an additional assumption that for any $\eta\in \dbL^2(\cF_0)$, 
the adapted processes $(\phi(\eta,t,B),\psi(\eta,t,B))_{t\ge0 }$   belong to $L_r^2$. 
Indeed, we will subsequently prove that all the solutions to the FBSDEs admit this representation.

We attempt to prove the Yamada-Watanabe theorem for a broader class of  FBSDEs, 
stating that pathwise uniqueness implies uniqueness in distribution.
Consider the FBSDE of the form: 
\begin{equation}
  \label{eq: 6fbsde}
  \begin{cases}
    \rmd X_t=G(t,X_t,Y_t,\mathcal{L}_{(X_t,\phi_t(\eta, B))}, \psi_t(\eta,B) )\rmd t+ \rmd B_t, \\
    \rmd Y_t=-F(t,X_t,Y_t,\mathcal{L}_{(X_t,\phi_t(\eta, B))}, \psi_t(\eta,B))\rmd t+Z_t\rmd B_t,\\
    X_0=\xi.
    \end{cases}
\end{equation}
where $\xi, \eta\in \dbL^2(\cF_0)$ and 
$G, F: \dbR^+\times \dbR\times \dbR\times \cP_2(\dbR^{m+1})\times\dbR^m \rightarrow \dbR$
are two measurable functions.
We then give the definitions of strong and weak uniqueness.

\begin{defn}[Strong uniqueness]
  We say that the strong uniqueness holds for FBSDE (\ref{eq: 6fbsde}) if on any  set-up
  $(\Omega,\cF,\dbP,\dbF)$ with inputs $(\xi,\eta, B)$, for any two $\dbF$-progressively measurable $L^2_r$
  three-tuples
  \begin{equation}
    (X^1_t,Y^1_t,Z^1_t)_{t\ge 0},\quad (X^2_t,Y^2_t,Z^2_t)_{t\ge 0}
  \end{equation}
  satisfying the FBSDE (\ref{eq: 6fbsde}) with the same initial condition $(\xi,\eta)$ (up to an exceptional event),
  it holds that
  \begin{equation}
    \dbE\left[\int_0^\infty \rme^{-rt} \left( |X^1_t-X^2_t|^2+
     |Y^1_t-Y^2_t|^2 +|Z^1_t-Z^2_t|^2 \right)\rmd t\right]=0
  \end{equation}
\end{defn}

\begin{defn}[Weak uniqueness]
For any two set-ups
\((\Omega^i,\mathcal F^i,\dbP^i,\dbF^i)\), \(i=1,2\), with inputs
\((\xi^i,\eta^i,B^i)\) such that \((\xi^1,\eta^1)\) and
\((\xi^2,\eta^2)\) have the same joint law on \(\dbR^2\), we say that
weak uniqueness holds for FBSDE \eqref{eq: 6fbsde} if, for any
\(L_r^2\) solutions \((X_t^i,Y_t^i,Z_t^i)_{t\geq0}\) on the
corresponding set-ups, the processes
\[
 (X_t^1,Y_t^1,\int_0^tZ_s^1\,\rmd s)_{t\geq0},
 \qquad
 (X_t^2,Y_t^2,\int_0^tZ_s^2\,\rmd s)_{t\geq0}
\]
have the same distribution.
\end{defn}

We use the same scheme as the one developed by Yamada and Watanabe to
 prove that path-wise uniqueness of solutions of FBSDE implies uniqueness in the sense of
 probability law.

\begin{thm}
  Assume that the strong uniqueness property above holds for
  \eqref{eq: 6fbsde}, and that the equation admits an $L_r^2$ solution on the
  filtered spaces under consideration.  In particular, on any one such
  space with inputs $(\xi,\eta,B)$, write the solution as
  $(X_t,Y_t,Z_t)_{t\ge0}$.
  Then the law of $(X_t,Y_t,\int_0^tZ_s\rmd s)_{t\ge 0}$ only depends on $\mathcal{L}_{(\xi,\eta)}$.
  Moreover, there exists a measurable function $\Phi:\dbR^2\times \dbR^+\times C_0([0,\infty);\dbR)\rightarrow \dbR^3$,
  satisfying:
  \begin{enumerate}
    \item $\Phi$ is $\cB(\dbR^2)\times \cB(\dbR^+)\times \bar{\cB}_\infty/\cB(\dbR^3)$-measurable;
  \item for any $(a,b)\in\dbR^2, \omega\in C_0([0,\infty);\dbR)$, we have $\Phi(a,b,\cdot,\omega)\in C([0,\infty);\dbR^3)$; and
  \item for any $(a,b)\in\dbR^2, t\in\dbR^+$, $\Phi(a,b,t,\cdot)$ is $\bar{\cB}_t/\cB(\dbR^3)$-measurable,
  \end{enumerate}
  such that, for any $t\ge 0$, we have $\dbP$-a.s.,
  \begin{equation}
    (X_t,Y_t,\int_0^tZ_s\rmd s)=\Phi(\xi,\eta,t,B).
  \end{equation}
\end{thm}

\proof

Let us consider  two filtered probability spaces $(\Omega^i, \mathcal{F}^i ,\mathbb{P}^i,\mathbb{F}^i )$
with identically distributed inputs $(\xi^i,\eta^i,B^i)$, $i=1,2$,
on each of which a solution $(X_t^i,Y_t^i,\int_0^tZ^i_s\rmd s)_{t\ge0}$
to the FBSDE (\ref{eq: 6fbsde}) is defined.
Define
\begin{equation}
  \begin{split}
  &\Oin\triangleq\dbR^2\times C_0([0,\infty);\dbR),\\ &\Oout\triangleq C([0,\infty);\dbR)\times C([0,\infty);\dbR)\times C([0,\infty);\dbR),\\
&\Ocan\triangleq \Oin\times \Oout.
  \end{split}
\end{equation}
For \(i=1,2\), denote by \(Q^i\) the law of
\[
 (\xi^i,\eta^i,B_t^i,X_t^i,Y_t^i,\int_0^tZ_s^i\,\rmd s)_{t\geq0}
\]
on \(\Ocan=\Oin\times\Oout\), and by \(\Qin\) the common law of
\((\xi^i,\eta^i,B^i)\) on \(\Oin\).

Let us now define for $i\in\{1,2\}$,
\begin{equation}
  q^i(\oin;F):\Oin\times \mathcal{B} (\Oout)\rightarrow [0,1]
\end{equation}
 as the regular conditional probability for $\mathcal{B} (\Oout)$ given $\oin\in \Oin$ (under $Q^i$). It satisfies:
\begin{itemize}
  \item $\forall \oin\in \Oin$, $q^i(\oin; \cdot)$ is a probability measure on $(\Oout,\mathcal{B} (\Oout))$.
  \item For every \(F\in\mathcal B(\Oout)\), the map
  \(\oin\mapsto q^i(\oin;F)\) is measurable with respect to
  \(\cB(\dbR^2)\otimes\cB(C_0([0,\infty);\dbR))\).
  \item $\forall F\in \cB(\Oout), \forall G\in \cB(\Oin)$:
        \begin{equation}
          Q^i(G\times F)=\int _G q^i(\oin;F)\Qin(\rmd \oin).
        \end{equation}
\end{itemize}

Next, we need an enlarged space $(\Oto, \mathcal{G}, Q )$ to support all processes. Define
\begin{equation}
  \Oto\triangleq \Oin \times\Oout\times\Oout,
\end{equation}
 and $\mathcal{G} $ is the
 completion of the $\sigma$-field $\cB(\Ocan)\otimes \cB(\Oout)$ by the collection $\cN$ of all null sets under the
 probability measure
 \begin{equation}
  Q(G\times F_1\times F_2)=\int_G q^1(\oin;F_1)q^2(\oin;F_2)\Qin(\rmd \oin),
 \end{equation}
where $F_1,F_2\in\cB(\Oout), G\in \cB(\Oin)$.

We observe that $Q(G\times F_1\times\Oout)=Q^1(G\times F_1), Q(G\times \Oout\times F_2)=Q^2(G\times F_2).$
In particular, we denote by $(a,b,w,x^1,y^1,\zeta^1,x^2,y^2,\zeta^2)$ the canonical process on $\Oto$,
then $(a,b,w,x^1,y^1,\zeta^1)$ has distribution $Q^1$ and $(a,b,w,x^2,y^2,\zeta^2)$ has distribution $Q^2$.

For $i\in\{1,2\}$, define the derivative of the third canonical coordinate
by
\begin{equation}
  z^i_t=\begin{cases}
 \lim_{n\rightarrow \infty} n\left( \zeta^i_t-\zeta^i_{(t-\frac{1}{n})+}  \right)\quad  ~\mbox{if the limit exists,}
  \\ 0\quad ~\mbox{otherwise.}
   \end{cases}
\end{equation}
The difference quotients are progressively measurable, and hence so is
$z^i$.  Since $\dbP^i$-a.s. $\int_0^tZ_s^i\rmd s$ is absolutely continuous on
every finite interval, the same is true for $\zeta^i$ under $Q$ because its
canonical marginal is $Q^i$.  Therefore
\begin{equation}
  \zeta^i_t=\int_0^tz^i_s\rmd s,\quad t\ge0.
\end{equation}
In particular, $z^i$ is the canonical version of $Z^i$.  Fatou's lemma and
Jensen's inequality give
\begin{equation}
  \begin{split}
  \dbE^Q\int_0^\infty \rme^{-rt} |z^i_t|^2\rmd t 
  &\le  \liminf_{n\rightarrow \infty}\dbE^Q\int_0^\infty \rme^{-rt}|n(\zeta^i_t-\zeta^i_{(t-\frac{1}{n})+})|^2\rmd t\\
  &=  \liminf_{n\rightarrow \infty}n^2\dbE^{\dbP^i}\int_0^\infty \rme^{-rt}|\int_{(t-\frac{1}{n})+}^tZ_s^i\rmd s |^2\rmd t\\
  &\le \liminf_{n\rightarrow \infty}n\dbE^{\dbP^i}\int_0^\infty \rme^{-rt}\int_{(t-\frac{1}{n})+}^t|Z_s^i|^2\rmd s \rmd t\\
  &\le \dbE^{\dbP^i}\int_0^\infty \rme^{-rs}|Z_s^i|^2\rmd s<\infty,
  \end{split}
\end{equation}
where the last inequality follows from Fubini's theorem and
$n\int_s^{s+1/n}\rme^{-rt}\rmd t\leq\rme^{-rs}$.  Thus $z^i\in L_r^2$ under
the corresponding canonical marginal, and also under $Q$.

Let us now endow $(\Oto, \cG, Q)$ with the filtration $\dbG$, where
$\dbG=\{\cG_t\}_{t\ge 0}$ is the complete and right-continuous augmentation under $Q$
of the canonical filtration
\begin{equation}
  \cH_t=\{(a,b,w_s,x^1_s,y^1_s,\zeta^1_s,x^2_s,y^2_s,\zeta^2_s);0\le s\le t\}
\end{equation}
on $\Oto$.
It's easy to see that $(a,b)$ are $\cG_0$-measurable and that $(w_s,x^1_s,y^1_s,z^1_s,x^2_s,y^2_s,z^2_s)_{t\ge 0}$
are $\{\cG_t\}_{t\ge 0}$-progressively measurable. Moreover, for $i\in\{1,2\}$,
\begin{equation}
  Q\left\{\omega\in \Oto, (a,b,w,x^i,y^i,\zeta^i)\in A\right\}
  =Q^i\left\{ (\xi^i, \eta^i ,B^i,X^i,Y^i,\int_0^\cdot Z^i_s\rmd s)\in A\right\},\quad A\in\cB(\Ocan).
\end{equation}

Actually, we just have to prove that $(w_t)_{t\ge 0}$ is a $\{\cG_t\}_{t\ge 0 }$ Brownian motion:
we follow the proof given in \cite[Remark~1.6]{DELARUE2002209}.
Let us firstly define
\begin{equation}
  \pi_t: C_0([0,\infty);\dbR)\rightarrow C_0([0,t];\dbR),\quad  h\rightarrow h|_{[0,t]},
\end{equation}
and 
\begin{equation}
  \pi'_t:\gamma\rightarrow \gamma_t,\quad h\rightarrow h|_{[0,t]},
\end{equation}
where $\gamma=\Oout$ and 
\begin{equation}
  \gamma_t=C([0,t];\dbR)\times C([0,t];\dbR)\times C([0,t];\dbR).
\end{equation}
Endowing $C([0,t];\dbR)$ and $\gamma_t$ with their Borel $\sigma$-fields, we define
\begin{equation}
  \cK_t\triangleq \sigma\{\pi_t\},\quad \cK'_t\triangleq \sigma\{\pi'_t\}.
\end{equation}
Using the separability of the spaces $C_0([0,t];\dbR)$ and $\gamma_t$,
we see that
\begin{equation}
  \cK_t=\sigma\{w_s; 0\le s \le t\},
\end{equation}
and that $\forall i\in\{1,2\}, \forall A\in \cK'_t$, the set $\{(X^i,Y^i,\int_0^\cdot Z^i_s\rmd s)\in A \}$ belongs to $\cF^i_t$.

 Now, considering $A\in\cK'_t$ , we want to show that, for $i\in\{1,2\}$, the map
 \begin{equation}
  \Oin\rightarrow [0,1],\quad (a,b,w)\rightarrow q^i(a,b,w;A)
 \end{equation}
is measurable with respect to the completion of the $\sigma$-field $\cB(\dbR^2)\otimes \cK_t$ under the
probability measure $\Qin$, denoted $\overline{\cB(\dbR^2)\otimes \cK_t }$.
Indeed, let us consider $F\in\cB(\dbR^2)$, $G_1\in\cK_t$  and $G_2\in \sigma \{w_s-w_t; s\ge t\}$. Then, $\forall i\in\{1,2\}$
\begin{equation}
\begin{split}
  &\int I_F(a,b)I_{G_1}(w)I_{G_2}(w)q^i(\xi,w;A)\Qin(\rmd \xi \rmd w)\\
=& \dbE^{\dbP^i}\left[I_F(\xi^i,\eta^i)I_{G_1}(B^i)I_{G_2}(B^i)I_A(X^i,Y^i,\int_0^{\cdot}Z_s^i\rmd s)   \right]\\
=& \dbE^{\dbP^i}\left[I_F(\xi^i,\eta^i)I_{G_1}(B^i)I_A(X^i,Y^i,\int_0^{\cdot}Z_s^i\rmd s)   \right]
     \dbE^{\dbP^i}\left[I_{G_2}(B^i)  \right]\\
=&\int I_F(a,b)I_{G_1}(w)q^i(\xi,w;A)\Qin(\rmd \xi \rmd w)\int I_{G_2}(w)\Qin(\rmd \xi \rmd w).
\end{split}
\end{equation}
By \cite[Chapter~V, Exercise~17.10]{Rogers_Williams_2000},
the map $(a,b,w)\rightarrow q^i(a,b,w;A)$ is measurable with respect to $\overline{\cB(\dbR^2)\otimes \cK_t }$.

Now we prove that $(w_t)_{t\ge 0}$ is a $\{\cG_t\}_{t\ge 0 }$ Brownian motion.
Let us consider $(A,A')\in(\cK'_t)^2, F\in \cB(\dbR^2), G_1\in\cK_t$  and $G_2\in \sigma \{w_s-w_t; s\ge t\}$.
Then,
\begin{equation}
\begin{split}
  &\dbE^Q \left[ I_F(a,b)I_{G_1}(w)I_{G_2}(w)I_A(x^1,y^1,\zeta^1)I_{A'}(x^2,y^2,\zeta^2)   \right]\\
=&\int I_F(a,b)I_{G_1}(w)I_{G_2}(w)q^1(\xi,w;A)q^2(\xi,w;A')\Qin(\rmd \xi \rmd w)\\
=&\int I_F(a,b)I_{G_1}(w)q^1(\xi,w;A)q^2(\xi,w;A')\Qin(\rmd \xi \rmd w)\int I_{G_2}(w)\Qin(\rmd \xi \rmd w)\\
=&\dbE^Q \left[ I_F(a,b)I_{G_1}(w)I_A(x^1,y^1,\zeta^1)I_{A'}(x^2,y^2,\zeta^2)   \right]
   \dbE^Q \left[I_{G_2}(w)\right].
\end{split}
\end{equation}
Noting that $\cH^t=\cB(\dbR^2)\otimes \cK_t\otimes \cK'_t\otimes \cK'_t$, 
we conclude that $(w_t)_{t\ge 0}$ is a $\{\cG_t\}_{t\ge 0 }$ Brownian motion.

At last, the stochastic-integral transfer lemma in
\cite[Vol.~II, Lemma~1.27]{carmona2018probabilistic} can be applied to
the difference quotients defining $z^i$.  It transfers the finite-time
BSDE identity from $(\dbP^i,B^i,Z^i)$ to the canonical variables.  Thus, on a
common $Q$-full set, for every rational $T\geq0$ and $0\leq t\leq T$,
\begin{equation}
  \begin{cases}
    x^i_t=a+\int_0^t G(s,x^i_s,y^i_s,\mathcal{L}_{(x^i_s, \phi_s(b,w))},\psi_s(b,w))\rmd s+ w_t, \\
   y^i_t=y^i_T+\int_t^T F(s,x^i_s,y^i_s,\mathcal{L}_{(x^i_s, \phi_s(b,w))},\psi_s(b,w))\rmd s-\int_t^T z^i_s\rmd w_s,\\
    \dbE^Q\left[\int_0^\infty \rme^{-rt} \left( |x^i_t|^2+
     |y^i_t|^2 +|z^i_t|^2 \right)\rmd t\right]<\infty.
  \end{cases}
\end{equation}
The identities for all $t$ and $T$ follow by continuity of the two sides.
Thus both canonical triples are solutions on the same filtered probability
space, with the same inputs $(a,b,w)$.  Applying strong uniqueness on this
space gives
\begin{equation}
  \dbE^Q\int_0^\infty\rme^{-rt}
  \left(|x_t^1-x_t^2|^2+|y_t^1-y_t^2|^2+|z_t^1-z_t^2|^2\right)\rmd t=0.
\end{equation}
The continuity of $x^i,y^i,\zeta^i$ then implies that the two output paths
are equal $Q$-a.s.  Let $\Delta$ denote the diagonal in
$\Oout\times\Oout$.  Disintegration gives
\begin{equation}
  1=Q(\Delta)=\int_{\Oin}
  (q^1(\oin;\cdot)\times q^2(\oin;\cdot))(\Delta)\Qin(\rmd\oin).
\end{equation}
Consequently, for $\Qin$-a.e. $\oin$, the product measure
$q^1(\oin;\cdot)\times q^2(\oin;\cdot)$ is concentrated on the diagonal.
A product measure with this property must be a common Dirac mass: there is
a point $\Psi(\oin)\in\Oout$ such that, after modifying on a $\Qin$-null set,
\begin{equation}
  q^1(\oin;\cdot)=q^2(\oin;\cdot)=\delta_{\Psi(\oin)}.
\end{equation}
Then we have, for any $t\ge 0$, $i\in\{1,2\} $, $\dbP^i$-a.s.,
\begin{equation}
  (X_t^i(\omega),Y_t^i(\omega),\int_0^tZ^i_s(\omega)\rmd s)=\Psi_t(\xi^i(\omega),\eta^i(\omega),B^i_{\cdot}(\omega)).
\end{equation}
The standard measurable-Dirac-kernel theorem permits the version $\Psi$ to
be chosen $\cB(\dbR^2)\otimes\bar{\cB}_\infty$-measurable.  Define
$\Phi(a,b,t,w)\triangleq\Psi_t(a,b,w)$.  Since $\Psi(a,b,w)$ is a triple of
continuous paths, $\Phi$ is measurable in $(a,b,t,w)$ and continuous in $t$.
The measurability of the conditional kernels proved above, applied to the
coordinate restrictions up to time $t$, gives that for every fixed
$(a,b,t)$, the map $w\mapsto\Phi(a,b,t,w)$ is
$\bar{\cB}_t/\cB(\dbR^3)$-measurable.  This proves all the asserted
properties of $\Phi$.
Now we finish the proof.
\qed

\subsection{\texorpdfstring{\(L^p\)}{Lp} estimates for scalar linear FBSDEs}
\label{subsec: linear Lp}

The estimates in this subsection will be used for the higher-order
variational equations in Section~\ref{sec: classical}. For \(p\geq2\), set
\[
 \|U\|_{L_r^p}^p\triangleq\dbE\int_0^\infty \rme^{-rt}|U_t|^p\,\rmd t,
 \qquad
 \|R\|_{\mathbb H_r^p}^p
 \triangleq\dbE\left(\int_0^\infty\rme^{-2rt/p}|R_t|^2\,\rmd t\right)^{p/2}.
\]
Let \(\alpha,\beta,\gamma\) be progressively measurable adapted
real-valued processes, and consider the scalar linear FBSDE
\begin{equation}
\label{eq: scalar linear p}
\begin{cases}
 \rmd U_t=(\alpha_tU_t+\beta_tV_t)\,\rmd t,\\
 \rmd V_t=-(\gamma_tU_t+\alpha_tV_t-rV_t)\,\rmd t+R_t\,\rmd B_t,\\
 U_0=\eta.
\end{cases}
\end{equation}
The Riccati construction below follows the method of Guatteri and
Tessitore \cite{GuatteriTessitore2008}.

\begin{lem}
\label{prop: linear Lp}
Let \(a\geq0\), \(b_0,c_0>r/2\), and suppose that
\begin{equation}
\label{eq: linear dissipativity}
 b_0c_0>a(r+3a).
\end{equation}
There are upper bounds \(b_1\geq b_0\) and \(c_1\geq c_0\) such that,
for any progressively measurable coefficients satisfying
\begin{equation}
\label{eq: scalar linear coefficient bounds}
 |\alpha_t|\leq a,\qquad b_0\leq-\beta_t\leq b_1,\qquad
 c_0\leq\gamma_t\leq c_1,
\end{equation}
and any \(\eta\in\dbL^p(\cF_0)\), \(2\leq p<\infty\), the unique
\(L_r^2\) solution of \eqref{eq: scalar linear p} belongs to
\(L_r^p\times L_r^p\times\mathbb H_r^p\).
\end{lem}

\proof
Set
\[
 F_t(q)\triangleq\gamma_t-(r-2\alpha_t)q+\beta_tq^2.
\]
For \(q\geq0\),
\[
 c_0-(r+2a)q-b_1q^2\leq F_t(q)
 \leq c_1-(r-2a)q-b_0q^2.
\]
Define
\begin{equation}
\label{eq: linear upper pi}
 \overline\pi\triangleq
 \frac{\sqrt{(r-2a)^2+4b_0c_1}-(r-2a)}{2b_0},
\end{equation}
and let
\[
 \Pi(q)\triangleq(0\vee q)\wedge\overline\pi,
 \qquad \widehat F_t(q)\triangleq F_t(\Pi(q)).
\]
The generator \(\widehat F\) is globally Lipschitz and
\(|\widehat F_t(0)|\leq c_1\). Thus, by the standard existence and
uniqueness theorem for Lipschitz BSDEs, for every \(T>0\) there is a
unique solution
\((P^T,Q^T)\in\mathcal S^2([0,T])\times\mathcal H^2([0,T])\) of
\begin{equation}
\label{eq: linear finite Riccati}
 P_t^T=\int_t^T\widehat F_s(P_s^T)\,\rmd s
       -\int_t^TQ_s^T\,\rmd B_s,
 \qquad P_T^T=0.
\end{equation}
Since \(\widehat F_t(0)\geq0\) and
\[
 \widehat F_t(\overline\pi)=F_t(\overline\pi)
 \leq c_1-(r-2a)\overline\pi-b_0\overline\pi^2=0,
\]
the comparison principle gives \(0\leq P_t^T\leq\overline\pi\).
Hence \(\widehat F_t(P_t^T)=F_t(P_t^T)\). For \(T_2>T_1\), comparison
on \([0,T_1]\), with terminal values
\(P_{T_1}^{T_2}\geq0=P_{T_1}^{T_1}\), also gives
\(P_t^{T_1}\leq P_t^{T_2}\).

Fix \(n\in\mathbb N\). The terminal values \(P_n^T\), \(T\geq n\),
are increasing and bounded, and therefore converge in \(\dbL^2\).
The usual stability estimate for Lipschitz BSDEs on \([0,n]\) gives
\[
 \dbE\sup_{0\leq t\leq n}|P_t^{T_2}-P_t^{T_1}|^2
 +\dbE\int_0^n|Q_t^{T_2}-Q_t^{T_1}|^2\,\rmd t
 \longrightarrow0
 \quad\text{as }T_1,T_2\to\infty.
\]
The limits on these intervals agree on their overlaps. Letting
\(T\to\infty\), we obtain processes \((P,Q)\), with \(P\) continuous
and bounded and \(Q\) locally square integrable, satisfying
\begin{equation}
\label{eq: linear infinite Riccati}
 \rmd P_t=-\bigl[\gamma_t-(r-2\alpha_t)P_t+\beta_tP_t^2\bigr]\,\rmd t
             +Q_t\,\rmd B_t,
 \qquad 0\leq P_t\leq\overline\pi.
\end{equation}

We now improve the lower bound. Let \(v\) solve
\[
 v'(s)=c_0-(r+2a)v(s)-b_1v(s)^2,\qquad v(0)=0.
\]
Then \(v\) increases to
\begin{equation}
\label{eq: linear lower pi}
 \underline\pi\triangleq
 \frac{\sqrt{(r+2a)^2+4b_1c_0}-(r+2a)}{2b_1}.
\end{equation}
For every \(T>t\), comparison on \([0,T]\) between \(P\), with
terminal value \(P_T\geq0\), and \(v(T-\cdot)\), with terminal value
zero and generator \(c_0-(r+2a)q-b_1q^2\), gives
\(P_t\geq v(T-t)\). Letting \(T\to\infty\), we obtain
\[
 \underline\pi\leq P_t\leq\overline\pi,\qquad t\geq0.
\]

Define
\begin{equation}
\label{eq: linear Riccati representation}
 U_t=\eta\exp\left(\int_0^t(\alpha_s+\beta_sP_s)\,\rmd s\right),
 \qquad V_t=P_tU_t,\qquad R_t=Q_tU_t.
\end{equation}
By the product rule, \((U,V,R)\) satisfies \eqref{eq: scalar linear p}. Set
\(\delta_*\triangleq b_0\underline\pi-a\). For \(a>0\), choosing
\begin{equation}
\label{eq: linear admissible b1}
 b_1\in\left[b_0,\,
 \frac{b_0\bigl[b_0c_0-a(r+2a)\bigr]}{a^2}\right)
\end{equation}
gives \(\delta_*>0\); the interval is nonempty by
\eqref{eq: linear dissipativity}. If \(a=0\), any finite
\(b_1\geq b_0\) gives \(\delta_*>0\). In both cases, \(c_1\) can be
any finite constant with \(c_1\geq c_0\). More generally, the estimates
below hold whenever the coefficient bounds satisfy \(\delta_*>0\).
Since
\(\alpha_t+\beta_tP_t\leq a-b_0\underline\pi=-\delta_*\), we have,
for every \(t\geq0\),
\begin{equation}
\label{eq: linear path estimate}
 |U_t|\leq|\eta|\rme^{-\delta_*t},\qquad
 |V_t|\leq\overline\pi|\eta|\rme^{-\delta_*t}.
\end{equation}
Integration gives
\begin{equation}
\label{eq: linear XY p estimate}
 \|U\|_{L_r^p}^p+\|V\|_{L_r^p}^p
 \leq\frac{1+\overline\pi^p}{r+p\delta_*}\dbE|\eta|^p.
\end{equation}

To estimate \(R\), set
\(\widetilde V_t=\rme^{-rt/p}V_t\),
\(\widetilde R_t=\rme^{-rt/p}R_t\), and
\(H_t=\gamma_tU_t+\alpha_tV_t\). Then
\begin{equation}
 \rmd\widetilde V_t
 =\left(\frac{p-1}{p}r\widetilde V_t
       -\rme^{-rt/p}H_t\right)\rmd t
       +\widetilde R_t\,\rmd B_t.
\end{equation}
The pathwise bounds give
\[
 \sup_{t\geq0}|\widetilde V_t|
 +\int_0^\infty\left|
 \frac{p-1}{p}r\widetilde V_t-\rme^{-rt/p}H_t
 \right|\rmd t\leq C|\eta|.
\]
Consequently, the local martingale
\(M_t\triangleq\int_0^t\widetilde R_s\,\rmd B_s\) satisfies
\(\sup_{t\geq0}|M_t|\leq C|\eta|\). The
Burkholder--Davis--Gundy inequality, applied after localization, yields
\begin{equation}
\label{eq: linear Z p estimate}
 \|R\|_{\mathbb H_r^p}^p
 \leq C_p\dbE\sup_{t\geq0}|M_t|^p
 \leq C_p\dbE|\eta|^p,
\end{equation}
where \(C_p\) depends only on \(p,r,a,b_0,b_1,c_0,c_1\).
Taking \(p=2\) shows that the constructed triple belongs to \(L_r^2\),
and Theorem~\ref{thm: fbsde} identifies it with the \(L_r^2\) solution
of \eqref{eq: scalar linear p}. \qed

\section{Lions-differentiability of the value function's derivative}
\label{sec: diff}
\setcounter{equation}{0}

In this section, for the $L_r^2$ solutions of FBSDEs (\ref{eq: 11}) and (\ref{eq: 12}) with some
  $\xi\in\dbL^2(\cF_0)$,
we  define $\mcv(x,\mu)\triangleq Y_0^{x,\xi}$.
We will prove that $\mcv(x,\mu)$ is
Lions-differentiable with respect to $\mu$. 
To this end, we first establish the existence of the function $\pa_\mu \mcv(x,\mu,\tilde{x}):\dbR\times\cP_2\times\dbR\to\dbR$, 
which satisfies
\begin{equation}
  \lim_{\delta\to 0}\frac{1}{\delta}\left[ 
\mcv(x,\cL_{\xi+\delta\eta})-\mcv(x,\cL_{\xi})
  \right] =\dbE[\pa_{\mu}\mcv(x,\cL_{\xi},\xi)\cdot\eta]
\end{equation}
for all $\xi,\eta\in\dbL^2(\cF_0)$.
Referring to the construction of the Lions-derivative for the parabolic master equation in \cite{gangbo2022mean,mou2024wellposedness},
we derive a representation 
of $\pa_\mu\mcv$ via  solutions of FBSDEs,
such that $\pa_\mu \mcv(x,\mu,\tilde{x})$ is uniformly bounded and jointly continuous.
Consequently, $\pa_\mu \mcv(x,\cL_\xi,\xi)$ becomes continuous in $\xi$ with respect to the $\dbL^2$-norm, 
which ensures the Lions-differentiability of $\mcv$ with respect to the measure $\mu$; the derivative
$\pa_\mu \mcv(x,\mu,\tilde{x})$
is then identified as the Lions-derivative.

We require the following stronger conditions in this section,
and it is straightforward to verify that they subsume \cite[Assumption~4.1]{song2025infinite}.
Therefore both FBSDEs (\ref{eq: 11}) and (\ref{eq: 12}) admit unique $L_r^2$ solutions.
\begin{assum}
  \label{assum: Hb}

\noindent (i) $H(x,\mu,y)$ has at most quadratic growth.  
$\pa_x H, \pa_y H, \pa_{xx}H,\pa_{xy}H, \pa_{yy}H, $ \\
$\quad \pa_{\mu}H(x,\mu,y,\tilde{x}),
\pa_{x\mu}H(x,\mu,y,\tilde{x}),\pa_{y\mu}H(x,\mu,y,\tilde{x})$ exist and are Lipschitz continuous.

\noindent (ii) There exist constants $\lambda_1, \lambda_2>0$
such that $-\lambda_1+2\lambda_2 <-r/2$ and
\begin{equation}
\begin{aligned}
& \partial_{yy} H(x, \mu, y) \leq -\lambda_1, 
&\quad& \partial_{xx} H(x, \mu, y) \geq \lambda_1, 
 \\
& \left| \partial_{x\mu} H(x, \mu, y, \tilde{x}) \right| \leq \lambda_2, 
&\quad& \left| \partial_{y\mu} H(x, \mu, y, \tilde{x}) \right| \leq \lambda_2.
\end{aligned}
\end{equation}

\noindent (iii) There exist a constant $\lambda_3>0$ such that
  \begin{equation}
    |\pa_{xx}H(x,\mu,y)|\leq\lambda_3,\quad
    |\pa_{yy}H(x,\mu,y)|\leq \lambda_3\quad |\pa_{xy}H(x,\mu,y)|\leq \lambda_3.
  \end{equation}
\end{assum}

\subsection{Existence of directional derivative}

For arbitrary $\xi,\eta\in \dbL^2(\cF_0), x\in\dbR$,  we introduce the following FBSDEs:
\begin{equation}
  \label{eq: dX}
  \begin{cases}
    \begin{aligned}
      \rmd \dX_t=&\left\{\dX_t \partial_{xy}H(X_{t}^{\xi},\mathcal{L} _{X_{t}^{\xi}},Y_{t}^{\xi}) 
                        +\dY_t \partial_{yy}H(X_{t}^{\xi},\mathcal{L} _{X_{t}^{\xi}},Y_{t}^{\xi}) \right.\\
          &\left.+\tilde{\dbE}_{\cF_t} \left[
            \partial_{y\mu}H(X_{t}^{\xi},\mathcal{L} _{X_{t}^{\xi}},Y_{t}^{\xi},\tilde{X}_{t}^{\xi}) \tdX_t
          \right]\right\}\rmd t,
    \end{aligned}\\
    \begin{aligned}
      \rmd \dY_t=-&\left\{\dX_t \partial_{xx}H(X_{t}^{\xi},\mathcal{L} _{X_{t}^{\xi}},Y_{t}^{\xi}) 
                        +\dY_t \partial_{xy}H(X_{t}^{\xi},\mathcal{L} _{X_{t}^{\xi}},Y_{t}^{\xi})-r\dY_t
                         \right.\\
          &\left.+\tilde{\dbE}_{\cF_t} \left[
            \partial_{x\mu}H(X_{t}^{\xi},\mathcal{L} _{X_{t}^{\xi}},Y_{t}^{\xi},\tilde{X}_{t}^{\xi}) \tdX_t
          \right]\right\}\rmd t +  \dZ_t\rmd B_t ,
    \end{aligned}\\
    \dX_0=\eta;
  \end{cases}
\end{equation}

\begin{equation}
  \label{eq: dxX}
  \begin{cases}
    \begin{aligned}
      \rmd \dxX_t=&\left\{\dxX_t \partial_{xy}H(X_{t}^{x,\xi},\mathcal{L} _{X_{t}^{\xi}},Y_{t}^{x,\xi}) 
                        +\dxY_t \partial_{yy}H(X_{t}^{x,\xi},\mathcal{L} _{X_{t}^{\xi}},Y_{t}^{x,\xi}) \right.\\
          &\left.+\tilde{\dbE}_{\cF_t} \left[
            \partial_{y\mu}H(X_{t}^{x,\xi},\mathcal{L} _{X_{t}^{\xi}},Y_{t}^{x,\xi},\tilde{X}_{t}^{\xi}) \tdX_t
          \right]\right\}\rmd t,
    \end{aligned}\\
    \begin{aligned}
      \rmd \dxY_t=-&\left\{\dxX_t \partial_{xx}H(X_{t}^{x,\xi},\mathcal{L} _{X_{t}^{\xi}},Y_{t}^{x,\xi}) 
                        +\dxY_t \partial_{xy}H(X_{t}^{x,\xi},\mathcal{L} _{X_{t}^{\xi}},Y_{t}^{x,\xi})
                         \right.\\
          &\left.-r\dxY_t+\tilde{\dbE}_{\cF_t} \left[
            \partial_{x\mu}H(X_{t}^{x,\xi},\mathcal{L} _{X_{t}^{\xi}},Y_{t}^{x,\xi},\tilde{X}_{t}^{\xi}) \tdX_t
          \right]\right\}\rmd t\\& +  \dxZ_t\rmd B_t ,
    \end{aligned}\\
    \dxX_0=0.
  \end{cases}
\end{equation}
Under Assumption \ref{assum: Hb}, 
it's clear that the above FBSDEs  admit  unique solutions in $L_r^2$.
By an argument similar to the proof of \cite[Theorem~4.2]{song2025infinite}, we can prove that
  \begin{equation}
    \begin{aligned}
    &\lim_{\delta\to 0}\left\lVert \frac{1}{\delta} \left(X^{\xi+\delta \eta}-X^{\xi} \right)-\dX \right\rVert _r= 0,\quad
    \quad\;\;\, \lim_{\delta\to 0}\left\lVert \frac{1}{\delta} \left(Y^{\xi+\delta \eta}-Y^{\xi} \right)-\dY \right\rVert _r= 0,\\     
    &\lim_{\delta\to 0}\left\lVert \frac{1}{\delta} \left(X^{x,\xi+\delta \eta}-X^{x,\xi} \right)-\dxX \right\rVert _r= 0,\quad
     \lim_{\delta\to 0}\left\lVert \frac{1}{\delta} \left(Y^{x,\xi+\delta \eta}-Y^{x,\xi} \right)-\dxY \right\rVert _r= 0.     
    \end{aligned}
  \end{equation}
And most importantly,
\begin{equation}
  \lim_{\delta\to 0}\left\lvert \frac{1}{\delta} (Y_0^{x,\xi+\delta \eta}-Y_0^{x,\xi})-\dxY_0\right\rvert =0.
\end{equation}

\begin{lem}
\label{lem: blf}
  Fix $x\in\dbR,\xi\in \dbL^2(\cF_0)$,
  the mapping $\eta\to \dxY_0$ is a bounded linear functional on the Hilbert space $\dbL^2(\cF_0)$.
\end{lem}
\proof
By the linearity of the FBSDE (\ref{eq: dX}) and (\ref{eq: dxX}), 
it is straightforward to see that the mapping is linear. 
Thus, it suffices to prove its boundedness.
Suppose that $\dbE[\eta^2]\leq 1$. Applying It\^{o}'s formula
to $\rme^{-rt}\dX_t\dY_t$ and integating it from 0 to infinity, we can get that

\begin{equation}
\begin{aligned}
-\dbE[\eta \dY_0] = &\dbE\bigg[ \int_0^\infty \rme^{-rt} \bigg( 
 (\dY_t)^2 \, \partial_{yy}H(X_t^\xi, \mathcal{L}_{X_t^\xi}, Y_t^\xi) 
- (\dX_t)^2 \, \partial_{xx}H(X_t^\xi, \mathcal{L}_{X_t^\xi}, Y_t^\xi) \\
&+ \dY_t \cdot \tilde{\dbE}_{\mathcal{F}_t}\left[\partial_{y\mu}H(X_t^\xi, \mathcal{L}_{X_t^\xi}, Y_t^\xi, \tilde{X}_t^\xi) \tdX_t\right] \\
&- \dX_t \cdot \tilde{\dbE}_{\mathcal{F}_t}\left[\partial_{x\mu}H(X_t^\xi, \mathcal{L}_{X_t^\xi}, Y_t^\xi, \tilde{X}_t^\xi) \tdX_t\right] 
\bigg) \rmd t \bigg]\\
\leq &-\frac{r}{2} \dbE\bigg[ \int_0^\infty \rme^{-rt} 
\bigg(  (\dX_t)^2 +
 (\dY_t)^2 \bigg)  \rmd t \bigg].
\end{aligned}
\end{equation}
Therefore,
\begin{equation}
  \label{eq: norm v eta}
\begin{aligned}
  \dbE\bigg[ \int_0^\infty \rme^{-rt} 
\bigg(  (\dX_t)^2 +
 (\dY_t)^2 \bigg)  \rmd t \bigg]\leq& \frac{2}{r}\dbE[\eta \dY_0]\\
 \leq & \epsilon \dbE[(\dY_0)^2]+\frac{1}{\epsilon r^2}.
\end{aligned}
\end{equation}
Applying It\^{o}'s formula
to $\rme^{-rt}(\dY_t)^2$, we can get that
\begin{equation}
\begin{aligned}
\dbE[(\dY_0)^2] =& \dbE\bigg[ \int_0^\infty \rme^{-rt} \bigg( 
 2\dX_t \dY_t \, \partial_{xx}H(X_t^\xi, \mathcal{L}_{X_t^\xi}, Y_t^\xi) 
+ 2(\dY_t)^2 \, \partial_{xy}H(X_t^\xi, \mathcal{L}_{X_t^\xi}, Y_t^\xi) \\
&+ 2\dY_t \cdot \tilde{\dbE}_{\mathcal{F}_t}\left[\partial_{x\mu}H(X_t^\xi, \mathcal{L}_{X_t^\xi}, Y_t^\xi, \tilde{X}_t^\xi) \tdX_t\right] \\
&- r(\dY_t)^2 - (\dZ_t)^2
\bigg) \rmd t \bigg]\\
\leq & (\lambda_2+3\lambda_3) \dbE\bigg[ \int_0^\infty \rme^{-rt} 
\bigg(  (\dX_t)^2 +
 (\dY_t)^2 \bigg)  \rmd t \bigg].
\end{aligned}
\end{equation}
Combining it with (\ref{eq: norm v eta}) and taking $\epsilon=\frac{1}{2(\lambda_2+3\lambda_3)}$,
we can get that
\begin{equation}
  \dbE\bigg[ \int_0^\infty \rme^{-rt} 
\bigg(  (\dX_t)^2 +
 (\dY_t)^2 \bigg)  \rmd t \bigg]\leq C_1,
\end{equation}
where $C_1>0$ is a constant.
At last, we apply the same procedure 
to $\rme^{-rt}\dxX_t\dxY_t$ and $\rme^{-rt}(\dY_t)^2$.
We can get that 
\begin{equation}
\dbE\bigg[ \int_0^\infty \rme^{-rt} 
\bigg(  (\dxX_t)^2 +
 (\dxY_t)^2 \bigg)  \rmd t \bigg]\leq
 C_2\dbE\bigg[ \int_0^\infty \rme^{-rt} 
\bigg(  (\dX_t)^2 +
 (\dY_t)^2 \bigg)  \rmd t \bigg]
\end{equation}
and
\begin{equation}
\begin{aligned}
  (\dxY_0)^2\leq & C_3 \dbE\bigg[ \int_0^\infty \rme^{-rt} 
\bigg(  (\dxX_t)^2 +
 (\dxY_t)^2 \bigg)  \rmd t \bigg]\\
 &+C_4\dbE\bigg[ \int_0^\infty \rme^{-rt} 
\bigg(  (\dX_t)^2 +
 (\dY_t)^2 \bigg)  \rmd t \bigg],
\end{aligned}
\end{equation}
where $C_2,C_3,C_4$ are all positive constants.
Thus, we have proved that $\dxY_0$ is bounded for fixed $(x,\xi)$ and all  $\eta$
such that $\dbE[\eta^2]\leq 1$.
\qed

\begin{cor}
  \label{cor: cont}
  The following mappings are jointly continuous:
  \begin{enumerate}
    \item $\dbL^2(\cF_0)\times\dbL^2(\cF_0)\to L_r^2\times L_r^2\times \dbL^2(\cF_0) : (\xi,\eta) \to (\dX, \dY, \dY_0)$,
    \item $\dbR\times\dbL^2(\cF_0)\times\dbL^2(\cF_0)\to L_r^2\times L_r^2\times \dbR: (x, \xi, \eta) \to (\dxX, \dxY, \dxY_0)$.
  \end{enumerate}
\end{cor}
\proof
Fix some $(x,\xi,\eta)\in \dbR\times\dbL^2(\cF_0)\times\dbL^2(\cF_0)$
and consider a sequence $\{(x_n,\xi_n,\eta_n)\in \dbR\times\dbL^2(\cF_0)\times\dbL^2(\cF_0);n\geq 1\}$
such that 
\begin{equation}
  \lim_{n\to\infty}\left(|x_n-x|+ \dbE[|\xi_n-\xi|^2]+ \dbE[|\eta_n-\eta|^2]      \right)=0.
\end{equation}
By Lemma \ref{lem: blf}, all solutions to the FBSDEs that appear in the subsequent proof are uniformly bounded.
Denote
\begin{equation}
  X^n=\delta X^{\xi_n,\eta_n}-\dX,\quad 
  Y^n=\delta Y^{\xi_n,\eta_n}-\dY,\quad 
  Z^n=\delta Z^{\xi_n,\eta_n}-\dZ.
\end{equation}
By applying Itô's formula to $\rme^{-rt}X^n_t Y^n_t$
and $\rme^{-rt}(Y^n_t)^2$
(instead of $\rme^{-rt}\dX_t\dY_t$ and $\rme^{-rt}(\dY_t)^2$ in the proof of Lemma \ref{lem: blf}), 
we can get that
\begin{equation}
  \dbE\bigg[ \int_0^\infty \rme^{-rt} 
\bigg(  (X^n_t)^2 +
 (Y^n_t)^2 \bigg)  \rmd t \bigg]\leq 
 C\dbE\bigg[|\eta_n-\eta|^2+ \int_0^\infty \rme^{-rt} 
\bigg( | X^n_tA_t| +
 |Y^n_tB_t| \bigg)  \rmd t \bigg],
\end{equation}
where $C$ is a positive constant and 
\begin{equation}
\begin{aligned}
A_t \triangleq& \dX_t \left( \partial_{xx} H(X_t^{\xi_n}, \mathcal{L}_{X_t^{\xi_n}}, Y_t^{\xi_n}) - \partial_{xx} H(X_t^{\xi}, \mathcal{L}_{X_t^\xi}, Y_t^{\xi}) \right) \\
  &+ \dY_t \left(\partial_{xy} H(X_t^{\xi_n}, \mathcal{L}_{X_t^{\xi_n}}, Y_t^{\xi_n}) - \partial_{xy} H(X_t^{\xi}, \mathcal{L}_{X_t^\xi}, Y_t^{\xi}) \right) \\
  &+\tilde{\dbE}_{\mathcal{F}_t}\left[\tdX_t \left( \partial_{x\mu}H(X_t^{\xi_n}, \mathcal{L}_{X_t^{\xi_n}}, Y_t^{\xi_n}, \tilde{X}_t^{\xi_n})-\partial_{x\mu}H(X_t^{\xi}, \mathcal{L}_{X_t^\xi}, Y_t^{\xi}, \tilde{X}_t^{\xi}) \right)\right],\\
B_t \triangleq& \dX_t \left( \partial_{xy} H(X_t^{\xi_n}, \mathcal{L}_{X_t^{\xi_n}}, Y_t^{\xi_n}) - \partial_{xy} H(X_t^{\xi}, \mathcal{L}_{X_t^\xi}, Y_t^{\xi}) \right) \\
  &+ \dY_t \left(\partial_{yy} H(X_t^{\xi_n}, \mathcal{L}_{X_t^{\xi_n}}, Y_t^{\xi_n}) - \partial_{yy} H(X_t^{\xi}, \mathcal{L}_{X_t^\xi}, Y_t^{\xi}) \right) \\
  &+\tilde{\dbE}_{\mathcal{F}_t}\left[\tdX_t \left( \partial_{y\mu}H(X_t^{\xi_n}, \mathcal{L}_{X_t^{\xi_n}}, Y_t^{\xi_n}, \tilde{X}_t^{\xi_n})-\partial_{y\mu}H(X_t^{\xi}, \mathcal{L}_{X_t^\xi}, Y_t^{\xi}, \tilde{X}_t^{\xi}) \right)\right]\\
  &+\dX_t \left( \partial_{xx} H(X_t^{\xi_n}, \mathcal{L}_{X_t^{\xi_n}}, Y_t^{\xi_n}) - \partial_{xx} H(X_t^{\xi}, \mathcal{L}_{X_t^\xi}, Y_t^{\xi}) \right) \\
  &+ \dY_t \left(\partial_{xy} H(X_t^{\xi_n}, \mathcal{L}_{X_t^{\xi_n}}, Y_t^{\xi_n}) - \partial_{xy} H(X_t^{\xi}, \mathcal{L}_{X_t^\xi}, Y_t^{\xi}) \right) \\
  &+\tilde{\dbE}_{\mathcal{F}_t}\left[\tdX_t \left( \partial_{x\mu}H(X_t^{\xi_n}, \mathcal{L}_{X_t^{\xi_n}}, Y_t^{\xi_n}, \tilde{X}_t^{\xi_n})-\partial_{x\mu}H(X_t^{\xi}, \mathcal{L}_{X_t^\xi}, Y_t^{\xi}, \tilde{X}_t^{\xi}) \right)\right].
\end{aligned}
\end{equation}
By dominated convergence, as in the proof of \cite[Theorem~4.2]{song2025infinite}, we obtain:
\begin{equation}
\lim_{n\to\infty}  \mathbb{E} \bigg[(Y^n_0)^2+  \int_0^\infty \rme^{-rt} \biggl( ( X^n_t)^2 + (Y^n_t)^2 \biggr)\mathrm{d}t \bigg]=0.
\end{equation}
Hence, the continuity (1) is proved. The continuity (2) can be proved similarly.
\qed

Since the mapping $\eta\to \dxY_0$ is a bounded linear functional on the Hilbert space $\dbL^2(\cF_0)$,
we can find a random variable $D^{x,\xi}\in \dbL^2(\cF_0)$, such that:
\begin{equation}
  \dxY_0=\dbE[D^{x,\xi}\cdot\eta].
\end{equation}
Whereas the framework in \cite{carmona2018probabilistic} successfully derives the Lions-differentiability of $V(t,x,\mu)$ with respect to
$\mu$ for parabolic 
master equations and finite-time FBSDEs by establishing the continuity of $D^{x,\xi}$ in $\xi$, 
our framework encounters 
a significant hurdle. Specifically, the continuity result in Corollary \ref{cor: cont} is 
merely pointwise, and the 
continuity of $\dxY_0$ with respect to $\xi$ lacks uniformity over $\eta$. 
This fundamental difference prompts our 
investigation into an alternative methodology.

Next, we will prove that for any $x\in \dbR$, $D^{x,\xi}$ can be expressed as $f(\xi)$, 
where $f$ is a Borel function depends only on the distribution of $\xi$. 
This implies the existence of the directional derivative of $\mcv(x,\mu)$.

\begin{thm}
There exists a function $\pa_{\mu}\mcv(x,\mu,\tilde{x}):\dbR\times \cP_2\times \dbR\to \dbR$, such that
\begin{equation}
  \label{eq: mcv dm}
  \lim_{\delta\to 0}\frac{1}{\delta}\left[ 
\mcv(x,\cL_{\xi+\delta\eta})-\mcv(x,\cL_{\xi})
  \right] =\dbE[\pa_{\mu}\mcv(x,\cL_{\xi},\xi)\cdot\eta]
\end{equation}
for any $\xi,\eta\in\dbL^2(\cF_0)$.
\end{thm}
\proof
Fix $x\in\dbR$, we already know that
\begin{equation}
  \lim_{\delta\to 0}\frac{1}{\delta}\left[ 
\mcv(x,\cL_{\xi+\delta\eta})-\mcv(x,\cL_{\xi})
  \right] =\dxY_0=\dbE[D^{x,\xi}\cdot\eta].
\end{equation}
For some $\xi\in \dbL^2(\cF_0)$, we want to show that $D^{x,\xi}=f(\xi)$ for some Borel function $f$.
We will prove this by contradiction.
Assume that $D^{x,\xi}\neq \dbE[D^{x,\xi}|\xi]$, take 
$\eta_1= D^{x,\xi}-\dbE[D^{x,\xi}|\xi]$. 
Construct a random variable $\eta_2$ such that, conditional on $\xi$, 
$\eta_1$ and $\eta_2$ are independent and identically distributed (i.i.d.).
On the one hand,
\begin{equation}
  \delta Y_0^{x,\xi,\eta_1}=\dbE[D^{x,\xi}\cdot\eta_1]=\dbE[(\eta_1)^2]>0.
\end{equation}
On the other hand,
\begin{equation}
  \begin{aligned}
\delta Y_0^{x,\xi,\eta_2}=&\dbE[D^{x,\xi}\cdot\eta_2]\\
=&\dbE[\eta_1\eta_2]+\dbE[\dbE[D^{x,\xi}|\xi]\eta_2]\\
=&\dbE[\dbE[\eta_1|\xi]\dbE[\eta_2|\xi]]+\dbE[\dbE[D^{x,\xi}|\xi]\dbE[\eta_2|\xi]]\\
=&0.
  \end{aligned}
\end{equation}
However, by the weak uniqueness of FBSDEs,
we know that $\delta Y_0^{x,\xi,\eta_1}=\delta Y_0^{x,\xi,\eta_2}$.
We therefore conclude that
\begin{equation}
  D^{x,\xi}= \dbE[D^{x,\xi}|\xi],\quad ~\mbox{a.s.}
\end{equation}

We already know that $D^{x,\xi}=f(\xi)$. 
Next, we will prove that $f$ depends only on the distribution of $\xi$. Select $\xi_1$
 and $\xi_2$
 such that $\cL_{\xi_1}=\cL_{\xi_2}=\mu$, and $D^{x,\xi_1}=f^1(\xi_1),D^{x,\xi_2}=f^2(\xi_2)$. 
Then, for any $g\in C_b(\dbR)$, take $\eta_1=g(\xi_1)$ and $\eta_2=g(\xi_2)$. 
We have $\delta Y_0^{x,\xi_1,\eta_1}=\delta Y_0^{x,\xi_2,\eta_2}$, which implies
\begin{equation}
  \int_\dbR f^1(x)g(x)\mu(\rmd x)=\int_\dbR f^2(x)g(x)\mu(\rmd x), \quad \forall g\in C_b(\dbR).
\end{equation}
Then we know
\begin{equation}
  f^1(x)=f^2(x),\quad ~\mbox{$\mu$-a.e.}
\end{equation}

Now we can conclude that, for any $(x,\mu)\in \dbR\times \cP_2$, there exists
a function $f^{x,\mu}(\tilde{x})$ defined on the support of $\mu$, such that
$D^{x,\xi}=f^{x,\mu}(\xi)$ for all $\xi\in \dbL^2(\cF_0;\mu)$.
Then we can define
\begin{equation}
  \pa_{\mu}\mcv(x,\mu,\tilde{x})\triangleq f^{x,\mu}(\tilde{x})
\end{equation}
for some version of $f^{x,\mu}$. This satisfies the conditions of the theorem.
\qed

It should be noted that although we have proved that $\pa_{\mu}\mcv(x,\mu,\tilde{x})$ satisfies 
(\ref{eq: mcv dm}), this only demonstrates 
the differentiability of $\mcv$ with respect to the measure in the sense of directional derivative, which is 
still far from constituting a  Lions-derivative.
Similarly, for the value function 
\begin{equation}
  V(x,\mu)=\dbE\bigg[ \int_0^\infty \rme^{-rt} 
F(X^{x,\xi}_t,\cL_{X_t^{\xi}},Y_t^{x,\xi})  \rmd t \bigg]
\end{equation} 
where $F(x,\mu,y)\triangleq f(x,\mu,\hat{\alpha}(x,\mu,y))=H(x,\mu,y)-y\pa_yH(x,\mu,y)$,
it can also be shown that:
\begin{equation}
  \begin{aligned}
    \lim_{\delta\to 0} \frac{1}{\delta} \left[ 
      V(x, \cL_{\xi + \delta \eta}) - V(x, \cL_{\xi})
    \right] 
    = &\dbE\bigg[ \int_0^\infty \rme^{-rt} \biggl(
        \pa_x F(X^{x,\xi}_t, \cL_{X_t^{\xi}}, Y_t^{x,\xi}) \cdot \dxX_t \\
        &\quad + \pa_y F(X^{x,\xi}_t, \cL_{X_t^{\xi}}, Y_t^{x,\xi}) \cdot \dxY_t \\
        &\quad + \tilde{\dbE}_{\cF_t} \left[
            \pa_\mu F(X^{x,\xi}_t, \cL_{X_t^{\xi}}, Y_t^{x,\xi}, \tilde{X}_{t}^{\xi}) \cdot \tdX_t
          \right]
      \biggr) \rmd t \bigg].
  \end{aligned}
\end{equation}
By the same technique as before, one can prove the existence of a function 
$\pa_\mu V(x,\mu,\tilde{x})$ satisfying
\begin{equation}
  \lim_{\delta\to 0} \frac{1}{\delta} \left[ 
      V(x, \cL_{\xi + \delta \eta}) - V(x, \cL_{\xi})
    \right] 
    = \dbE[\pa_{\mu}V(x,\cL_{\xi},\xi)\cdot\eta]
\end{equation}
for all $\xi,\eta\in\dbL^2(\cF_0)$.

Next, we aim to select a sufficiently regular version of $\pa_\mu \mcv$ such that
$\pa_\mu \mcv(x, \cL_{\xi},\xi)$ is continuous with respect to $\xi$. 
This allows us to strengthen the G\^{a}teaux derivative to the Fr\'echet derivative, thereby
\begin{equation}
      \mcv(x, \cL_{\xi + \eta}) - \mcv(x, \cL_{\xi})
=\dbE[\pa_{\mu}\mcv(x,\cL_{\xi},\xi)\cdot\eta] +o(\left\lVert \eta\right\rVert_{\dbL^2} ).
\end{equation}
This demonstrates that $\pa_\mu \mcv $ is the Lions-derivative of $\mcv$.

\subsection{Representation of the Lions-derivative}
First we consider the case that $\xi$ is discrete:
\begin{equation}
  \label{eq: discrete conditon}
  p_i=\dbP(\xi=x_i),\quad i=1,2,\cdots,n.
\end{equation} 
Then we have
\begin{equation}
  \label{eq: d repre}
  \delta Y_0^{x,\xi,I_{\{\xi=x_i\}}}=\dbE[\pa_{\mu}\mcv(x,\cL_{\xi},\xi)\cdot I_{\{\xi=x_i\}}]
  =p_i\cdot \pa_\mu \mcv(x,\cL_{\xi},x_i).
\end{equation}
Thus, we can determine the values of $\pa_\mu \mcv(x,\cL_{\xi},\cdot)$ on the support of $\cL_{\xi}$.

To further investigate $\delta Y_0^{x,\xi,I_{\{\xi=x_i\}}}$, 
we introduce the following FBSDEs:
\begin{equation}
\label{eq: d part}
\begin{cases}
  \begin{aligned}
    \rmd \gXi_t = & \left\{ 
        \gXi_t \partial_{xy}H(X_{t}^{x_i,\xi},\mathcal{L}_{X_{t}^{\xi}},Y_{t}^{x_i,\xi}) 
        + \gYi_t \partial_{yy}H(X_{t}^{x_i,\xi},\mathcal{L}_{X_{t}^{\xi}},Y_{t}^{x_i,\xi}) 
        \right. \\
                  & + p_i\tilde{\dbE}_{\cF_t} \left[ 
        \tgXi_t \partial_{y\mu}H(X_{t}^{x_i,\xi},\mathcal{L}_{X_{t}^{\xi}},Y_{t}^{x_i,\xi},\tilde{X}_{t}^{x_i,\xi}) 
        \right. \\
                  & \left. \left. 
        + \tgXii_t \partial_{y\mu}H(X_{t}^{x_i,\xi},\mathcal{L}_{X_{t}^{\xi}},Y_{t}^{x_i,\xi},\tilde{X}_{t}^{\xi})
        \right] 
        \right\} \rmd t,
  \end{aligned}\\
  \begin{aligned}
    \rmd \gXii_t = & \left\{ 
        \gXii_t \partial_{xy}H(X_{t}^{\xi},\mathcal{L}_{X_{t}^{\xi}},Y_{t}^{\xi}) 
        + \gYii_t \partial_{yy}H(X_{t}^{\xi},\mathcal{L}_{X_{t}^{\xi}},Y_{t}^{\xi}) 
        \right. \\
                  & + \tilde{\dbE}_{\cF_t} \left[ 
        \tgXi_t \partial_{y\mu}H(X_{t}^{\xi},\mathcal{L}_{X_{t}^{\xi}},Y_{t}^{\xi},\tilde{X}_{t}^{x_i,\xi}) 
        \right. \\
                  & \left. \left. 
        + \tgXii_t \partial_{y\mu}H(X_{t}^{\xi},\mathcal{L}_{X_{t}^{\xi}},Y_{t}^{\xi},\tilde{X}_{t}^{\xi})
        \right] \cdot I_{\{\xi\neq x_i\}}
        \right\} \rmd t,
  \end{aligned}\\
  \begin{aligned}
    \rmd \gYi_t =- & \left\{ 
        \gXi_t \partial_{xx}H(X_{t}^{x_i,\xi},\mathcal{L}_{X_{t}^{\xi}},Y_{t}^{x_i,\xi}) 
        + \gYi_t \partial_{xy}H(X_{t}^{x_i,\xi},\mathcal{L}_{X_{t}^{\xi}},Y_{t}^{x_i,\xi}) 
        \right. \\
                  &-r\gYi_t + p_i\tilde{\dbE}_{\cF_t} \left[ 
        \tgXi_t \partial_{x\mu}H(X_{t}^{x_i,\xi},\mathcal{L}_{X_{t}^{\xi}},Y_{t}^{x_i,\xi},\tilde{X}_{t}^{x_i,\xi}) 
        \right. \\
                  & \left. \left. 
        + \tgXii_t \partial_{x\mu}H(X_{t}^{x_i,\xi},\mathcal{L}_{X_{t}^{\xi}},Y_{t}^{x_i,\xi},\tilde{X}_{t}^{\xi})
        \right] 
        \right\} \rmd t\\
        +& \gZi_t \rmd B_t ,
  \end{aligned}\\
    \begin{aligned}
    \rmd \gYii_t = -& \left\{ 
        \gXii_t \partial_{xx}H(X_{t}^{\xi},\mathcal{L}_{X_{t}^{\xi}},Y_{t}^{\xi}) 
        + \gYii_t \partial_{xy}H(X_{t}^{\xi},\mathcal{L}_{X_{t}^{\xi}},Y_{t}^{\xi}) 
        \right. \\
                  &-r\gYii_t + \tilde{\dbE}_{\cF_t} \left[ 
        \tgXi_t \partial_{x\mu}H(X_{t}^{\xi},\mathcal{L}_{X_{t}^{\xi}},Y_{t}^{\xi},\tilde{X}_{t}^{x_i,\xi}) 
        \right. \\
                  & \left. \left. 
        + \tgXii_t \partial_{x\mu}H(X_{t}^{\xi},\mathcal{L}_{X_{t}^{\xi}},Y_{t}^{\xi},\tilde{X}_{t}^{\xi})
        \right] \cdot I_{\{\xi\neq x_i\}}
        \right\} \rmd t\\
        +& \gZii_t \rmd B_t,
  \end{aligned}\\
  \gXi_0=1,\quad \gXii_0=0.
\end{cases}
\end{equation}

\begin{equation}
\label{eq: d total}
\begin{cases}
    \begin{aligned}
    \rmd \gXxi_t = & \left\{ 
        \gXxi_t \partial_{xy}H(X_{t}^{x,\xi},\mathcal{L}_{X_{t}^{\xi}},Y_{t}^{x,\xi}) 
        + \gYxi_t \partial_{yy}H(X_{t}^{x,\xi},\mathcal{L}_{X_{t}^{\xi}},Y_{t}^{x,\xi}) 
        \right. \\
                  & + \tilde{\dbE}_{\cF_t} \left[ 
        \tgXi \partial_{y\mu}H(X_{t}^{x,\xi},\mathcal{L}_{X_{t}^{\xi}},Y_{t}^{x,\xi},\tilde{X}_{t}^{x_i,\xi}) 
        \right. \\
                  & \left. \left. 
        + \tgXii \partial_{y\mu}H(X_{t}^{x,\xi},\mathcal{L}_{X_{t}^{\xi}},Y_{t}^{x,\xi},\tilde{X}_{t}^{\xi})
        \right] 
        \right\} \rmd t,
  \end{aligned}\\
  \begin{aligned}
    \rmd \gYxi_t =- & \left\{ 
        \gXxi_t \partial_{xx}H(X_{t}^{x,\xi},\mathcal{L}_{X_{t}^{\xi}},Y_{t}^{x,\xi}) 
        + \gYxi_t \partial_{xy}H(X_{t}^{x,\xi},\mathcal{L}_{X_{t}^{\xi}},Y_{t}^{x,\xi}) 
        \right. \\
                  &-r\gYxi_t + \tilde{\dbE}_{\cF_t} \left[ 
        \tgXi_t \partial_{x\mu}H(X_{t}^{x,\xi},\mathcal{L}_{X_{t}^{\xi}},Y_{t}^{x,\xi},\tilde{X}_{t}^{x_i,\xi}) 
        \right. \\
                  & \left. \left. 
        + \tgXii_t \partial_{x\mu}H(X_{t}^{x,\xi},\mathcal{L}_{X_{t}^{\xi}},Y_{t}^{x,\xi},\tilde{X}_{t}^{\xi})
        \right] 
        \right\} \rmd t\\
        +& \gZxi_t \rmd B_t ,
  \end{aligned}\\
  \gXxi_0=0.
\end{cases}
\end{equation}

\begin{thm}
  Under Assumption \ref{assum: Hb} and the condition that $\xi\in \dbL^2(\cF_0)$ 
  is a discrete random variable satisfying (\ref{eq: discrete conditon}),
  both FBSDEs (\ref{eq: d part}) and (\ref{eq: d total})
  admit unique solutions in $L^2_r$. And we have the relationship:
\begin{equation}
  \delta Y_t^{x,\xi,I_{\{\xi=x_i\}}}=p_i\cdot \nabla_{\mu}Y^{x,\xi,x_i}_t.
\end{equation}
This means
\begin{equation}
  \pa_\mu \mcv(x,\cL_{\xi},x_i)=\gYxi_0.
\end{equation}
\end{thm}
\proof
For the FBSDE (\ref{eq: d part}), we denote 
\begin{equation}
  \mathbf{X}\triangleq (\gXi,\gXii)^T,\quad \mathbf{Y}\triangleq (\gYi,\gYii)^T,
\end{equation}
thereby reformulating it as a two-dimensional linear FBSDE 
in terms of the pair $(\mathbf{X},\mathbf{Y})$. Under Assumption \ref{assum: Hb}, it can be easily shown that this FBSDE admits a 
unique solution $(\mathbf{X},\mathbf{Y})$ in $L_r^2([0,\infty);\dbR^2)$.
Next, we attempt to verify the following relation:
\begin{equation}
  \label{eq: relation1}
  \delta \Phi^{\xi,I_{\{\xi=x_i\}}}=
  \nabla \Phi^{\xi,x_i}I_{\{\xi=x_i\}}+p_i \nabla \Phi^{\xi,x_i,\star},\quad \Phi\in\{X,Y,Z\}.
\end{equation}
Since the solutions to all equations have been explicitly obtained and all the equations are linear, 
we can verify by direct substitution.
Noticing that
\begin{equation}
  \Phi^{\xi}=\sum_{i=1}^{n} \Phi^{x_i,\xi}I_{\{\xi=x_i\}},\quad \Phi\in\{X,Y,Z\},
\end{equation}
we have
\begin{equation}
\begin{aligned}
 & \gXi_t I_{\{\xi=x_i\}} \cdot\partial_{xy}H(X_{t}^{x_i,\xi},\mathcal{L}_{X_{t}^{\xi}},Y_{t}^{x_i,\xi}) +
  p_i\gXii_t \partial_{xy}H(X_{t}^{\xi},\mathcal{L}_{X_{t}^{\xi}},Y_{t}^{\xi}) \\
=&(\gXi_t I_{\{\xi=x_i\}}+p_i\gXii_t)  \partial_{xy}H(X_{t}^{\xi},\mathcal{L}_{X_{t}^{\xi}},Y_{t}^{\xi}).
\end{aligned}
\end{equation}
Moreover, the processes $(\gXi,\gYi,\gZi)$ are independent of $\xi$, then we have
\begin{equation}
\begin{aligned}
& I_{\{\xi=x_i\}}\cdot p_i\tilde{\dbE}_{\cF_t} \left[ 
        \tgXi_t \partial_{y\mu}H(X_{t}^{x_i,\xi},\mathcal{L}_{X_{t}^{\xi}},Y_{t}^{x_i,\xi},\tilde{X}_{t}^{x_i,\xi})  
        + \tgXii_t \partial_{y\mu}H(X_{t}^{x_i,\xi},\mathcal{L}_{X_{t}^{\xi}},Y_{t}^{x_i,\xi},\tilde{X}_{t}^{\xi})
        \right]\\
=&I_{\{\xi=x_i\}}\cdot
\tilde{\dbE}_{\cF_t} \left[ 
        (\tgXi_t I_{\{\tilde{\xi}=x_i\}}+p_i\tgXii_t)
        \partial_{y\mu}H(X_{t}^{\xi},\mathcal{L}_{X_{t}^{\xi}},Y_{t}^{\xi},\tilde{X}_{t}^{\xi})  
        \right]
\end{aligned}
\end{equation}
and
\begin{equation}
  \begin{aligned}
&p_i\tilde{\dbE}_{\cF_t} \left[ 
        \tgXi_t \partial_{y\mu}H(X_{t}^{\xi},\mathcal{L}_{X_{t}^{\xi}},Y_{t}^{\xi},\tilde{X}_{t}^{x_i,\xi}) 
        + \tgXii_t \partial_{y\mu}H(X_{t}^{\xi},\mathcal{L}_{X_{t}^{\xi}},Y_{t}^{\xi},\tilde{X}_{t}^{\xi})
        \right]\cdot I_{\{\xi\neq x_i\}}\\
=&\tilde{\dbE}_{\cF_t} \left[ 
        (\tgXi_t I_{\{\tilde{\xi}=x_i\}}+p_i\tgXii_t) \partial_{y\mu}H(X_{t}^{\xi},\mathcal{L}_{X_{t}^{\xi}},Y_{t}^{\xi},\tilde{X}_{t}^{\xi}) 
        \right]\cdot I_{\{\xi\neq x_i\}}.
\end{aligned}
\end{equation}
Combining all the relations above, we can verify Equation (\ref{eq: relation1}).

By substituting the solution of FBSDE (\ref{eq: d part}) into FBSDE (\ref{eq: d total}), 
it is straightforward to conclude that FBSDE (\ref{eq: d total}) admits a unique $L_r^2$ solution.
Observing that
\begin{equation}
\begin{aligned}
&\tilde{\dbE}_{\cF_t} \left[ 
      \delta\tilde{X}_t^{\xi,I_{\{\xi=x_i\}}} \partial_{y\mu}H(X_{t}^{x,\xi},\mathcal{L}_{X_{t}^{\xi}},Y_{t}^{x,\xi},\tilde{X}_{t}^{\xi}) 
        \right] \\
=&  \tilde{\dbE}_{\cF_t} \left[ 
        \tgXi I_{\{\tilde{\xi}=x_i\}}\cdot \partial_{y\mu}H(X_{t}^{x,\xi},\mathcal{L}_{X_{t}^{\xi}},Y_{t}^{x,\xi},\tilde{X}_{t}^{x_i,\xi}) 
        + p_i \tgXii \partial_{y\mu}H(X_{t}^{x,\xi},\mathcal{L}_{X_{t}^{\xi}},Y_{t}^{x,\xi},\tilde{X}_{t}^{\xi})
        \right] \\
=& p_i \tilde{\dbE}_{\cF_t} \left[ 
        \tgXi \partial_{y\mu}H(X_{t}^{x,\xi},\mathcal{L}_{X_{t}^{\xi}},Y_{t}^{x,\xi},\tilde{X}_{t}^{x_i,\xi}) 
        + \tgXii \partial_{y\mu}H(X_{t}^{x,\xi},\mathcal{L}_{X_{t}^{\xi}},Y_{t}^{x,\xi},\tilde{X}_{t}^{\xi})
        \right], 
\end{aligned}
\end{equation}
we obtain 
\begin{equation}
  \delta\Phi^{x,\xi,I_{\{\xi=x_i\}}}=p_i\nabla_\mu \Phi^{x,\xi,x_i},\quad \Phi\in\{X,Y,Z\}.
\end{equation}
Particularly,
\begin{equation}
  \delta Y_0^{x,\xi,I_{\{\xi=x_i\}}}=p_i\cdot \nabla_{\mu}Y^{x,\xi,x_i}_0.
\end{equation}
Combined with (\ref{eq: d repre}), we can get that:
\begin{equation}
  \pa_\mu \mcv(x,\cL_{\xi},x_i)=\gYxi_0.
\end{equation}

\qed

Next, we approximate the absolutely continuous distribution by the discrete ones, 
thereby obtaining a version of $\pa_\mu \mcv(x,\mu,\cdot)$ when $\mu$ is absolutely continuous. 
To this end, for any $(x,\xi,\tilde{x})\in \dbR\times \dbL^2(\cF_0)\times \dbR$, 
we introduce the following three FBSDEs:
\begin{equation}
\label{eq: g pa1}
\begin{cases}
\mathrm{d}  \gxX_t = \left[ \gxX_t \partial_{xy}H(X_{t}^{x,\xi},\mathcal{L} _{X_{t}^{\xi}},Y_{t}^{x, \xi}) + \gxY_t \partial_{yy}H(X_{t}^{x,\xi},\mathcal{L} _{X_{t}^{\xi}},Y_{t}^{x,\xi}) \right] \mathrm{d} t, \\
\begin{aligned}
  \mathrm{d} \gxY_t = & -\left[ \gxX_t \partial_{xx} {H} (X_{t}^{x,\xi},\mathcal{L} _{X_{t}^{\xi}}, Y_{t}^{x,\xi}) + \gxY_t \partial_{xy} {H} (X_{t}^{x,\xi},\mathcal{L} _{X_{t}^{\xi}}, Y_{t}^{x,\xi}) - r\gxY_t \right] \mathrm{d} t \\
  & + \gxZ_t \mathrm{d} B_t,
\end{aligned} \\
\gxX_0 = 1.
\end{cases}
\end{equation}

\begin{equation}
\label{eq: g pa2}
\begin{cases}
    \begin{aligned}
    \rmd \gXx_t = & \left\{ 
        \gXx_t \partial_{xy}H(X_{t}^{\xi},\mathcal{L}_{X_{t}^{\xi}},Y_{t}^{\xi}) 
        + \gYx_t \partial_{yy}H(X_{t}^{\xi},\mathcal{L}_{X_{t}^{\xi}},Y_{t}^{\xi}) 
        \right. \\
                  & + \tilde{\dbE}_{\cF_t} \left[ 
        \tgxX_t \partial_{y\mu}H(X_{t}^{\xi},\mathcal{L}_{X_{t}^{\xi}},Y_{t}^{\xi},\tilde{X}_{t}^{x,\xi}) 
        \right. \\
                & \left. \left. 
        + \tgXx_t \partial_{y\mu}H(X_{t}^{\xi},\mathcal{L}_{X_{t}^{\xi}},Y_{t}^{\xi},\tilde{X}_{t}^{\xi})
        \right] 
        \right\} \rmd t,
  \end{aligned}\\
      \begin{aligned}
    \rmd \gYx_t = -& \left\{ 
        \gXx_t \partial_{xx}H(X_{t}^{\xi},\mathcal{L}_{X_{t}^{\xi}},Y_{t}^{\xi}) 
        + \gYx_t \partial_{xy}H(X_{t}^{\xi},\mathcal{L}_{X_{t}^{\xi}},Y_{t}^{\xi}) 
        \right. \\
                  &-r\gYx_t + \tilde{\dbE}_{\cF_t} \left[ 
        \tgxX_t \partial_{x\mu}H(X_{t}^{\xi},\mathcal{L}_{X_{t}^{\xi}},Y_{t}^{\xi},\tilde{X}_{t}^{x,\xi}) 
        \right. \\
                  & \left. \left. 
        + \tgXx_t \partial_{x\mu}H(X_{t}^{\xi},\mathcal{L}_{X_{t}^{\xi}},Y_{t}^{\xi},\tilde{X}_{t}^{\xi})
        \right] 
        \right\} \rmd t\\
        +& \gZx_t \rmd B_t,
  \end{aligned}\\
  \gXx_0=0.
\end{cases} 
\end{equation}

\begin{equation}
\label{eq: g total}
\begin{cases}
    \begin{aligned}
    \rmd \gXxx_t = & \left\{ 
        \gXxx_t \partial_{xy}H(X_{t}^{x,\xi},\mathcal{L}_{X_{t}^{\xi}},Y_{t}^{x,\xi}) 
        + \gYxx_t \partial_{yy}H(X_{t}^{x,\xi},\mathcal{L}_{X_{t}^{\xi}},Y_{t}^{x,\xi}) 
        \right. \\
                  & + \tilde{\dbE}_{\cF_t} \left[ 
        \tgtxX_t \partial_{y\mu}H(X_{t}^{x,\xi},\mathcal{L}_{X_{t}^{\xi}},Y_{t}^{x,\xi},\tilde{X}_{t}^{x,\xi}) 
        \right. \\
                  & \left. \left. 
        + \tgXtx_t \partial_{y\mu}H(X_{t}^{x,\xi},\mathcal{L}_{X_{t}^{\xi}},Y_{t}^{x,\xi},\tilde{X}_{t}^{\xi})
        \right] 
        \right\} \rmd t,
  \end{aligned}\\
  \begin{aligned}
    \rmd \gYxx_t =- & \left\{ 
        \gXxx_t \partial_{xx}H(X_{t}^{x,\xi},\mathcal{L}_{X_{t}^{\xi}},Y_{t}^{x,\xi}) 
        + \gYxx_t \partial_{xy}H(X_{t}^{x,\xi},\mathcal{L}_{X_{t}^{\xi}},Y_{t}^{x,\xi}) 
        \right. \\
                  &-r\gYxx_t + \tilde{\dbE}_{\cF_t} \left[ 
        \tgtxX_t \partial_{x\mu}H(X_{t}^{x,\xi},\mathcal{L}_{X_{t}^{\xi}},Y_{t}^{x,\xi},\tilde{X}_{t}^{x,\xi}) 
        \right. \\
                  & \left. \left. 
        + \tgXtx_t \partial_{x\mu}H(X_{t}^{x,\xi},\mathcal{L}_{X_{t}^{\xi}},Y_{t}^{x,\xi},\tilde{X}_{t}^{\xi})
        \right] 
        \right\} \rmd t\\
        +& \gZxx_t \rmd B_t ,
  \end{aligned}\\
  \gXxx_0=0.
\end{cases}
\end{equation}

In the direct source terms of \eqref{eq: g total}, the copied
state paired with \(\nabla\widetilde X^{\tilde x,\xi}\) is
\(\widetilde X^{\tilde x,\xi}\); the superscript \(x\) printed in
the last argument of \(\partial_{q\mu}H\), \(q=x,y\), is understood
as the measure label \(\tilde x\).

\begin{lem}
  \label{lem: b and c}
  For any $(x, \xi, \tilde{x}) \in \mathbb{R} \times \dbL^2(\mathcal{F}_0) \times \mathbb{R}$, 
  the FBSDEs \eqref{eq: g pa1}, \eqref{eq: g pa2}, and \eqref{eq: g total} admit unique solutions in $L_r^2$. 
  Moreover, the following mappings are uniformly bounded and jointly continuous:
  \begin{enumerate}
    \item $(x, \xi) \to ( \gxX, \gxY , \gxY_0)$ for the solution of \eqref{eq: g pa1},
    \item $(x, \xi) \to ( \gXx ,  \gYx , \gYx_0)$ for the solution of \eqref{eq: g pa2},
    \item $(x, \xi, \tilde{x}) \to ( \gXxx ,  \gYxx , \gYxx_0)$ for the solution of \eqref{eq: g total}.
  \end{enumerate}
\end{lem}

\proof
Under Assumption \ref{assum: Hb}, the existence and uniqueness of solutions to these three FBSDEs is straightforward. 
Moreover, the uniform boundedness of their solutions can be readily established by following a 
reasoning similar to the proof of Lemma \ref{lem: blf},
and the continuity of the solutions with respect to their initial values can be established by 
following a proof analogous to that in Corollary \ref{cor: cont}.
\qed

\begin{rem}
If we define $\psi(x,\mu,\tilde{x})\triangleq \gYxx_0$ for some $\xi\in\dbL^2(\cF_0;\mu)$,
then $\psi$ is uniformly bounded and jointly continuous on $\dbR\times\cP_2\times\dbR$.
\end{rem}

Now, we consider the case that $\mu$ is absolutely continuous and choose some $\xi\in \dbL^2(\cF_0;\mu)$.
For each $n\geq 1$, set
\begin{equation}
  x^n_i\triangleq \frac{i}{n},\quad \triangle^n_i\triangleq [\frac{i}{n},\frac{i+1}{n}),\quad i\in \dbZ.
\end{equation}
For any $x\in\dbR$, let $i^n(x)$ be the index function such that $x\in\triangle^n_{i^n(x)}$.
Denote
\begin{equation}
  \label{eq: de xin}
  \xi_n\triangleq \sum_{i=-n^2}^{n^2-1} x^n_i\cdot I_{\triangle^n_i}(\xi)-n^2I_{(-\infty,-n^2)}(\xi)
  +n^2I_{[n^2,\infty)}(\xi).
\end{equation}
It's clear that $\lim_{n\to\infty}\xi_n=\xi$, a.s., $\lim_{n\to\infty}\dbE[|\xi_n-\xi|^2]=0$,
and thus $\lim_{n\to\infty}\cW_2(\cL_{\xi_n},\cL_{\xi})=0$.
For any $\eta\in \dbL^2(\cF_0)$, by stability of FBSDE (\ref{eq: dxX}), we have that
\begin{equation}
  \dbE[\pa_\mu\mcv(x,\mu,\xi)\eta]=\dxY_0=\lim_{n\to\infty} \delta Y^{x,\xi_n,\eta}_0.
\end{equation}

The following lemma reveals the limiting relationship between FBSDEs (\ref{eq: d part})-(\ref{eq: d total}) 
and FBSDEs (\ref{eq: g pa1})-(\ref{eq: g total}), 
which provides a tool for our further study of $\delta Y^{x,\xi_n,\eta}_0$. 
Its proof differs slightly from that of Lemma \ref{lem: b and c}; 
hence, we provide a detailed demonstration here.

\begin{lem}
For each $(x, \xi, \tilde{x}) \in \mathbb{R} \times \dbL^2(\mathcal{F}_0) \times \mathbb{R}$
such that $\cL_\xi$ is absolutely continuous, and let $\{\xi_n; n \geq 1\}$ be defined as in (\ref{eq: de xin}).
We have that, as $n \to \infty$, the following limits hold (the convergence of the stochastic process is in the $L_r^2$ sense):
\begin{align}
& \text{(1)}\quad ( \nabla X^{\xi_n,\frac{i^n(\tilde{x})}{n}} , \nabla Y^{\xi_n,\frac{i^n(\tilde{x})}{n}} , \nabla Y^{\xi_n,\frac{i^n(\tilde{x})}{n}}_0)
  \to (\nabla X^{\tilde{x},\xi}, \nabla Y^{\tilde{x},\xi},\nabla Y^{\tilde{x},\xi}_0 ); \\
& \text{(2)}\quad ( \nabla X^{\xi_n,\frac{i^n(\tilde{x})}{n},\star} , \nabla Y^{\xi_n,\frac{i^n(\tilde{x})}{n},\star} , \nabla Y^{\xi_n,\frac{i^n(\tilde{x})}{n},\star}_0)
  \to (\nabla X^{\xi,\tilde{x}}, \nabla Y^{\xi,\tilde{x}},\nabla Y^{\xi,\tilde{x}}_0 ); \\
& \text{(3)}\quad ( \nabla_\mu X^{x,\xi_n,\frac{i^n(\tilde{x})}{n}} , \nabla_\mu Y^{x,\xi_n,\frac{i^n(\tilde{x})}{n}} , \nabla_\mu Y^{x,\xi_n,\frac{i^n(\tilde{x})}{n}}_0)
  \to (\nabla_\mu X^{x,\xi,\tilde{x}}, \nabla Y^{x,\xi,\tilde{x}},\nabla Y^{x,\xi,\tilde{x}}_0 ).
\end{align}
\end{lem}

\proof
For each $\tilde{x}\in\dbR$, we have $(\cL_{\xi_n},i^n(\tilde{x})/n)\to (\cL_\xi, \tilde{x})$ as $n\to\infty$,
then we have 
\begin{equation}
  (X^{\xi_n},Y^{\xi_n})\to (X^{\xi},Y^{\xi}),\quad (X^{\frac{i^n(\tilde{x})}{n},\xi_n},Y^{\frac{i^n(\tilde{x})}{n},\xi_n})
  \to (X^{\tilde{x},\xi},Y^{\tilde{x},\xi})
\end{equation}
in $L_r^2$.
An essential observation is that the solutions of FBSDEs \eqref{eq: d part}-\eqref{eq: d total} and \eqref{eq: g pa1}-\eqref{eq: g total} 
are uniformly bounded in $L_r^2$
for all initial values, which is important to the subsequent proof.
Denote
\begin{equation}
  X^n=\nabla X^{\xi_n,\frac{i^n(\tilde{x})}{n}}-\nabla X^{\tilde{x},\xi},\quad 
  Y^n=\nabla Y^{\xi_n,\frac{i^n(\tilde{x})}{n}}-\nabla X^{\tilde{x},\xi},\quad 
  Z^n=\nabla Z^{\xi_n,\frac{i^n(\tilde{x})}{n}}-\nabla X^{\tilde{x},\xi},
\end{equation}
then they satisfy:
\begin{equation}
\begin{cases}
  \begin{aligned}
    \rmd X^n_t = & \left\{ 
        X^n_t \partial_{xy}H(X_{t}^{\frac{i^n(\tilde{x})}{n},\xi_n},\mathcal{L}_{X_{t}^{\xi_n}},Y_{t}^{\frac{i^n(\tilde{x})}{n},\xi_n}) 
        + Y^n_t \partial_{yy}H(X_{t}^{\frac{i^n(\tilde{x})}{n},\xi_n},\mathcal{L}_{X_{t}^{\xi_n}},Y_{t}^{\frac{i^n(\tilde{x})}{n},\xi_n})  
        \right. \\
        &+X^{\tilde{x},\xi}_t H_{xy}^n+ Y^{\tilde{x},\xi}_t H_{yy}^n
        \\
                  & + p^n\tilde{\dbE}_{\cF_t} \left[ 
        \nabla \tilde{X}^{\xi_n,\frac{i^n(\tilde{x})}{n}} _t \partial_{y\mu}H(X_{t}^{\frac{i^n(\tilde{x})}{n},\xi_n},\mathcal{L}_{X_{t}^{\xi_n}},Y_{t}^{\frac{i^n(\tilde{x})}{n},\xi_n},\tilde{X}_{t}^{\frac{i^n(\tilde{x})}{n},\xi_n}) 
        \right. \\
         &          \left. \left. 
        + \nabla \tilde{X}^{\xi_n,\frac{i^n(\tilde{x})}{n},\star}_t \partial_{y\mu}H(X_{t}^{\frac{i^n(\tilde{x})}{n},\xi_n},\mathcal{L}_{X_{t}^{\xi_n}},Y_{t}^{\frac{i^n(\tilde{x})}{n},\xi_n},\tilde{X}_{t}^{\xi_n})
        \right] 
        \right\} \rmd t,
  \end{aligned}\\
\begin{aligned}
    \rmd Y^n_t = -& \left\{ 
        X^n_t \partial_{xx}H(X_{t}^{\frac{i^n(\tilde{x})}{n},\xi_n},\mathcal{L}_{X_{t}^{\xi_n}},Y_{t}^{\frac{i^n(\tilde{x})}{n},\xi_n}) 
        + Y^n_t \partial_{xy}H(X_{t}^{\frac{i^n(\tilde{x})}{n},\xi_n},\mathcal{L}_{X_{t}^{\xi_n}},Y_{t}^{\frac{i^n(\tilde{x})}{n},\xi_n}) 
        \right. \\
        &-rY^n_t + X^{\tilde{x},\xi}_t H_{xx}^n + Y^{\tilde{x},\xi}_t H_{xy}^n \\
                  & + p_i^n\tilde{\dbE}_{\cF_t} \left[ 
        \nabla \tilde{X}^{\xi_n,\frac{i^n(\tilde{x})}{n}} _t \partial_{x\mu}H(X_{t}^{\frac{i^n(\tilde{x})}{n},\xi_n},\mathcal{L}_{X_{t}^{\xi_n}},Y_{t}^{\frac{i^n(\tilde{x})}{n},\xi_n},\tilde{X}_{t}^{\frac{i^n(\tilde{x})}{n},\xi_n}) 
        \right. \\
                  & \left. \left. 
        + \nabla \tilde{X}^{\xi_n,\frac{i^n(\tilde{x})}{n},\star} _t \partial_{x\mu}H(X_{t}^{\frac{i^n(\tilde{x})}{n},\xi_n},\mathcal{L}_{X_{t}^{\xi_n}},Y_{t}^{\frac{i^n(\tilde{x})}{n},\xi_n},\tilde{X}_{t}^{\xi_n})
        \right] 
        \right\} \rmd t \\
        +& Z^n_t \rmd B_t ,
\end{aligned}\\
X^n_0=0,
\end{cases}
\end{equation}
where
\begin{equation}
  \begin{aligned}
    H_{xx}^n&\triangleq \partial_{xx}H(X_{t}^{\frac{i^n(\tilde{x})}{n},\xi_n},\mathcal{L}_{X_{t}^{\xi_n}},Y_{t}^{\frac{i^n(\tilde{x})}{n},\xi_n}) 
                      -\partial_{xx}H(X_{t}^{\tilde{x},\xi},\mathcal{L}_{X_{t}^{\xi}},Y_{t}^{\tilde{x},\xi}), \\
      H_{xy}^n&\triangleq \partial_{xy}H(X_{t}^{\frac{i^n(\tilde{x})}{n},\xi_n},\mathcal{L}_{X_{t}^{\xi_n}},Y_{t}^{\frac{i^n(\tilde{x})}{n},\xi_n}) 
                      -\partial_{xy}H(X_{t}^{\tilde{x},\xi},\mathcal{L}_{X_{t}^{\xi}},Y_{t}^{\tilde{x},\xi}), \\
      H_{yy}^n&\triangleq \partial_{yy}H(X_{t}^{\frac{i^n(\tilde{x})}{n},\xi_n},\mathcal{L}_{X_{t}^{\xi_n}},Y_{t}^{\frac{i^n(\tilde{x})}{n},\xi_n}) 
                      -\partial_{yy}H(X_{t}^{\tilde{x},\xi},\mathcal{L}_{X_{t}^{\xi}},Y_{t}^{\tilde{x},\xi}), \\
p^n&\triangleq \dbP(\xi_n=\frac{i^n(\tilde{x})}{n})=\dbP(\xi\in \triangle^n_{i^n(\tilde{x})}).
  \end{aligned}
\end{equation}
Applying It\^{o}'s formula to $\rme^{-rt}X^n_tY^n_t$, we get that
\begin{equation}
\begin{aligned}
0=&
\mathbb{E} \bigg[ \int_0^\infty \rme^{-rt} 
 \biggl( - 
       ( X^n_t)^2 \partial_{xx}H(X_{t}^{\frac{i^n(\tilde{x})}{n},\xi_n},\mathcal{L}_{X_{t}^{\xi_n}},Y_{t}^{\frac{i^n(\tilde{x})}{n},\xi_n})\\
 &\quad{}+ (Y^n_t)^2 \partial_{yy}H(X_{t}^{\frac{i^n(\tilde{x})}{n},\xi_n},\mathcal{L}_{X_{t}^{\xi_n}},Y_{t}^{\frac{i^n(\tilde{x})}{n},\xi_n})
 \\
&-X^n_t (X^{\tilde{x},\xi}_t H_{xx}^n + Y^{\tilde{x},\xi}_t H_{xy}^n   )
+Y^n_t(X^{\tilde{x},\xi}_t H_{xy}^n + Y^{\tilde{x},\xi}_t H_{yy}^n)\\
&-p^n X^n_t \tilde{\mathbb{E}}_{\mathcal{F}_t} \biggl[ 
        \nabla \tilde{X}^{\xi_n,\frac{i^n(\tilde{x})}{n}} _t \partial_{x\mu}H(X_{t}^{\frac{i^n(\tilde{x})}{n},\xi_n},\mathcal{L}_{X_{t}^{\xi_n}},Y_{t}^{\frac{i^n(\tilde{x})}{n},\xi_n},\tilde{X}_{t}^{\frac{i^n(\tilde{x})}{n},\xi_n}) 
        \\
&\quad + \nabla \tilde{X}^{\xi_n,\frac{i^n(\tilde{x})}{n},\star} _t \partial_{x\mu}H(X_{t}^{\frac{i^n(\tilde{x})}{n},\xi_n},\mathcal{L}_{X_{t}^{\xi_n}},Y_{t}^{\frac{i^n(\tilde{x})}{n},\xi_n},\tilde{X}_{t}^{\xi_n})
        \biggr] \\
&+p^n Y^n_t \tilde{\mathbb{E}}_{\mathcal{F}_t} \biggl[ 
        \nabla \tilde{X}^{\xi_n,\frac{i^n(\tilde{x})}{n}} _t \partial_{y\mu}H(X_{t}^{\frac{i^n(\tilde{x})}{n},\xi_n},\mathcal{L}_{X_{t}^{\xi_n}},Y_{t}^{\frac{i^n(\tilde{x})}{n},\xi_n},\tilde{X}_{t}^{\frac{i^n(\tilde{x})}{n},\xi_n}) 
        \\
&\quad + \nabla \tilde{X}^{\xi_n,\frac{i^n(\tilde{x})}{n},\star}_t \partial_{y\mu}H(X_{t}^{\frac{i^n(\tilde{x})}{n},\xi_n},\mathcal{L}_{X_{t}^{\xi_n}},Y_{t}^{\frac{i^n(\tilde{x})}{n},\xi_n},\tilde{X}_{t}^{\xi_n})
        \biggr] \biggr)
 \mathrm{d}t \bigg]. 
\end{aligned}
\end{equation}
Using the conditions in Assumption \ref{assum: Hb}, we have that
\begin{equation}
\begin{aligned}
  &\mathbb{E} \bigg[ \int_0^\infty \rme^{-rt} \biggl( ( X^n_t)^2 + (Y^n_t)^2 \biggr)\mathrm{d}t \bigg]\\
\leq& \frac{2}{r} \mathbb{E} \bigg[ \int_0^\infty \rme^{-rt} \biggl(
-X^n_t (X^{\tilde{x},\xi}_t H_{xx}^n + Y^{\tilde{x},\xi}_t H_{xy}^n   )
+Y^n_t(X^{\tilde{x},\xi}_t H_{xy}^n + Y^{\tilde{x},\xi}_t H_{yy}^n)\\
&\quad +p^n\lambda_2
( |X^n_t|+    |Y^n_t|) \tilde{\mathbb{E}}_{\mathcal{F}_t} \left[ 
        |\nabla \tilde{X}^{\xi_n,\frac{i^n(\tilde{x})}{n}} _t |+ |\nabla \tilde{X}^{\xi_n,\frac{i^n(\tilde{x})}{n},\star} _t|    \right]
\biggr)\mathrm{d}t \bigg].
\end{aligned}
\end{equation}
From the uniform boundedness of $X^n$ and $Y^n$ and the Dominated Convergence Theorem, we conclude that:
\begin{equation}
  \lim_{n\to\infty}
  \mathbb{E} \bigg[ \int_0^\infty \rme^{-rt} \biggl(
-X^n_t (X^{\tilde{x},\xi}_t H_{xx}^n + Y^{\tilde{x},\xi}_t H_{xy}^n   )
+Y^n_t(X^{\tilde{x},\xi}_t H_{xy}^n + Y^{\tilde{x},\xi}_t H_{yy}^n)
\biggr)\mathrm{d}t \bigg]=0.
\end{equation}
Since $\cL_\xi$ is absolutely continuous, we have $p^n\to 0$, and thus:
\begin{equation}
\lim_{n\to\infty}  \mathbb{E} \bigg[ \int_0^\infty \rme^{-rt} \biggl(p^n
( |X^n_t|+    |Y^n_t|) \tilde{\mathbb{E}}_{\mathcal{F}_t} \left[ 
        |\nabla \tilde{X}^{\xi_n,\frac{i^n(\tilde{x})}{n}} _t |+ |\nabla \tilde{X}^{\xi_n,\frac{i^n(\tilde{x})}{n},\star} _t|    \right]
\biggr)\mathrm{d}t \bigg]=0.
\end{equation}
By combining all the above relations, we arrive at the following convergence result:
\begin{equation}
\lim_{n\to\infty}  \mathbb{E} \bigg[ \int_0^\infty \rme^{-rt} \biggl( ( X^n_t)^2 + (Y^n_t)^2 \biggr)\mathrm{d}t \bigg]=0.
\end{equation}
Applying It\^{o}'s formula to  $\rme^{-rt}(Y^n_t)^2$, we can easily obtain 
$Y^n_0\to 0$, thus finishing the proof of (1). The proofs for (2) and (3) follow analogously.
\qed

The following theorem provides a bounded and continuous representation of  $\pa_\mu\mcv(x,\mu,\tilde{x})$.
This implies that $\mcv(x,\mu)$ is Lions-differentiable with respect to $\mu$.

\begin{thm}
\label{thm: Lions representation}
Define $\psi(x,\mu,\tilde{x})\triangleq \gYxx_0$ for some $\xi\in\dbL^2(\cF_0;\mu)$,
then for any $(x,\mu)\in\dbR\times \cP_2$, $\psi(x,\mu,\cdot)$ is a version of $\pa_\mu\mcv(x,\mu,\cdot)$.
\end{thm}
\proof
For any $(x,\mu)$, we first consider the case where $\mu$ is absolutely continuous.
Take $\xi\in \dbL^2(\cF_0;\mu)$ and construct $\{\xi_n;n\geq 1\}$ as in (\ref{eq: de xin}).
For any bounded and Lipschitz continuous function $\varphi$, 
we have $\{\varphi(\xi_n)\}$ converges to $\varphi(\xi)$ in $\dbL^2$.
So we can get that
\begin{equation}
\begin{aligned}
  \dbE[\pa_\mu \mcv(x,\mu,\xi)\varphi(\xi)]=&\delta Y_0^{x,\xi,\varphi(\xi)}
  =\lim_{n\to\infty} \delta Y_0^{x,\xi_n,\varphi(\xi_n)}\\
=&\lim_{n\to\infty}\sum_{i=-n^2}^{n^2}\varphi(x^n_i) \delta Y_0^{x,\xi_n,I_{\{\xi_n=x^n_i\}}}\\
=&\lim_{n\to\infty}\sum_{i=-n^2}^{n^2}\varphi(x^n_i) \nabla_\mu Y_0^{x,\xi_n,{x^n_i}}\dbP(\xi_n=x^n_i)\\
=&\lim_{n\to\infty}\dbE\left[\varphi(\xi_n(\omega)) \nabla_\mu Y_0^{x,\xi_n,\xi_n(\omega)}\right]\\
=&\dbE\left[\varphi(\xi(\omega)) \nabla_\mu Y_0^{x,\xi,\xi(\omega)} \right].
\end{aligned}
\end{equation}
The last equality holds due to the convergence of $\varphi(\xi_n(\omega)) \nabla_\mu Y_0^{x,\xi_n,\xi_n(\omega)}$ 
to $\varphi(\xi(\omega)) \nabla_\mu Y_0^{x,\xi,\xi(\omega)}  $ in probability and the Bounded Convergence Theorem.
Then we have
\begin{equation}
  \pa_\mu \mcv(x,\mu,\cdot)= \psi(x,\mu,\cdot),\quad ~\mbox{$\mu$-a.e.}
\end{equation}
whenever $\mu$ is absolutely continuous.

Now we consider the general case. Fix an arbitrary $(\mu,\xi)$. One can easily construct
$\{\xi_n;n\geq 1\}$ such that $\cL_{\xi_n}$ is absolutely continuous, and $\dbE[|\xi_n-\xi|^2]\to 0$.
Then we have
\begin{equation}
\begin{aligned}
  \dbE[\pa_\mu \mcv(x,\mu,\xi)\varphi(\xi)]=&\delta Y_0^{x,\xi,\varphi(\xi)}
  =\lim_{n\to\infty} \delta Y_0^{x,\xi_n,\varphi(\xi_n)}\\
=&\lim_{n\to\infty} \dbE\left[\psi(x,\cL_{\xi_n},\xi_n)\varphi(\xi_n)\right]\\
=&\dbE\left[\psi(x,\cL_{\xi},\xi)\varphi(\xi)\right].
\end{aligned}
\end{equation}
Thus, we finish the proof.
\qed

\section{Classical regularity of the value function's derivative}
\label{sec: classical}
\setcounter{equation}{0}

In this section we strengthen the assumptions on the Hamiltonian and prove
the two derivatives that are still missing from the classical regularity
of $\mcv=\partial_xV$: $\partial_{xx}\mcv=\partial_{xxx}V$ and
$\partial_{\tilde x}\partial_\mu\mcv$.  The existence of
$\partial_x\mcv=\partial_{xx}V$ and its representation by
$\nabla Y^{x,\xi}_0$ were proved in \cite{song2025infinite}.
The new spatial equation has quadratic source terms and therefore uses the
$L^4$ estimate from Lemma~\ref{prop: linear Lp}.  The label
derivative is obtained by differentiating \eqref{eq: g total}.

\subsection{The second spatial variation and
\texorpdfstring{$\partial_{xxx}V$}{the third spatial derivative}}

For the spatial variational equations, the proof of
Lemma~\ref{prop: linear Lp} gives the stability condition in the
following assumption.

\begin{assum}
\label{assum: Hc}
Assumption~\ref{assum: Hb} holds.  In addition:

\noindent (i) There are constants
\begin{equation}
  a\geq0,\qquad
 \frac r2<b_0\leq b_1,\qquad
 \frac r2<c_0\leq c_1,
\end{equation}
such that, for every $(x,\mu,y)$,
\begin{equation}
 |\partial_{xy}H(x,\mu,y)|\leq a,\qquad
 b_0\leq-\partial_{yy}H(x,\mu,y)\leq b_1,\qquad
 c_0\leq\partial_{xx}H(x,\mu,y)\leq c_1.
\end{equation}
With $\underline\pi$ defined by \eqref{eq: linear lower pi}, we assume
\begin{equation}
\label{eq: Hamiltonian stability margin}
 \delta_*\triangleq b_0\underline\pi-a>0.
\end{equation}
\noindent (ii) All derivatives $\partial_x^i\partial_y^jH$ with
$i+j=3$ exist, are bounded, and are globally Lipschitz in
$(x,\mu,y)$, where the measure variable is equipped with $\cW_2$.

\noindent (iii) For $q\in\{x,y\}$,
$\partial_{q\mu}H(x,\mu,y,\tilde x)$ is differentiable in its last
variable.  The derivatives
$\partial_{\tilde x}\partial_{q\mu}H(x,\mu,y,\tilde x)$ are bounded and
globally Lipschitz in all their variables.
\end{assum}

The constants below are independent of $(x,\mu,\tilde x)$.

\begin{lem}
\label{lem: first variation moments}
Under Assumption~\ref{assum: Hc}, for every $p\geq2$,
\begin{equation}
\label{eq: first variation moments}
 \|\gxX\|_{L_r^p}^p+\|\gxY\|_{L_r^p}^p
 +\|\gxZ\|_{\mathbb H_r^p}^p\leq C_p.
\end{equation}
The same estimate holds, uniformly for $0<|h|\leq1$, for
\begin{equation}
 \nabla X^{x,\xi,h}\triangleq\frac{X^{x+h,\xi}-X^{x,\xi}}h,\quad
 \nabla Y^{x,\xi,h}\triangleq\frac{Y^{x+h,\xi}-Y^{x,\xi}}h,\quad
 \nabla Z^{x,\xi,h}\triangleq\frac{Z^{x+h,\xi}-Z^{x,\xi}}h.
\end{equation}
Moreover,
\begin{equation}
\label{eq: first quotient L4}
 (\nabla X^{x,\xi,h},\nabla Y^{x,\xi,h})
 \xrightarrow{h\to0}(\gxX,\gxY)
 \quad\hbox{in }L_r^4\times L_r^4,
\end{equation}
and $(x,\xi)\mapsto(\gxX,\gxY)$ is continuous with values in
$L_r^4\times L_r^4$.
\end{lem}

\proof
Equation \eqref{eq: g pa1} is \eqref{eq: scalar linear p} with
\begin{equation}
 \begin{aligned}
 \alpha_t&=\partial_{xy}H(X_t^{x,\xi},\cL_{X_t^\xi},Y_t^{x,\xi}),\\
 \beta_t&=\partial_{yy}H(X_t^{x,\xi},\cL_{X_t^\xi},Y_t^{x,\xi}),\\
 \gamma_t&=\partial_{xx}H(X_t^{x,\xi},\cL_{X_t^\xi},Y_t^{x,\xi}),
 \qquad \gxX_0=1.
 \end{aligned}
\end{equation}
Thus \eqref{eq: first variation moments} follows from
Lemma~\ref{prop: linear Lp}.  By the mean-value formula,
$(\nabla X^{x,\xi,h},\nabla Y^{x,\xi,h},\nabla Z^{x,\xi,h})$
satisfies the same linear equation with the Hessian of $H$ averaged
along the segment joining $(X^{x,\xi},Y^{x,\xi})$ and
$(X^{x+h,\xi},Y^{x+h,\xi})$.  The bounds in Assumption
\ref{assum: Hc}(i) are preserved by averaging, so Lemma
\ref{prop: linear Lp} gives the uniform estimate.

The first differentiability result of \cite{song2025infinite} gives
convergence in $L_r^2$.  We spell out the interpolation step.  Put
\begin{equation*}
 \nu(\rmd t,\rmd\omega)
 \triangleq\rme^{-rt}\rmd t\,\dbP(\rmd\omega),
 \qquad
 \nu([0,\infty)\times\Omega)=\frac1r<\infty.
\end{equation*}
For $V_h\triangleq U_h-U$, Hölder's inequality with conjugate exponents
$3/2$ and $3$ gives
\begin{equation*}
\begin{aligned}
 \int |V_h|^4\,\rmd\nu
 &=\int (|V_h|^2)^{2/3}(|V_h|^8)^{1/3}\,\rmd\nu\\
 &\leq
 \left(\int |V_h|^2\,\rmd\nu\right)^{2/3}
 \left(\int |V_h|^8\,\rmd\nu\right)^{1/3}.
\end{aligned}
\end{equation*}
Taking fourth roots, and using the uniform estimate with $p=8$, yields
\begin{equation}
 \|U_h-U\|_{L_r^4}
 \leq\|U_h-U\|_{L_r^2}^{1/3}
      \|U_h-U\|_{L_r^8}^{2/3}\xrightarrow{h\to0}0
\end{equation}
for $U_h=\nabla X^{x,\xi,h},\nabla Y^{x,\xi,h}$ and
$U=\gxX,\gxY$, respectively.  Thus $L_r^8$ is only an auxiliary
uniform bound used to obtain the required $L_r^4$ convergence; no
eighth-moment assumption is imposed on $\xi$.  The $L_r^4$ continuity
follows in the same way from the $L_r^2$ continuity proved in
\cite{song2025infinite} and the uniform $L_r^8$ estimate. \qed

For fixed $x\in\dbR$ and $\xi\in\dbL^2(\cF_0)$, consider
\begin{equation}
\label{eq: spatial second variation}
\begin{cases}
\begin{aligned}
 \rmd\ggxX_t={}&\bigl[
 \partial_{xy}H(X_t^{x,\xi},\cL_{X_t^\xi},Y_t^{x,\xi})\ggxX_t
 +\partial_{yy}H(X_t^{x,\xi},\cL_{X_t^\xi},Y_t^{x,\xi})\ggxY_t\\
 &+\partial_{xxy}H(X_t^{x,\xi},\cL_{X_t^\xi},Y_t^{x,\xi})(\gxX_t)^2\\
 &+2\partial_{xyy}H(X_t^{x,\xi},\cL_{X_t^\xi},Y_t^{x,\xi})
       \gxX_t\gxY_t\\
 &+\partial_{yyy}H(X_t^{x,\xi},\cL_{X_t^\xi},Y_t^{x,\xi})
       (\gxY_t)^2\bigr]\rmd t,
\end{aligned}\\
\begin{aligned}
 \rmd\ggxY_t=-{}&\bigl[
 \partial_{xx}H(X_t^{x,\xi},\cL_{X_t^\xi},Y_t^{x,\xi})\ggxX_t
 +\partial_{xy}H(X_t^{x,\xi},\cL_{X_t^\xi},Y_t^{x,\xi})\ggxY_t
 -r\ggxY_t\\
 &+\partial_{xxx}H(X_t^{x,\xi},\cL_{X_t^\xi},Y_t^{x,\xi})(\gxX_t)^2\\
 &+2\partial_{xxy}H(X_t^{x,\xi},\cL_{X_t^\xi},Y_t^{x,\xi})
       \gxX_t\gxY_t\\
 &+\partial_{xyy}H(X_t^{x,\xi},\cL_{X_t^\xi},Y_t^{x,\xi})
       (\gxY_t)^2\bigr]\rmd t+\ggxZ_t\rmd B_t,
\end{aligned}\\
 \ggxX_0=0.
\end{cases}
\end{equation}

\begin{thm}
\label{thm: third spatial derivative}
Under Assumption~\ref{assum: Hc}, equation
\eqref{eq: spatial second variation} has a unique $L_r^2$ solution and
\begin{equation}
\label{eq: identify spatial second variation}
 \left(\frac{\nabla X^{x+h,\xi}-\nabla X^{x,\xi}}h,
       \frac{\nabla Y^{x+h,\xi}-\nabla Y^{x,\xi}}h,
       \frac{\nabla Z^{x+h,\xi}-\nabla Z^{x,\xi}}h\right)
 \xrightarrow{h\to0}(\ggxX,\ggxY,\ggxZ)
\end{equation}
in $L_r^2\times L_r^2\times L_r^2$.  Moreover,
\begin{equation}
\label{eq: identify spatial second initial}
 \left|\frac{\nabla Y_0^{x+h,\xi}-\nabla Y_0^{x,\xi}}h
 -\ggxY_0\right|\xrightarrow{h\to0}0.
\end{equation}
Consequently,
\begin{equation}
\label{eq: third derivative V representation}
 \partial_{xx}\mcv(x,\mu)=\partial_{xxx}V(x,\mu)=\ggxY_0,
 \qquad \mu=\cL_\xi.
\end{equation}
This derivative is bounded and jointly continuous in $(x,\mu)$.
\end{thm}

\proof
The inhomogeneous terms in \eqref{eq: spatial second variation} are
quadratic in $(\gxX,\gxY)$.  Lemma~\ref{lem: first variation moments}
with $p=4$ and the boundedness of the third derivatives give
\begin{equation}
 \dbE\int_0^\infty\rme^{-rt}
 \bigl(|\gxX_t|+|\gxY_t|\bigr)^4\rmd t\leq C.
\end{equation}
The homogeneous part of \eqref{eq: spatial second variation} satisfies
the strict monotonicity condition, so Lemma~\ref{lem: affine stability}
gives a unique $L_r^2$ solution.

We now verify that this solution is the genuine second spatial
variation.  Put
\begin{equation}
 \Delta X_t^h=X_t^{x+h,\xi}-X_t^{x,\xi},
 \Delta Y_t^h=Y_t^{x+h,\xi}-Y_t^{x,\xi}.
\end{equation}
For $ij\in\{xx,xy,yy\}$ and $k\in\{x,y\}$, define
\begin{equation}
 \begin{aligned}
 \partial_{kij}H_t^h
 \triangleq\int_0^1&\partial_k\partial_{ij}H
 \bigl(X_t^{x,\xi}+\theta\Delta X_t^h,\cL_{X_t^\xi},
       Y_t^{x,\xi}+\theta\Delta Y_t^h\bigr)\rmd\theta,\\
 \mathcal A_{ij,t}^h
 \triangleq&\partial_{xij}H_t^h\,\nabla X_t^{x,\xi,h}
   +\partial_{yij}H_t^h\,\nabla Y_t^{x,\xi,h}.
 \end{aligned}
\end{equation}
Thus the mean-value formula gives, for example,
\begin{equation}
 \frac{\partial_{xy}H(X_t^{x+h,\xi},\cL_{X_t^\xi},Y_t^{x+h,\xi})
 -\partial_{xy}H(X_t^{x,\xi},\cL_{X_t^\xi},Y_t^{x,\xi})}{h}
 =\mathcal A_{xy,t}^h.
\end{equation}
Define
\begin{align}
 \mathcal F_t^h
  &\triangleq\gxX_t\mathcal A_{xy,t}^h+\gxY_t\mathcal A_{yy,t}^h,
  \label{eq: prelimit spatial F}\\
 \mathcal G_t^h
  &\triangleq\gxX_t\mathcal A_{xx,t}^h+\gxY_t\mathcal A_{xy,t}^h.
  \label{eq: prelimit spatial G}
\end{align}
After subtracting \eqref{eq: g pa1} at $x+h$ and at $x$ and dividing
by $h$, one obtains
\begin{equation}
\label{eq: spatial quotient equation}
\begin{cases}
\begin{aligned}
 \rmd\nabla^2X_t^{x,\xi,h}
 ={}&\bigl[
 \partial_{xy}H(X_t^{x+h,\xi},\cL_{X_t^\xi},Y_t^{x+h,\xi})
       \nabla^2X_t^{x,\xi,h}\\
 &+\partial_{yy}H(X_t^{x+h,\xi},\cL_{X_t^\xi},Y_t^{x+h,\xi})
       \nabla^2Y_t^{x,\xi,h}+\mathcal F_t^h\bigr]\rmd t,
\end{aligned}\\
\begin{aligned}
 \rmd\nabla^2Y_t^{x,\xi,h}
 =-{}&\bigl[
 \partial_{xx}H(X_t^{x+h,\xi},\cL_{X_t^\xi},Y_t^{x+h,\xi})
       \nabla^2X_t^{x,\xi,h}\\
 &+\partial_{xy}H(X_t^{x+h,\xi},\cL_{X_t^\xi},Y_t^{x+h,\xi})
       \nabla^2Y_t^{x,\xi,h}
 -r\nabla^2Y_t^{x,\xi,h}+\mathcal G_t^h\bigr]\rmd t\\
 &+\nabla^2Z_t^{x,\xi,h}\rmd B_t,
\end{aligned}\\
 \nabla^2X_0^{x,\xi,h}=0,
\end{cases}
\end{equation}
where
\begin{equation}
 \nabla^2\Phi^{x,\xi,h}
 \triangleq\frac{\nabla\Phi^{x+h,\xi}-\nabla\Phi^{x,\xi}}h,
 \quad \Phi\in\{X,Y,Z\}.
\end{equation}

Define the candidate limiting source processes in
\eqref{eq: spatial quotient equation} by
\begin{align}
 \mathcal F_t={}&
 \partial_{xxy}H(X_t^{x,\xi},\cL_{X_t^\xi},Y_t^{x,\xi})(\gxX_t)^2
 +2\partial_{xyy}H(X_t^{x,\xi},\cL_{X_t^\xi},Y_t^{x,\xi})
       \gxX_t\gxY_t\nonumber\\
 &+\partial_{yyy}H(X_t^{x,\xi},\cL_{X_t^\xi},Y_t^{x,\xi})
       (\gxY_t)^2,\label{eq: limiting spatial F}\\
 \mathcal G_t={}&
 \partial_{xxx}H(X_t^{x,\xi},\cL_{X_t^\xi},Y_t^{x,\xi})(\gxX_t)^2
 +2\partial_{xxy}H(X_t^{x,\xi},\cL_{X_t^\xi},Y_t^{x,\xi})
       \gxX_t\gxY_t\nonumber\\
 &+\partial_{xyy}H(X_t^{x,\xi},\cL_{X_t^\xi},Y_t^{x,\xi})
       (\gxY_t)^2.\label{eq: limiting spatial G}
\end{align}
We claim that
\begin{equation}
 (\mathcal F^h,\mathcal G^h)
 \xrightarrow{h\to0}(\mathcal F,\mathcal G)
 \quad\hbox{in }L_r^2\times L_r^2.
 \label{eq: spatial source convergence}
\end{equation}
To prove this claim, recall explicitly that
\begin{equation}
 \begin{gathered}
  \nabla X^{x,\xi,h}\xrightarrow{h\to0}
  \nabla Y^{x,\xi,h}\xrightarrow{h\to0}\gxY
  \quad\hbox{in }L_r^4,\\
  \sup_{0<|h|\leq1}
  \left(\|\nabla X^{x,\xi,h}\|_{L_r^8}
       +\|\nabla Y^{x,\xi,h}\|_{L_r^8}\right)<\infty.
 \end{gathered}
 \label{eq: recalled spatial quotient bounds}
\end{equation}
For $ij\in\{xx,xy,yy\}$, define
\begin{equation}
 \begin{aligned}
  \mathcal A_{ij,t}\triangleq{}&
  \partial_x\partial_{ij}H
  (X_t^{x,\xi},\cL_{X_t^\xi},Y_t^{x,\xi})\gxX_t\\
  &+\partial_y\partial_{ij}H
  (X_t^{x,\xi},\cL_{X_t^\xi},Y_t^{x,\xi})\gxY_t.
 \end{aligned}
 \label{eq: limiting spatial A}
\end{equation}
We first prove that $\mathcal A_{ij}^h$ converges to
$\mathcal A_{ij}$ in $L_r^4$.  By the Lipschitz continuity of the third
derivatives,
\begin{equation}
\begin{aligned}
 &\left|\partial_{kij}H_t^h-
 \partial_k\partial_{ij}H
 (X_t^{x,\xi},\cL_{X_t^\xi},Y_t^{x,\xi})\right|\\
 &\qquad\leq C|h|
 \bigl(|\nabla X_t^{x,\xi,h}|+
       |\nabla Y_t^{x,\xi,h}|\bigr).
\end{aligned}
\end{equation}
Consequently, for example,
\begin{equation}
\begin{aligned}
 &\left\|
 \left(\partial_{kij}H^h-\partial_k\partial_{ij}H
 (X^{x,\xi},\cL_{X^\xi},Y^{x,\xi})\right)
 \nabla X^{x,\xi,h}\right\|_{L_r^4}\\
 &\qquad\leq C|h|
 \left(\|\nabla X^{x,\xi,h}\|_{L_r^8}
      +\|\nabla Y^{x,\xi,h}\|_{L_r^8}\right)
 \|\nabla X^{x,\xi,h}\|_{L_r^8}
 \xrightarrow{h\to0}0.
\end{aligned}
\end{equation}
Moreover, the difference $\mathcal A_{ij}^h-\mathcal A_{ij}$ is the sum
of the following four terms:
\begin{equation}
\begin{aligned}
 &\left(\partial_{xij}H^h-
 \partial_x\partial_{ij}H(X^{x,\xi},\cL_{X^\xi},Y^{x,\xi})\right)
 \nabla X^{x,\xi,h}\\
 &\quad+
 \partial_x\partial_{ij}H(X^{x,\xi},\cL_{X^\xi},Y^{x,\xi})
 \left(\nabla X^{x,\xi,h}-\gxX\right)\\
 &\quad+
 \left(\partial_{yij}H^h-
 \partial_y\partial_{ij}H(X^{x,\xi},\cL_{X^\xi},Y^{x,\xi})\right)
 \nabla Y^{x,\xi,h}\\
 &\quad+
 \partial_y\partial_{ij}H(X^{x,\xi},\cL_{X^\xi},Y^{x,\xi})
 \left(\nabla Y^{x,\xi,h}-\gxY\right).
\end{aligned}
\label{eq: spatial A difference decomposition}
\end{equation}
The estimate above and the uniform $L_r^8$ bound in
\eqref{eq: recalled spatial quotient bounds} show that the first and
third terms in \eqref{eq: spatial A difference decomposition} converge
to zero in $L_r^4$.  The boundedness of the third derivatives and the
$L_r^4$ convergence in \eqref{eq: recalled spatial quotient bounds} do
the same for the second and fourth terms.  We have therefore proved the
intermediate conclusion
\begin{equation}
 \mathcal A_{ij}^h\xrightarrow{h\to0}\mathcal A_{ij}
  \quad\hbox{in }L_r^4,\quad ij\in\{xx,xy,yy\}.
 \label{eq: spatial A convergence}
\end{equation}

By \eqref{eq: limiting spatial A} and the symmetry of the third
derivatives,
\begin{equation}
 \mathcal F=\gxX\mathcal A_{xy}+\gxY\mathcal A_{yy},
 \qquad
 \mathcal G=\gxX\mathcal A_{xx}+\gxY\mathcal A_{xy}.
\end{equation}
Consequently, H\"older's inequality and
\eqref{eq: spatial A convergence} give
\begin{align}
 \|\mathcal F^h-\mathcal F\|_{L_r^2}
 &\leq \|\gxX\|_{L_r^4}
          \|\mathcal A_{xy}^h-\mathcal A_{xy}\|_{L_r^4}
       +\|\gxY\|_{L_r^4}
          \|\mathcal A_{yy}^h-\mathcal A_{yy}\|_{L_r^4}
 \xrightarrow{h\to0}0,\label{eq: spatial F source convergence}\\
 \|\mathcal G^h-\mathcal G\|_{L_r^2}
 &\leq \|\gxX\|_{L_r^4}
          \|\mathcal A_{xx}^h-\mathcal A_{xx}\|_{L_r^4}
       +\|\gxY\|_{L_r^4}
          \|\mathcal A_{xy}^h-\mathcal A_{xy}\|_{L_r^4}
 \xrightarrow{h\to0}0.\label{eq: spatial G source convergence}
\end{align}
This proves \eqref{eq: spatial source convergence}.  It also makes clear
where fourth moments enter: the source terms are products of two
$L_r^4$ processes and therefore belong to $L_r^2$.

We now apply Lemma~\ref{lem: affine stability} directly to
\eqref{eq: spatial quotient equation} and
\eqref{eq: spatial second variation}.  These are, respectively, the
approximating and limiting FBSDEs in that lemma, and both have forward
initial value zero.  Their direct source terms are
$(\mathcal F^h,\mathcal G^h)$ and $(\mathcal F,\mathcal G)$.  Their
convergence in $L_r^2\times L_r^2$ was proved in
\eqref{eq: spatial source convergence}.

To verify the coefficient part of Lemma~\ref{lem: affine stability}, set
\begin{align}
 \alpha_t^h&\triangleq\partial_{xy}H(X_t^{x+h,\xi},\cL_{X_t^\xi},Y_t^{x+h,\xi}),
 &\beta_t^h&\triangleq\partial_{yy}H(X_t^{x+h,\xi},\cL_{X_t^\xi},Y_t^{x+h,\xi}),\nonumber\\
 \gamma_t^h&\triangleq\partial_{xx}H(X_t^{x+h,\xi},\cL_{X_t^\xi},Y_t^{x+h,\xi}),
 \label{eq: spatial quotient coefficients}
\end{align}
In the notation of Lemma~\ref{lem: affine stability}, the fourth
coefficient is not independent here:
\begin{equation}
 D^h\triangleq \alpha^h,\qquad D\triangleq \alpha,
\end{equation}
because both coefficients are the same derivative $\partial_{xy}H$.
Let $\alpha,\beta,\gamma$ denote the corresponding coefficients evaluated at
$(X^{x,\xi},Y^{x,\xi})$.  Continuity of the second derivatives, the
$L_r^2$ continuity of $(X^{x+h,\xi},Y^{x+h,\xi})$, and their uniform
boundedness imply, since $\rme^{-rt}\rmd t\otimes\dbP$ is finite,
\begin{equation}
 \|\alpha^h-\alpha\|_{L_r^2}+\|\beta^h-\beta\|_{L_r^2}+
 \|\gamma^h-\gamma\|_{L_r^2}\xrightarrow{h\to0}0,
 \label{eq: spatial quotient coefficient convergence}
\end{equation}
The same convergence for $D^h-D$ follows from $D^h=\alpha^h$ and $D=\alpha$.
The coefficient bounds in Assumption~\ref{assum: Hc} are uniform in $h$.
Thus all hypotheses of Lemma~\ref{lem: affine stability} are satisfied,
and its proof gives the following estimate.  The backward initial values here
are deterministic because these are tagged equations with deterministic
initial state $x$ and dependence on $\xi$ only through $\cL_\xi$;
therefore the initial-value term is an ordinary square rather than an
expectation:
\begin{equation}
 \begin{aligned}
 &\left|\nabla^2Y_0^{x,\xi,h}-\ggxY_0\right|^2
 +\left\|\nabla^2X^{x,\xi,h}-\ggxX\right\|_r^2\\
 &\quad+\left\|\nabla^2Y^{x,\xi,h}-\ggxY\right\|_r^2
 +\left\|\nabla^2Z^{x,\xi,h}-\ggxZ\right\|_r^2
 \xrightarrow{h\to0}0.
 \end{aligned}
\end{equation}
This proves \eqref{eq: identify spatial second variation} and
\eqref{eq: identify spatial second initial}.

By \cite{song2025infinite},
$\partial_x\mcv(x,\mu)=\nabla Y_0^{x,\xi}$ and
$\mcv=\partial_xV$.  Hence \eqref{eq: third derivative V representation}
follows.  Finally, let $(x_n,\xi_n)\to(x,\xi)$ in
$\dbR\times\dbL^2$.  The base equations and first variations converge
in $L_r^2$, while Lemma~\ref{lem: first variation moments} and
interpolation improve the convergence of the first variations to
$L_r^4$.  Therefore the two sources in
\eqref{eq: spatial second variation} converge in $L_r^2$.
Lemma~\ref{lem: affine stability} gives convergence of
$(\nabla^2X^{x_n,\xi_n},\nabla^2Y^{x_n,\xi_n})$ and their backward
initial values.  Law invariance then gives boundedness and joint
continuity in $(x,\mu)$. \qed

\subsection{Differentiating the Lions-derivative in its label}

We retain the notation of \eqref{eq: g pa1}--\eqref{eq: g total}.
In particular, for $\mu=\cL_\xi$, Theorem~\ref{thm: Lions representation}
gives
\begin{equation}
\label{eq: Lions representation for label}
 \partial_\mu\mcv(x,\mu,\tilde x)=\gYxx_0.
\end{equation}
Tildes below always denote independent copies on the product space.

\begin{assum}
\label{assum: measure stability}
Let $\lambda_2$ be the common bound for
$\partial_{x\mu}H$ and $\partial_{y\mu}H$ in
Assumption~\ref{assum: Hb}(ii), and let $\overline\pi$ and $\delta_*$
be as in \eqref{eq: linear upper pi} and
\eqref{eq: Hamiltonian stability margin}. Assume that
\begin{equation}
\label{eq: measure derivative smallness}
 \vartheta_\mu
 \triangleq\frac{\lambda_2}{\delta_*}
 \left(1+\frac{b_1(1+\overline\pi)}{r+\delta_*}\right)<1.
\end{equation}
\end{assum}

\begin{lem}
\label{lem: population label moments}
Under Assumptions~\ref{assum: Hc} and~\ref{assum: measure stability}, for $p\geq2$,
\begin{equation}
\label{eq: population label moments}
\begin{aligned}
 &\sup_{t\geq0}\dbE\left[
 |\nabla X_t^{\xi,\tilde x}|^p+
 |\nabla Y_t^{\xi,\tilde x}|^p\right]
 +\|\nabla Z^{\xi,\tilde x}\|_{\mathbb H_r^p}^p
 \leq C_p.
\end{aligned}
\end{equation}
In particular,
\begin{equation}
 \|\nabla X^{\xi,\tilde x}\|_{L_r^p}^p+
 \|\nabla Y^{\xi,\tilde x}\|_{L_r^p}^p+
 \|\nabla Z^{\xi,\tilde x}\|_{\mathbb H_r^p}^p
 \leq C_p,
\end{equation}
uniformly in $(\xi,\tilde x)$.
\end{lem}

\proof
Write
\begin{equation}
\begin{aligned}
 \alpha_t&=\partial_{xy}H(X_t^\xi,\cL_{X_t^\xi},Y_t^\xi),&
 \beta_t&=\partial_{yy}H(X_t^\xi,\cL_{X_t^\xi},Y_t^\xi),\\
 \gamma_t&=\partial_{xx}H(X_t^\xi,\cL_{X_t^\xi},Y_t^\xi).
\end{aligned}
\end{equation}
For this proof only, write
\begin{equation}
 (X,Y,Z)\triangleq
 (\nabla X^{\xi,\tilde x},\nabla Y^{\xi,\tilde x},
  \nabla Z^{\xi,\tilde x}).
\end{equation}
Then \eqref{eq: g pa2} becomes
\begin{equation}
\label{eq: label variation rewritten}
\begin{cases}
\begin{aligned}
 \rmd X_t={}&\left\{\alpha_tX_t+\beta_tY_t
 +\widetilde{\dbE}_{\cF_t}\left[
 \nabla\widetilde X_t^{\tilde x,\xi}
 \partial_{y\mu}H(X_t^\xi,\cL_{X_t^\xi},Y_t^\xi,
                      \widetilde X_t^{\tilde x,\xi})\right.\right.\\
 &\left.\left.\hspace{36mm}
 +\widetilde X_t
 \partial_{y\mu}H(X_t^\xi,\cL_{X_t^\xi},Y_t^\xi,
                      \widetilde X_t^\xi)
 \right]\right\}\rmd t,
\end{aligned}\\
\begin{aligned}
 \rmd Y_t=-{}&\left\{\gamma_tX_t+\alpha_tY_t-rY_t
 +\widetilde{\dbE}_{\cF_t}\left[
 \nabla\widetilde X_t^{\tilde x,\xi}
 \partial_{x\mu}H(X_t^\xi,\cL_{X_t^\xi},Y_t^\xi,
                      \widetilde X_t^{\tilde x,\xi})\right.\right.\\
 &\left.\left.\hspace{36mm}
 +\widetilde X_t
 \partial_{x\mu}H(X_t^\xi,\cL_{X_t^\xi},Y_t^\xi,
                      \widetilde X_t^\xi)
 \right]\right\}\rmd t+Z_t\rmd B_t,
\end{aligned}\\
 X_0=0.
\end{cases}
\end{equation}
Here $\widetilde X$ without superscripts is an independent copy of the
temporary process $X$, whereas $\widetilde X^\xi$ and
$\widetilde X^{\tilde x,\xi}$ are the previously defined state
processes; $\nabla\widetilde X^{\tilde x,\xi}$ is the corresponding
spatial first variation.

Let $(P,Q)$ be the Riccati solution constructed in the proof of
Lemma~\ref{prop: linear Lp}.  Then
\begin{equation}
 \underline\pi\leq P_t\leq\overline\pi,\qquad
 \alpha_t+\beta_tP_t\leq-\delta_*,
\end{equation}
and
\begin{equation}
 r-\alpha_t-\beta_tP_t\geq r+\delta_*.
\label{eq: population label two decay rates}
\end{equation}

For a given adapted process $u$, define the forward and backward source
terms
\begin{align}
 f_t^u\triangleq{}&\widetilde{\dbE}_{\cF_t}\left[
 \nabla\widetilde X_t^{\tilde x,\xi}\,
 \partial_{y\mu}H(X_t^\xi,\cL_{X_t^\xi},Y_t^\xi,
                    \widetilde X_t^{\tilde x,\xi})
 +\widetilde u_t\,
 \partial_{y\mu}H(X_t^\xi,\cL_{X_t^\xi},Y_t^\xi,
                    \widetilde X_t^\xi)\right],\\
 g_t^u\triangleq{}&\widetilde{\dbE}_{\cF_t}\left[
 \nabla\widetilde X_t^{\tilde x,\xi}\,
 \partial_{x\mu}H(X_t^\xi,\cL_{X_t^\xi},Y_t^\xi,
                    \widetilde X_t^{\tilde x,\xi})
 +\widetilde u_t\,
 \partial_{x\mu}H(X_t^\xi,\cL_{X_t^\xi},Y_t^\xi,
                    \widetilde X_t^\xi)\right].
\label{eq: label affine sources}
\end{align}
Conditional expectation is a contraction in $\dbL^p$.  Hence, for any
$u,u'$,
\begin{equation}
 \max\left\{
 \|f_t^u-f_t^{u'}\|_{\dbL^p},
 \|g_t^u-g_t^{u'}\|_{\dbL^p}\right\}
 \leq\lambda_2\|u_t-u_t'\|_{\dbL^p}.
\label{eq: label operator Lp bound}
\end{equation}
Moreover, the pathwise estimate \eqref{eq: linear path estimate}, applied
to $\nabla X^{\tilde x,\xi}$, yields
\begin{equation}
 \max\left\{\|f_t^0\|_{\dbL^p},
             \|g_t^0\|_{\dbL^p}\right\}
 \leq\lambda_2\rme^{-\delta_*t}.
\label{eq: label direct source decay}
\end{equation}

Let $\mathcal E^p$ be the Banach space of adapted processes $U$ such
that
\begin{equation}
 \|U\|_{\mathcal E^p}
 \triangleq\sup_{t\geq0}\|U_t\|_{\dbL^p}.
\end{equation}
Before considering contraction, we first verify that the frozen map is
well defined on this space.  Equations
\eqref{eq: label operator Lp bound} and
\eqref{eq: label direct source decay} give, for every
$u\in\mathcal E^p$,
\begin{equation}
 \max\left\{\|f^u\|_{\mathcal E^p},
             \|g^u\|_{\mathcal E^p}\right\}
 \leq\lambda_2\left(1+\|u\|_{\mathcal E^p}\right).
\label{eq: frozen source E estimate}
\end{equation}
In particular, since $p\geq2$,
\begin{equation}
 \|f^u\|_{L_r^2}^2
 \leq\frac{\|f^u\|_{\mathcal E^p}^2}{r},
 \qquad
 \|g^u\|_{L_r^2}^2
 \leq\frac{\|g^u\|_{\mathcal E^p}^2}{r}.
\end{equation}
Lemma~\ref{lem: affine stability} therefore gives a unique $L_r^2$
solution $(X^u,Y^u,Z^u)$ of the ordinary affine FBSDE
\begin{equation}
\label{eq: label frozen equation}
\begin{cases}
 \rmd X_t^u=(\alpha_tX_t^u+\beta_tY_t^u+f_t^u)\rmd t,\\
 \rmd Y_t^u=-(\gamma_tX_t^u+\alpha_tY_t^u-rY_t^u+g_t^u)\rmd t
                 +Z_t^u\rmd B_t,\\
 X_0^u=0.
\end{cases}
\end{equation}
Put $\Theta^u=Y^u-PX^u$.  Recall that
\begin{equation}
 \rmd P_t=-\left[\gamma_t-(r-2\alpha_t)P_t+\beta_tP_t^2\right]\rmd t
             +Q_t\rmd B_t.
\label{eq: frozen Riccati equation}
\end{equation}
Since $X^u$ has finite variation, $[P,X^u]=0$.  The product rule,
\eqref{eq: label frozen equation}, and
\eqref{eq: frozen Riccati equation} therefore give
\begin{align}
 \rmd(P_tX_t^u)
 ={}&X_t^u\rmd P_t+P_t\rmd X_t^u\nonumber\\
 ={}&\bigl[-\gamma_tX_t^u+(rP_t-\alpha_tP_t-\beta_tP_t^2)X_t^u
          +\beta_tP_tY_t^u+P_tf_t^u\bigr]\rmd t
          +Q_tX_t^u\rmd B_t.
\label{eq: frozen product calculation}
\end{align}
Set $k_t\triangleq r-\alpha_t-\beta_tP_t$.  Subtracting
\eqref{eq: frozen product calculation} from the equation for $Y^u$, and
using $Y^u-PX^u=\Theta^u$, yields the linear BSDE
\begin{equation}
 \rmd\Theta_t^u
 =-\bigl[-k_t\Theta_t^u+P_tf_t^u+g_t^u\bigr]\rmd t
   +(Z_t^u-Q_tX_t^u)\rmd B_t.
\label{eq: frozen theta equation}
\end{equation}
For $0\leq t\leq s$, set
\begin{equation}
 \Gamma_{t,s}\triangleq
 \exp\left(-\int_t^s k_v\rmd v\right).
\end{equation}
The finite-time variation-of-constants formula for this linear BSDE,
justified after the usual localization, gives, for every $T>t$,
\begin{equation}
 \Theta_t^u
 =\dbE_{\cF_t}\left[
   \Gamma_{t,T}\Theta_T^u
   +\int_t^T\Gamma_{t,s}(P_sf_s^u+g_s^u)\rmd s\right].
\label{eq: finite frozen theta representation}
\end{equation}

We next remove the terminal term.  Since
$(X^u,Y^u)\in L_r^2\times L_r^2$ and $P$ is bounded,
$\Theta^u\in L_r^2$.  Hence there is a deterministic sequence
$T_n\to\infty$ such that
\begin{equation}
 \rme^{-rT_n}\dbE|\Theta_{T_n}^u|^2\xrightarrow{n\to\infty}0.
\end{equation}
By \eqref{eq: population label two decay rates},
$\Gamma_{t,T}\leq\rme^{-(r+\delta_*)(T-t)}$, and consequently
\begin{equation}
\begin{aligned}
 \|\Gamma_{t,T_n}\Theta_{T_n}^u\|_{\dbL^2}
 &\leq \rme^{-(r+\delta_*)(T_n-t)}
          \|\Theta_{T_n}^u\|_{\dbL^2}\\
 &=\rme^{(r+\delta_*)t}
   \rme^{-(r/2+\delta_*)T_n}
   \left(\rme^{-rT_n}\dbE|\Theta_{T_n}^u|^2\right)^{1/2}
 \xrightarrow{n\to\infty}0.
\end{aligned}
\end{equation}
We also verify the convergence of the integral explicitly.  Since
$P\leq\overline\pi$, \eqref{eq: frozen source E estimate} implies
\begin{equation}
 \|P_sf_s^u+g_s^u\|_{\dbL^p}
 \leq(1+\overline\pi)\lambda_2
 \left(1+\|u\|_{\mathcal E^p}\right).
\end{equation}
Therefore, for $T\geq t$, Minkowski's inequality and
$\Gamma_{t,s}\leq\rme^{-(r+\delta_*)(s-t)}$ give
\begin{align}
 &\left\|\int_T^\infty
 \Gamma_{t,s}(P_sf_s^u+g_s^u)\rmd s\right\|_{\dbL^p}\nonumber\\
 &\quad\leq(1+\overline\pi)\lambda_2
 \left(1+\|u\|_{\mathcal E^p}\right)
 \int_T^\infty\rme^{-(r+\delta_*)(s-t)}\rmd s\nonumber\\
 &\quad=\frac{(1+\overline\pi)\lambda_2
 \left(1+\|u\|_{\mathcal E^p}\right)}
 {r+\delta_*}
 \rme^{-(r+\delta_*)(T-t)}
 \xrightarrow{T\to\infty}0.
\label{eq: frozen theta integral tail}
\end{align}
Thus the truncated integrals in
\eqref{eq: finite frozen theta representation} converge in $\dbL^p$.
Since $p\geq2$, they also converge in $\dbL^2$.  Letting
$T=T_n\to\infty$ in that equation therefore gives
\begin{equation}
 \Theta_t^u
 =\dbE_{\cF_t}\left[
 \int_t^\infty
 \Gamma_{t,s}
 (P_sf_s^u+g_s^u)\rmd s\right].
\label{eq: label theta representation}
\end{equation}

Substituting $Y^u=\Theta^u+PX^u$ into the forward equation in
\eqref{eq: label frozen equation} gives the linear forward equation
\begin{equation}
 \rmd X_t^u=\bigl[(\alpha_t+\beta_tP_t)X_t^u
                   +\beta_t\Theta_t^u+f_t^u\bigr]\rmd t,
 \qquad X_0^u=0.
\end{equation}
Its variation-of-constants formula is
\begin{equation}
 X_t^u
 =\int_0^t
 \exp\left(\int_s^t(\alpha_v+\beta_vP_v)\rmd v\right)
 (\beta_s\Theta_s^u+f_s^u)\rmd s.
\label{eq: label U representation}
\end{equation}
We now derive the two estimates used in the contraction argument.  First,
conditional expectation is a contraction in $\dbL^p$, and Minkowski's
inequality, $P\leq\overline\pi$, and
\eqref{eq: population label two decay rates} give
\begin{align}
 \|\Theta_t^u\|_{\dbL^p}
 &\leq \int_t^\infty
 \rme^{-(r+\delta_*)(s-t)}
 \left(\overline\pi\|f_s^u\|_{\dbL^p}
                   +\|g_s^u\|_{\dbL^p}\right)\rmd s\nonumber\\
 &\leq\left(\overline\pi\|f^u\|_{\mathcal E^p}
                   +\|g^u\|_{\mathcal E^p}\right)
 \int_t^\infty\rme^{-(r+\delta_*)(s-t)}\rmd s\nonumber\\
 &=\frac{\overline\pi\|f^u\|_{\mathcal E^p}
                   +\|g^u\|_{\mathcal E^p}}
 {r+\delta_*}.
\label{eq: pointwise frozen theta E estimate}
\end{align}
For the forward component, $\alpha+\beta P\leq-\delta_*$ implies
\begin{equation}
 \exp\left(\int_s^t(\alpha_v+\beta_vP_v)\rmd v\right)
 \leq\rme^{-\delta_*(t-s)},
 \qquad 0\leq s\leq t.
\end{equation}
Using this bound in \eqref{eq: label U representation}, together with
$|\beta|\leq b_1$ and Minkowski's inequality, yields
\begin{align}
 \|X_t^u\|_{\dbL^p}
 &\leq\int_0^t\rme^{-\delta_*(t-s)}
 \left(b_1\|\Theta_s^u\|_{\dbL^p}
                   +\|f_s^u\|_{\dbL^p}\right)\rmd s\nonumber\\
 &\leq\left(b_1\|\Theta^u\|_{\mathcal E^p}
                   +\|f^u\|_{\mathcal E^p}\right)
 \int_0^t\rme^{-\delta_*(t-s)}\rmd s\nonumber\\
 &\leq\frac{b_1\|\Theta^u\|_{\mathcal E^p}
                   +\|f^u\|_{\mathcal E^p}}
 {\delta_*},
\label{eq: pointwise frozen X E estimate}
\end{align}
Taking the supremum over $t\geq0$ in
\eqref{eq: pointwise frozen theta E estimate} and
\eqref{eq: pointwise frozen X E estimate} gives
\begin{align}
 \|\Theta^u\|_{\mathcal E^p}
 &\leq\frac{\overline\pi\|f^u\|_{\mathcal E^p}
                 +\|g^u\|_{\mathcal E^p}}
                {r+\delta_*},\label{eq: label theta E estimate}\\
 \|X^u\|_{\mathcal E^p}
 &\leq\frac{b_1\|\Theta^u\|_{\mathcal E^p}
                 +\|f^u\|_{\mathcal E^p}}
                {\delta_*}.
\label{eq: label U E estimate}
\end{align}
The right-hand sides are finite by
\eqref{eq: frozen source E estimate}.  Hence
$X^u,\Theta^u\in\mathcal E^p$, and
$Y^u=\Theta^u+PX^u\in\mathcal E^p$.  In particular,
$u\mapsto X^u$ is a well-defined self-map of
$\mathcal E^p$.

For two inputs $u,u'$, equations
\eqref{eq: label operator Lp bound}--\eqref{eq: label U E estimate}
therefore give
\begin{equation}
 \|X^u-X^{u'}\|_{\mathcal E^p}
 \leq\vartheta_\mu
       \|u-u'\|_{\mathcal E^p}.
\end{equation}
By \eqref{eq: measure derivative smallness},
$\vartheta_\mu<1$.  Hence the map $u\mapsto X^u$ is a contraction on
$\mathcal E^p$.  Its
fixed point is the forward component $X$ of
\eqref{eq: label variation rewritten}, hence the $L_r^2$ solution of
\eqref{eq: g pa2}, by uniqueness.  Taking $u'=0$ in the contraction
estimate and using $X=X^X$ gives
\begin{equation}
 \|X\|_{\mathcal E^p}
 \leq\frac{\|X^0\|_{\mathcal E^p}}
 {1-\vartheta_\mu}\leq C_p.
\end{equation}
Equations \eqref{eq: frozen source E estimate} and
\eqref{eq: label theta E estimate}, together with
$Y=\Theta^X+PX$, then yield
\begin{equation}
 \sup_{t\geq0}
 \dbE\left[|X_t|^p+|Y_t|^p\right]\leq C_p.
\label{eq: label XY uniform moments}
\end{equation}

It remains to estimate $Z$.  Set
\begin{equation}
 \widehat Y_t=\rme^{-rt/p}Y_t,\qquad
 \widehat Z_t=\rme^{-rt/p}Z_t.
\end{equation}
Then
\begin{equation}
 \rmd\widehat Y_t
 =\left\{\frac{p-1}{p}r\widehat Y_t
 -\rme^{-rt/p}(\gamma_tX_t+\alpha_tY_t+g_t^X)\right\}\rmd t
 +\widehat Z_t\rmd B_t.
\end{equation}
The uniform estimate \eqref{eq: label XY uniform moments} and
\eqref{eq: label operator Lp bound}--\eqref{eq: label direct source decay}
give, with a constant independent of $t$,
\begin{equation}
 \|\widehat Y_t\|_{\dbL^p}
 +\left\|\rme^{-rt/p}
 (\gamma_tX_t+\alpha_tY_t+g_t^X)\right\|_{\dbL^p}
 \leq C_p\rme^{-rt/p}.
\end{equation}
Since $r>0$, the right-hand side is integrable on $[0,\infty)$.
Moreover, $\widehat Y_t\to0$ in $\dbL^p$.  The transformed backward
equation therefore shows that the continuous martingale
\begin{equation}
 M_t\triangleq\int_0^t\widehat Z_s\rmd B_s
\end{equation}
converges in $\dbL^p$ as $t\to\infty$.  In particular,
$M_t=\dbE_{\cF_t}[M_\infty]$.  Doob's and BDG inequalities, first on
$[0,T]$, yield
\begin{equation}
 \dbE\left(\int_0^T
 |\widehat Z_t|^2\rmd t\right)^{p/2}
 \leq C_p\dbE\sup_{0\leq t\leq T}|M_t|^p
 \leq C_p\dbE|M_\infty|^p.
\end{equation}
Letting $T\to\infty$ and using the monotone convergence theorem gives
the required $\mathbb H_r^p$ estimate and completes the proof.
\qed

We first introduce the derivative of
$(\nabla X^{\xi,\tilde x},\nabla Y^{\xi,\tilde x},
\nabla Z^{\xi,\tilde x})$ with respect to its label.  For
$q\in\{x,y\}$, define
\begin{equation}
\label{eq: direct label source}
\begin{aligned}
 \rho_{q,t}^{\xi,\tilde x}
 \triangleq\widetilde{\dbE}_{\cF_t}\bigl[&
 \nabla^2\widetilde X_t^{\tilde x,\xi}
 \partial_{q\mu}H(X_t^\xi,\cL_{X_t^\xi},Y_t^\xi,
                   \widetilde X_t^{\tilde x,\xi})\\
 &+(\nabla\widetilde X_t^{\tilde x,\xi})^2
 \partial_{\tilde x}\partial_{q\mu}H
 (X_t^\xi,\cL_{X_t^\xi},Y_t^\xi,
  \widetilde X_t^{\tilde x,\xi})\bigr].
\end{aligned}
\end{equation}
Consider
\begin{equation}
\label{eq: second population label variation}
\begin{cases}
\begin{aligned}
 \rmd\nabla^2X_t^{\xi,\tilde x}={}&\biggl[
 \partial_{xy}H(X_t^\xi,\cL_{X_t^\xi},Y_t^\xi)
       \nabla^2X_t^{\xi,\tilde x}
 +\partial_{yy}H(X_t^\xi,\cL_{X_t^\xi},Y_t^\xi)
       \nabla^2Y_t^{\xi,\tilde x}\\
 &+\widetilde{\dbE}_{\cF_t}\bigl[
 \nabla^2\widetilde X_t^{\xi,\tilde x}
 \partial_{y\mu}H(X_t^\xi,\cL_{X_t^\xi},Y_t^\xi,
                   \widetilde X_t^\xi)\bigr]
 +\rho_{y,t}^{\xi,\tilde x}\biggr]\rmd t,
\end{aligned}\\
\begin{aligned}
 \rmd\nabla^2Y_t^{\xi,\tilde x}=-{}&\biggl[
 \partial_{xx}H(X_t^\xi,\cL_{X_t^\xi},Y_t^\xi)
       \nabla^2X_t^{\xi,\tilde x}
 +\partial_{xy}H(X_t^\xi,\cL_{X_t^\xi},Y_t^\xi)
       \nabla^2Y_t^{\xi,\tilde x}
 -r\nabla^2Y_t^{\xi,\tilde x}\\
 &+\widetilde{\dbE}_{\cF_t}\bigl[
 \nabla^2\widetilde X_t^{\xi,\tilde x}
 \partial_{x\mu}H(X_t^\xi,\cL_{X_t^\xi},Y_t^\xi,
                   \widetilde X_t^\xi)\bigr]
 +\rho_{x,t}^{\xi,\tilde x}\biggr]\rmd t
 +\nabla^2Z_t^{\xi,\tilde x}\rmd B_t,
\end{aligned}\\
 \nabla^2X_0^{\xi,\tilde x}=0.
\end{cases}
\end{equation}

\begin{thm}
\label{thm: second population label}
Under Assumptions~\ref{assum: Hc} and~\ref{assum: measure stability}, equation
\eqref{eq: second population label variation} has a unique $L_r^2$
solution.  For every $\Phi\in\{X,Y,Z\}$,
\begin{equation}
 \frac{\nabla\Phi^{\xi,\tilde x+h}-\nabla\Phi^{\xi,\tilde x}}h
 \xrightarrow{h\to0}\nabla^2\Phi^{\xi,\tilde x}
 \quad\hbox{in }L_r^2.
\label{eq: identify second population label}
\end{equation}
In particular, the backward initial values satisfy
\begin{equation}
 \dbE\left|
 \frac{\nabla Y_0^{\xi,\tilde x+h}
       -\nabla Y_0^{\xi,\tilde x}}h
 -\nabla^2Y_0^{\xi,\tilde x}\right|^2
 \xrightarrow{h\to0}0.
\label{eq: identify second population label initial}
\end{equation}
\end{thm}

\proof
Lemma~\ref{lem: population label moments} first gives all the required
moments of
$(\nabla X^{\xi,\tilde x},\nabla Y^{\xi,\tilde x},
\nabla Z^{\xi,\tilde x})$.  By Theorem~\ref{thm: third spatial derivative},
$\nabla^2X^{\tilde x,\xi}\in L_r^2$, and
Lemma~\ref{lem: first variation moments} gives
$(\nabla X^{\tilde x,\xi})^2\in L_r^2$.  Hence
\eqref{eq: direct label source} belongs to $L_r^2$, uniformly in
$(\xi,\tilde x)$.  The homogeneous mean-field system in
\eqref{eq: second population label variation} is the same as in
\eqref{eq: g pa2}; Assumption~\ref{assum: Hb}(ii) gives the strict
monotonicity margin.  Lemma~\ref{lem: affine stability} therefore gives
a unique solution of \eqref{eq: second population label variation}.

We next prove \eqref{eq: identify second population label}.  For
$\Phi\in\{X,Y,Z\}$ set
\begin{equation}
 \nabla^2\Phi^{\xi,\tilde x,h}
 \triangleq\frac{\nabla\Phi^{\xi,\tilde x+h}
              -\nabla\Phi^{\xi,\tilde x}}h.
\end{equation}
Subtracting \eqref{eq: g pa2} at $\tilde x+h$ and $\tilde x$ and
dividing by $h$, the only source term requiring attention is, for
$q\in\{x,y\}$,
\begin{equation}
\label{eq: label prelimit source}
\begin{aligned}
 \mathcal D_{q,t}^h
 \triangleq\frac1h\bigl[&
 \nabla\widetilde X_t^{\tilde x+h,\xi}
 \partial_{q\mu}H(X_t^\xi,\cL_{X_t^\xi},Y_t^\xi,
                   \widetilde X_t^{\tilde x+h,\xi})\\
 &-\nabla\widetilde X_t^{\tilde x,\xi}
 \partial_{q\mu}H(X_t^\xi,\cL_{X_t^\xi},Y_t^\xi,
                   \widetilde X_t^{\tilde x,\xi})\bigr].
\end{aligned}
\end{equation}
The mean-value formula rewrites it as
\begin{equation}
\label{eq: expanded label prelimit source}
\begin{aligned}
 \mathcal D_{q,t}^h={}&
 \frac{\nabla\widetilde X_t^{\tilde x+h,\xi}
       -\nabla\widetilde X_t^{\tilde x,\xi}}h
 \partial_{q\mu}H(X_t^\xi,\cL_{X_t^\xi},Y_t^\xi,
                   \widetilde X_t^{\tilde x+h,\xi})\\
 &+\nabla\widetilde X_t^{\tilde x,\xi}
 \int_0^1\partial_{\tilde x}\partial_{q\mu}H
 \bigl(X_t^\xi,\cL_{X_t^\xi},Y_t^\xi,
 \widetilde X_t^{\tilde x,\xi}
 +\theta(\widetilde X_t^{\tilde x+h,\xi}
         -\widetilde X_t^{\tilde x,\xi})\bigr)\rmd\theta\,
 \nabla\widetilde X_t^{\tilde x,\xi,h}.
\end{aligned}
\end{equation}
Theorem~\ref{thm: third spatial derivative} gives convergence of the
first quotient in \eqref{eq: expanded label prelimit source} to
$\nabla^2\widetilde X^{\tilde x,\xi}$ in $L_r^2$.
Lemma~\ref{lem: first variation moments} gives
\begin{equation}
 \nabla\widetilde X^{\tilde x,\xi}\in L_r^4,\qquad
 \nabla\widetilde X^{\tilde x,\xi,h}
 \xrightarrow{h\to0}\nabla\widetilde X^{\tilde x,\xi}
 \quad\hbox{in }L_r^4.
\end{equation}
Together with the bounded Lipschitz derivative
$\partial_{\tilde x}\partial_{q\mu}H$, these facts imply
\begin{equation}
\begin{aligned}
 \mathcal D_{q}^h\xrightarrow{h\to0}{}
 &\nabla^2\widetilde X^{\tilde x,\xi}
 \partial_{q\mu}H(X^\xi,\cL_{X^\xi},Y^\xi,
                   \widetilde X^{\tilde x,\xi})\\
 &+(\nabla\widetilde X^{\tilde x,\xi})^2
 \partial_{\tilde x}\partial_{q\mu}H
 (X^\xi,\cL_{X^\xi},Y^\xi,\widetilde X^{\tilde x,\xi})
\end{aligned}
\end{equation}
in $L_r^2$ on the product space.

The quotient FBSDE contains, in addition,
\begin{equation}
 \widetilde{\dbE}_{\cF_t}\left[
 \nabla^2\widetilde X_t^{\xi,\tilde x,h}
 \partial_{q\mu}H(X_t^\xi,\cL_{X_t^\xi},Y_t^\xi,
                   \widetilde X_t^\xi)\right].
\end{equation}
To display the error equation explicitly, set
\begin{equation}
\begin{aligned}
 \widehat X^{\xi,\tilde x,h}
 &\triangleq\nabla^2X^{\xi,\tilde x,h}-\nabla^2X^{\xi,\tilde x},\\
 \widehat Y^{\xi,\tilde x,h}
 &\triangleq\nabla^2Y^{\xi,\tilde x,h}-\nabla^2Y^{\xi,\tilde x},\\
 \widehat Z^{\xi,\tilde x,h}
 &\triangleq\nabla^2Z^{\xi,\tilde x,h}-\nabla^2Z^{\xi,\tilde x},
\end{aligned}
\end{equation}
and let
\begin{equation}
\begin{aligned}
 \mathcal D_{q,t}\triangleq{}&
 \nabla^2\widetilde X_t^{\tilde x,\xi}
 \partial_{q\mu}H(X_t^\xi,\cL_{X_t^\xi},Y_t^\xi,
                   \widetilde X_t^{\tilde x,\xi})\\
 &+(\nabla\widetilde X_t^{\tilde x,\xi})^2
 \partial_{\tilde x}\partial_{q\mu}H
 (X_t^\xi,\cL_{X_t^\xi},Y_t^\xi,
  \widetilde X_t^{\tilde x,\xi}),\\
 \varepsilon_{q,t}^h
 \triangleq{}&\widetilde{\dbE}_{\cF_t}
       [\mathcal D_{q,t}^h-\mathcal D_{q,t}],
 \qquad q\in\{x,y\}.
\end{aligned}
\end{equation}
Direct subtraction of \eqref{eq: second population label variation}
from the quotient equation gives
\begin{equation}
\label{eq: first label error equation}
\begin{cases}
\begin{aligned}
 \rmd\widehat X_t^{\xi,\tilde x,h}
 ={}&\biggl[
 \partial_{xy}H(X_t^\xi,\cL_{X_t^\xi},Y_t^\xi)
       \widehat X_t^{\xi,\tilde x,h}\\
 &+\partial_{yy}H(X_t^\xi,\cL_{X_t^\xi},Y_t^\xi)
       \widehat Y_t^{\xi,\tilde x,h}\\
 &+\widetilde{\dbE}_{\cF_t}\bigl[
 \widehat{\widetilde X}_t^{\xi,\tilde x,h}
 \partial_{y\mu}H(X_t^\xi,\cL_{X_t^\xi},Y_t^\xi,
                   \widetilde X_t^\xi)\bigr]
 +\varepsilon_{y,t}^h\biggr]\rmd t,
\end{aligned}\\
\begin{aligned}
 \rmd\widehat Y_t^{\xi,\tilde x,h}
 =-{}&\biggl[
 \partial_{xx}H(X_t^\xi,\cL_{X_t^\xi},Y_t^\xi)
       \widehat X_t^{\xi,\tilde x,h}\\
 &+\partial_{xy}H(X_t^\xi,\cL_{X_t^\xi},Y_t^\xi)
       \widehat Y_t^{\xi,\tilde x,h}
 -r\widehat Y_t^{\xi,\tilde x,h}\\
 &+\widetilde{\dbE}_{\cF_t}\bigl[
 \widehat{\widetilde X}_t^{\xi,\tilde x,h}
 \partial_{x\mu}H(X_t^\xi,\cL_{X_t^\xi},Y_t^\xi,
                   \widetilde X_t^\xi)\bigr]
 +\varepsilon_{x,t}^h\biggr]\rmd t
 +\widehat Z_t^{\xi,\tilde x,h}\rmd B_t,
\end{aligned}\\
 \widehat X_0^{\xi,\tilde x,h}=0.
\end{cases}
\end{equation}
Conditional expectation is a contraction in $L^2$, so the convergence
of $\mathcal D_q^h$ gives
$\|\varepsilon_q^h\|_r\xrightarrow{h\to0}0$ for $q=x,y$.
Applying the estimate in the proof of Lemma~\ref{lem: affine stability}
to \eqref{eq: first label error equation} yields
\begin{equation}
\begin{aligned}
 &\left\|\nabla^2X^{\xi,\tilde x,h}
       -\nabla^2X^{\xi,\tilde x}\right\|_r^2
 +\left\|\nabla^2Y^{\xi,\tilde x,h}
       -\nabla^2Y^{\xi,\tilde x}\right\|_r^2\\
 &\quad+\left\|\nabla^2Z^{\xi,\tilde x,h}
       -\nabla^2Z^{\xi,\tilde x}\right\|_r^2
 +\dbE\left|\nabla^2Y_0^{\xi,\tilde x,h}
       -\nabla^2Y_0^{\xi,\tilde x}\right|^2
 \xrightarrow{h\to0}0.
\end{aligned}
\end{equation}
This proves \eqref{eq: identify second population label}, including its
initial-time assertion. \qed

We now return to the Lions variation \eqref{eq: g total}.  For
$q\in\{x,y\}$, define
\begin{equation}
\label{eq: Lions label source}
\begin{aligned}
 \sigma_{q,t}^{x,\xi,\tilde x}
 \triangleq\widetilde{\dbE}_{\cF_t}\bigl[&
 \nabla^2\widetilde X_t^{\tilde x,\xi}
 \partial_{q\mu}H(X_t^{x,\xi},\cL_{X_t^\xi},Y_t^{x,\xi},
                   \widetilde X_t^{\tilde x,\xi})\\
 &+(\nabla\widetilde X_t^{\tilde x,\xi})^2
 \partial_{\tilde x}\partial_{q\mu}H
 (X_t^{x,\xi},\cL_{X_t^\xi},Y_t^{x,\xi},
  \widetilde X_t^{\tilde x,\xi})\\
 &+\nabla^2\widetilde X_t^{\xi,\tilde x}
 \partial_{q\mu}H(X_t^{x,\xi},\cL_{X_t^\xi},Y_t^{x,\xi},
                   \widetilde X_t^\xi)\bigr].
\end{aligned}
\end{equation}
Consider the FBSDE
\begin{equation}
\label{eq: derivative Lions label equation}
\left\{
\begin{aligned}
 \rmd\ggXxx_t={}&\bigl[
 \partial_{xy}H(X_t^{x,\xi},\cL_{X_t^\xi},Y_t^{x,\xi})
       \ggXxx_t\\
 &+\partial_{yy}H(X_t^{x,\xi},\cL_{X_t^\xi},Y_t^{x,\xi})
       \ggYxx_t
 +\sigma_{y,t}^{x,\xi,\tilde x}\bigr]\rmd t,
\\
 \rmd\ggYxx_t=-{}&\bigl[
 \partial_{xx}H(X_t^{x,\xi},\cL_{X_t^\xi},Y_t^{x,\xi})
       \ggXxx_t\\
 &+\partial_{xy}H(X_t^{x,\xi},\cL_{X_t^\xi},Y_t^{x,\xi})
       \ggYxx_t-r\ggYxx_t
 +\sigma_{x,t}^{x,\xi,\tilde x}\bigr]\rmd t
 +\ggZxx_t\rmd B_t,
\\
 \ggXxx_0={}&0.
\end{aligned}
\right.
\end{equation}

\begin{thm}
\label{thm: derivative Lions label}
Under Assumptions~\ref{assum: Hc} and~\ref{assum: measure stability}, equation
\eqref{eq: derivative Lions label equation} has a unique $L_r^2$
solution.  For every $\Phi\in\{X,Y,Z\}$,
\begin{equation}
 \frac{\nabla_\mu\Phi^{x,\xi,\tilde x+h}
       -\nabla_\mu\Phi^{x,\xi,\tilde x}}h
 \xrightarrow{h\to0}\nabla_{\tilde x\mu}
       \Phi^{x,\xi,\tilde x}
 \quad\hbox{in }L_r^2.
\label{eq: identify derivative Lions label}
\end{equation}
More importantly, at time zero,
\begin{equation}
 \left|
 \frac{\nabla_\mu Y_0^{x,\xi,\tilde x+h}
       -\nabla_\mu Y_0^{x,\xi,\tilde x}}h
 -\nabla_{\tilde x\mu}Y_0^{x,\xi,\tilde x}\right|
 \xrightarrow{h\to0}0.
\label{eq: identify derivative Lions label initial}
\end{equation}
Consequently,
\begin{equation}
\label{eq: label derivative representation}
 \partial_{\tilde x}\partial_\mu\mcv(x,\mu,\tilde x)
 =\partial_{\tilde x}\partial_\mu\partial_xV(x,\mu,\tilde x)
 =\ggYxx_0,\qquad \mu=\cL_\xi.
\end{equation}
The right-hand side is bounded and jointly continuous in
$(x,\mu,\tilde x)$ and is deterministic.
\end{thm}

\proof
Theorem~\ref{thm: second population label}, Theorem
\ref{thm: third spatial derivative}, and the $L_r^4$ estimate for
$\nabla X^{\tilde x,\xi}$ show that every term in
\eqref{eq: Lions label source} belongs to $L_r^2$.  Lemma
\ref{lem: affine stability} therefore gives a unique $L_r^2$ solution
of \eqref{eq: derivative Lions label equation}.

For $\Phi\in\{X,Y,Z\}$, set
\begin{equation}
 \nabla_{\tilde x\mu}\Phi^{x,\xi,\tilde x,h}
 \triangleq\frac{\nabla_\mu\Phi^{x,\xi,\tilde x+h}
              -\nabla_\mu\Phi^{x,\xi,\tilde x}}h.
\end{equation}
Subtracting \eqref{eq: g total} at $\tilde x+h$ and $\tilde x$ and
dividing by $h$ produces the same homogeneous linear FBSDE as
\eqref{eq: derivative Lions label equation}.  It remains only to check
its two source terms.  For the first term inside each conditional
expectation, the mean-value calculation
\eqref{eq: expanded label prelimit source}, with
$(X^\xi,Y^\xi)$ replaced by $(X^{x,\xi},Y^{x,\xi})$, gives convergence
in $L_r^2$ to the first two terms of
\eqref{eq: Lions label source}.  The quotient of the second term is
\begin{equation}
 \nabla^2\widetilde X_t^{\xi,\tilde x,h}
 \partial_{q\mu}H(X_t^{x,\xi},\cL_{X_t^\xi},Y_t^{x,\xi},
                   \widetilde X_t^\xi),
 \qquad q\in\{x,y\},
\end{equation}
and converges in $L_r^2$ to the third term of
\eqref{eq: Lions label source} by Theorem
\ref{thm: second population label}.  Since conditional expectation is
a contraction in $L^2$, the complete forward and backward source terms
converge in $L_r^2$.  The estimate in the proof of
Lemma~\ref{lem: affine stability} now proves
\eqref{eq: identify derivative Lions label} and
\eqref{eq: identify derivative Lions label initial}; for the latter,
the general $\dbL^2$ conclusion of that lemma reduces to ordinary
absolute-value convergence because both backward initial values are
deterministic.  Combining this with
\eqref{eq: Lions representation for label} gives
\eqref{eq: label derivative representation}.

To prove continuity, let
\begin{equation}
 (x_n,\xi_n,\tilde x_n)\xrightarrow{n\to\infty}
 (x,\xi,\tilde x)
 \quad\hbox{in }\dbR\times\dbL^2\times\dbR.
\end{equation}
The base equations and first variations converge in $L_r^2$.  Their
uniform $L_r^8$ estimates and interpolation give the $L_r^4$
convergence needed for the squared terms.  Theorem
\ref{thm: third spatial derivative} gives convergence of
$\nabla^2X^{\tilde x_n,\xi_n}$ in $L_r^2$.  Consequently the sources
\eqref{eq: direct label source} converge in $L_r^2$, and affine
stability gives convergence of
$\nabla^2X^{\xi_n,\tilde x_n}$.  The three terms in
\eqref{eq: Lions label source} then converge in $L_r^2$; a final
application of affine stability gives convergence of
$\nabla_{\tilde x\mu}Y_0^{x_n,\xi_n,\tilde x_n}$ in $\dbR$.
The same estimates give uniform boundedness.  Since the initial value
depends on $\xi$ only through $\cL_\xi$, the representation is bounded
and jointly continuous in $(x,\mu,\tilde x)$. \qed

\subsection{Classical regularity conclusion}

\begin{thm}
\label{thm: classical derivative master}
Under Assumptions~\ref{assum: Hc} and~\ref{assum: measure stability}, the function
\begin{equation}
 \mcv(x,\mu)=Y_0^{x,\xi},\qquad \cL_\xi=\mu,
\end{equation}
has the bounded continuous derivatives
\begin{equation}
\label{eq: all derivative representations}
\begin{aligned}
 \partial_x\mcv(x,\mu)&=\nabla Y_0^{x,\xi},&
 \partial_{xx}\mcv(x,\mu)&=\nabla^2Y_0^{x,\xi},\\
 \partial_\mu\mcv(x,\mu,\tilde x)
 &=\nabla_\mu Y_0^{x,\xi,\tilde x},&
 \partial_{\tilde x}\partial_\mu\mcv(x,\mu,\tilde x)
 &=\nabla_{\tilde x\mu}Y_0^{x,\xi,\tilde x}.
\end{aligned}
\end{equation}
In particular, \(\mcv\in C^{2,1}(\dbR\times\cP_2)\).
Consequently, the viscosity solution $\mcv=\partial_xV$ constructed in
\cite{yang2025discounted} is a classical solution of the elliptic
master equation for the value function's derivative:
\begin{equation}
\label{eq: classical derivative master}
\begin{aligned}
 r\mcu(x,\mu)={}&\partial_xH(x,\mu,\mcu(x,\mu))
 +\partial_yH(x,\mu,\mcu(x,\mu))\cdot\partial_x\mcu(x,\mu)
 +\frac12\partial_{xx}\mcu(x,\mu)\\
 &+\tilde{\dbE}\left[
 \frac12\partial_{\tilde x}\partial_\mu\mcu(x,\mu,\tilde\xi)
 +\partial_\mu\mcu(x,\mu,\tilde\xi)
 \partial_yH(\tilde\xi,\mu,\mcu(\tilde\xi,\mu))\right].
\end{aligned}
\end{equation}
\end{thm}

\proof
The derivatives in
\eqref{eq: all derivative representations} give
\(\mcv\in C^{2,1}(\dbR\times\cP_2)\) and the required bounds.  The
standard consistency of the viscosity definition with smooth functions
then turns the subsolution and supersolution inequalities into the two
opposite pointwise inequalities.  Therefore $\mcv$ is classical. \qed

\section*{Acknowledgements}
Song Y. is financially supported by National Key R\&D Program of China (No. 2024YFA1013503 \& No. 2020YFA0712700) and the National Natural Science Foundation of China (No. 12431017).

\end{document}